\documentclass[a4paper,12pt,thmsa]{article}

\usepackage{xcolor}
\usepackage{amsmath}
\usepackage{amsthm}
\usepackage{amsfonts}
\usepackage{amssymb}
\usepackage{enumerate}
\usepackage{dsfont}
\usepackage{graphicx}
\usepackage{xr}
\usepackage{bezier}
\usepackage{tikz}
\usepackage{epstopdf}
\usepackage{geometry}
\usepackage{stmaryrd}
\usepackage{bbm}
\usepackage[normalem]{ulem}
\usepackage[colorlinks=true,linkcolor=blue,urlcolor=blue,citecolor=red, hyperfigures=false]{hyperref}

\makeatletter

\renewcommand{\paragraph}{\@startsection{paragraph}{4}{0mm}{-3mm}{-3mm} {\noindent \bf}}

\renewcommand{\subparagraph}{\@startsection{paragraph}{4}{3mm}{-3mm}{-3mm} {\noindent \bf}}

\newcommand{\subsubparagraph}{\@startsection{paragraph}{4}{3mm}{-3mm}{-3mm} {\noindent \it}}

\let\inf\relax \DeclareMathOperator*\inf{\vphantom{p}inf}
\let\max\relax \DeclareMathOperator*\max{\vphantom{p}max}
\let\min\relax \DeclareMathOperator*\min{\vphantom{p}min}

\numberwithin{equation}{section}
\numberwithin{figure}{section}
\newtheorem{theorem}{Theorem}[section]
\theoremstyle{plain}

\newtheorem{corollary}[theorem]{Corollary}

\newtheorem{definition}[theorem]{Definition}

\newtheorem{lemma}[theorem]{Lemma}

\newtheorem{proposition}[theorem]{Proposition}
\newtheorem{remark}[theorem]{Remark}

\DeclareFontFamily{U}  {MnSymbolF}{}
\DeclareSymbolFont{symbolsMN}{U}{MnSymbolF}{m}{n}
\SetSymbolFont{symbolsMN}{bold}{U}{MnSymbolF}{b}{n}
\DeclareFontShape{U}{MnSymbolF}{m}{n}{
    <-6>  MnSymbolF5
   <6-7>  MnSymbolF6
   <7-8>  MnSymbolF7
   <8-9>  MnSymbolF8
   <9-10> MnSymbolF9
  <10-12> MnSymbolF10
  <12->   MnSymbolF12}{}
\DeclareFontShape{U}{MnSymbolF}{b}{n}{
    <-6>  MnSymbolF-Bold5
   <6-7>  MnSymbolF-Bold6
   <7-8>  MnSymbolF-Bold7
   <8-9>  MnSymbolF-Bold8
   <9-10> MnSymbolF-Bold9
  <10-12> MnSymbolF-Bold10
  <12->   MnSymbolF-Bold12}{}
\DeclareMathSymbol{\tbigtimes}{\mathop}{symbolsMN}{2}
\newcommand*{\bigtimes}{%
  \DOTSB
  \tbigtimes
  \slimits@
}

\newcommand{\be}{\begin{equation}}
\newcommand{\ee}{\end{equation}}

\newcommand{\RR}{{\mathbb{R}}}

\newcommand{\EE}{{\mathbb{E}}}

\newcommand{\PP}{{\mathbb{P}}}

\newcommand{\QQ}{{\mathbb{Q}}}

\newcommand{\CC}{{\mathcal{C}}}

\newcommand{\CF}{{\mathcal{F}}}

\newcommand{\CM}{{\mathcal{M}}}

\newcommand{\CI}{{\mathcal{I}}}

\newcommand{\CT}{{\mathcal{T}}}
\newcommand{\CB}{{\mathcal{B}}}

\newcommand{\indic}{{\mathbbm{1}}}
\newcommand{\comment}[1]{}

\begin{document}

\hyphenchar\font=-1

\baselineskip=7mm

\vspace{3cm}

\author{Jean-Paul D\'ecamps\thanks{Toulouse School of Economics, University of Toulouse Capitole, Toulouse, France. E-mail: \tt{jean-paul.}  \tt{jean-paul.decamps@tse-fr.eu}.}~~~~~~~Fabien Gensbittel\thanks{Toulouse School of Economics, University of Toulouse Capitole, Toulouse, France. E-mail: \tt{fabien.} \tt{fabien.gensbittel@tse-fr.eu}.}~~~~~~~Thomas Mariotti\thanks{Toulouse School of Economics, CNRS, University of Toulouse Capitole, Toulouse, France, CEPR, and CESifo. Email: \tt{thomas.mariotti@tse-fr.eu}.}} \vspace{1cm}

\title{\textbf{Mixed Markov-Perfect Equilibria in the Continuous-Time War of Attrition}\footnote{We thank two anonymous referees for very thoughtful and detailed comments. We also thank Erik Ekstr\"{o}m for very valuable feedback. This research has benefited from financial support of the ANR (Programme d'Investissements d'Avenir CHESS ANR-17-EURE-0010) and the research foundation TSE-Partnership (Chaire March\'es des Risques et Cr\'eation de Valeur, Fondation du Risque/SCOR).}}

\vspace{1cm}
\maketitle
\vspace*{6mm}

\begin{abstract}
We prove the existence of a Markov-perfect equilibrium in randomized stopping times for a model of the war of attrition in which the underlying state variable follows a homogenous linear diffusion. We first prove that the space of Markovian randomized stopping times can be topologized as a compact absolute retract. This in turn enables us to use a powerful fixed-point theorem by Eilenberg and Montgomery \cite{EilenbergMontgomery} to prove our existence theorem. We illustrate our results with an example of a war of attrition that admits a mixed-strategy Markov-perfect equilibrium but no pure-strategy Markov-perfect equilibrium.

\bigskip


\end{abstract}

\thispagestyle{empty}
\newpage
\setcounter{page}{1}

\section{Introduction}

{In this paper, we consider two-player nonzero-sum continuous-time stopping games in which the payoff for every player $i=1,2$ is defined by
\begin{align*}
J^i(x,\tau^i,\tau^j) : =\EE_x\! \left[\indic_{\tau^i \leq \tau^j} \,\mathrm e^{-r\tau^i}R^i(X_{\tau^i})+ \indic_{\tau^j<\tau^i} \,\mathrm e^{-r \tau^j}G^i(X_{\tau^j})\right]\hskip -1mm,
\end{align*}
where $j=3-i$ and
\begin{itemize}

\item[(i)]

$X:=(X_t)_{t \geq 0}$ is a continuous homogenous linear diffusion defined on an open interval ${\cal I} \subset \mathbb R$ and satisfying $X_0=x$;

\item[(ii)]

$\tau^i$ and $\tau^j$ are stopping times of the filtration of $X$ chosen by players $i$ and $j$, respectively;

\item [(iii)]

$r\geq 0$ is a constant discount rate;

\item[(iv)]

the reward functions $R^i$ and $G^i$ are continuous functions such that $R^i \leq G^i$ on $\mathcal I$ and whose discounted values at $X_t$ satisfy limit and integrability conditions when $t \to \infty$.

\end{itemize}
The assumption $R^i \leq G^i$ reflects a second-mover advantage for every player $i = 1,2$ and is typical in the stopping games referred to as \textit{wars of attrition} in the biology and economics literature (see \cite{DGM2, GeorgiadisKimKwon, Gieczewski, KwonPalczewski, Lambrecht, Murto} for examples of such games under Brownian uncertainty). In line with this terminology, a game satisfying assumptions (i)--(iv) is referred to in what follows as a \textit{linear Brownian war of attrition} (lBwa). The main goal of this paper is to show that any lBwa admits a (pure- or mixed-strategy) Markov-perfect equilibrium (Mpe), to be defined below. To this end, we bring to bear new topological methods for the analysis of continuous-time stopping games.

A pair of stopping times $(\tau^1 ,\tau ^2)$ is a pure-strategy Nash equilibrium for the lBwa with initial condition $x \in \CI$ if $J^i (x,\tau^i, \tau^j)= \sup_{\tau} J^i(x,\tau,\tau^j)$ for every player $i=1,2$. A pure-strategy Markov-perfect equilibrium (Mpe) is a pair of Markovian stopping times (characterized as hitting times of a closed subset of ${\cal I}$) that induces a pure-strategy Nash equilibrium for every initial condition. Markov-perfectness requires that the continuation equilibrium following any history only depend on payoff-relevant variables (Maskin and Tirole \cite{MaskinTirole}) and  rules out noncredible threats (Selten \cite{Selten}, Riedel and Steg \cite{RiedelSteg}).

The existence of pure-strategy Mpes for lBwas can be established under a variety of additional assumptions on the reward functions $R^i$ and $G^i$.
\begin{enumerate}

\item

It follows from very general results of Cattiaux and Lepeltier \cite{CattiauxLepeltier} and Lepeltier and Etourneau \cite{LepeltierEtourneau} that a pure-strategy Mpe exists under the additional assumption that the processes $(\mathrm e^{-rt}G^i(X_t))_{t \geq 0}$, $i=1,2$, are supermartingales.

\item

De Angelis, Ferrari, and Moriarty \cite{DeAFM18} prove the existence of a pure-strategy Mpe in trigger strategies under geometric conditions on the functions $R^i$, $i=1,2$. Similar results have been derived by Attard \cite{Attard17}  and Martyr and Moriarty \cite{Martyr}.

\item

Ekstr\"om and Villeneuve \cite{EkstromVilleneuve} and Ekstr\"om and Peskir \cite{EkstromPeskir} prove the existence of a pure-strategy Mpe in the zero-sum case where $R^i = - G^j$, $i=1,2$.

\end{enumerate}
In the absence of such additional assumptions, a pure-strategy Mpe may not exist---we provide a simple explicit counterexample in Section \ref{sec:example}. Notice in that respect that, according to very general results by Lepeltier and Maingueneau \cite{LepeltierMaingueneau} (in the zero-sum case) and Hamadene and Zhang \cite{HamadeneZhang} (in the nonzero-sum case), a lBwa always admits a pure-strategy Nash equilibrium. However, the strategies constructed in, e.g., \cite{HamadeneZhang} are not Markovian, and the resulting equilibrium is typically not subgame-perfect in the sense of Riedel and Steg \cite{RiedelSteg}, so that some player has to make a noncredible threat off the equilibrium path.

To recover the existence of an Mpe, we follow the classical approach in game theory: we extend the class of strategies by allowing players to use mixed strategies. In the present context, this amounts to considering randomized stopping times---whereby, loosely speaking, players choose a distribution on the set of stopping times.

Randomized stopping times have been considered for a long time in the theory of optimal stopping (see, e.g., Baxter and Chacon \cite{BaxterChacon}, Meyer \cite{Meyer}, El Karoui, Lepeltier, and Millet \cite{EKLM}) and in the analysis of stopping games (see, e.g., Touzi and Vieille \cite{TouziVieille}, Riedel and Steg \cite{RiedelSteg}, Laraki and Solan \cite{LarakiSolanZS}, Laraki and Solan \cite{LarakiSolanNZS}, De Angelis, Merkulov, and Palczewski \cite{DeAMP22}). The difference between these and the present paper is one of focus: we are ultimately interested in the existence of Mpes of lBwas. Therefore, whereas the \textit{games} we consider form a more modest class, our \textit{equilibrium concept} is more demanding. Correspondingly, instead of allowing for general randomized stopping times in possibly very general stochastic environments, our analysis focuses on Markovian randomized stopping times for continuous homogenous linear diffusions.

For this specific class of environments, D\'ecamps, Gensbittel, and Mariotti \cite{DGM2} derive from a representation result for multiplicative functionals due to Sharpe \cite{Sharpe} that any Markovian randomized stopping time can be represented by a pair $(\mu,S)$, where $S$ is a closed subset of $\CI$ and $\mu$ is a locally finite measure on $\CI\setminus S$ such that the conditional survival function $\Lambda_t$---that is, the probability to stop strictly after $t$ conditionally on $(X_s)_{s \in [0,t]}$---writes as $\Lambda_t=\indic_{t <\tau_S} \, \mathrm e^{-\int_ {\CI \setminus S} L^y_t \,\mu(\mathrm dy)} $, where $L^y_t$ is the local time of $X$ at $(y,t)$ and $\tau_S$ is the hitting time by $X$ of $S$. The set $S$ is the region of immediate stopping and $\mu$ is a (possibly singular) intensity of stopping outside of $S$. Equivalently, the pair $(\mu,S)$ can be seen as a nonnegative measure on $\CI$ that explodes on $S$. This allows one to identify the set of Markovian randomized stopping times with the set $\mathcal M(\mathcal I)$ of nonnegative (but not necessary locally finite) regular Borel measures.

Building on this representation theorem, \cite{DGM2} offers a detailed analysis of mixed-strategy Mpes of the lBwa under assumptions on reward functions that are special cases of those introduced in \cite{CattiauxLepeltier} and \cite{DeAFM18}. In this context, it can be shown that, if players are asymmetric, then all mixed-strategy Mpes of the lBwa involve strategies with discrete intensity measures of stopping whose supports consist of two intertwined sequences of randomization thresholds, generating singular conditional survival functions. These randomization thresholds and the players' equilibrium value functions are semi-explicitly characterized as the solutions to a variational system.

An open question is whether the existence of an Mpe for the lBwa can be established with no assumptions on reward functions besides assumptions (i)--(iv) above and, correspondingly, without the luxury of variational methods. The present paper answers this question in the positive by relying instead on topological methods.

To this end, the standard approach requires (A) endowing $\mathcal M(\mathcal I)$ with a convenient topological structure, and (B) checking that the best-reply correspondence is sufficiently well-behaved to apply an appropriate fixed-point theorem. However, an immediate difficulty in the implementation of this approach is that $\mathcal M(\mathcal I)$ lacks a natural convex structure, reflecting the possibility that a player may stop with infinite intensity on some subset $S$ of $\mathcal I$. This prevents us from using Kakutani's \cite{Kakutani} fixed-point theorem or its extension by Glicksberg \cite{Glicksberg} to compact convex subsets of a locally convex Hausdorff topological vector space. We show that this lack of convexity can be overcome by invoking a more powerful result, namely, a special case of the Eilenberg--Montgomery fixed-point theorem \cite[Theorem 1]{EilenbergMontgomery} that applies to (non necessarily convex) compact absolute retracts. (The topological notions involved in the statement of this theorem are recalled in Section \ref{unsurvol}.)

To carry out the two-step program (A)--(B), we establish a number of new topological properties of the set of Markovian randomized stopping times and of the best response correspondences that may be of independent interest.

(A) Our first main result, Theorem \ref{thm:compact_AR}, defines a topology $\vartheta$ on $\mathcal M(\mathcal I)$ that extends vague convergence of locally finite measures, and establishes that $(\mathcal M(\mathcal I),\vartheta)$ is a compact absolute retract. The proof consists of two steps that can be sketched as follows. (1) The first step is relatively standard and consists in showing that $(\CM (\CI),\vartheta)$ is metrizable and compact (Proposition \ref{prop:compact_metric}). As a by-product, we obtain that the topology induced by $\vartheta$ on the subset $\CM_{\rm{loc}}(\CI)\subset \mathcal M( \mathcal I)$ of locally finite measures coincides with the usual topology $\upsilon$ of vague convergence. (2) The second step is more challenging, and involves more advanced tools. By Step 1, one can identify $\mathcal M_{\rm{loc}}(\mathcal I)$ with a metrizable convex subset of the locally convex Hausdorff topological vector space of linear functionals on $\CC_c^+(\CI)$ endowed with the vague topology; this, in turn, implies that $(\CM_{\rm{loc}}(\CI), \upsilon)$ is an absolute neighborhood retract. The bulk of the proof then consists in proving that one can continuously deform any measure $m \in \CM(\CI)$ into the null measure $0 \in \CM_{\rm{loc}}(\CI)$ in such a way that, except perhaps for the initial point $m$, the whole path of this deformation is included in $\mathcal M_{\rm{loc}}(\mathcal I)$. This shows at once that $(\CM (\CI),\vartheta)$ is contractible and that it includes $(\mathcal M_{\rm{loc}}(\mathcal I),\upsilon)$ as an homotopy-dense absolute neighborhood retract (Proposition \ref{prop:contraction}). As $(\CM_{\rm{loc}}(\CI), \upsilon)$ is an absolute neighborhood retract, this in turn implies that $(\CM (\CI),\vartheta)$ is a contractible absolute neighborhood retract, i.e., an absolute retract, establishing Theorem \ref{thm:compact_AR}.

(B) To prove our second main result, Theorem \ref{thm:main}, which establishes that any lBwa admits an Mpe, we first show that, if player $j$ stops according to a Markovian randomised stopping time $m^j \in \CM(\CI)$, then player $i$ has a pure-strategy best reply that consists in a Markovian stopping time (Proposition \ref{prop:existence_pbr}). This best reply is perfect in that it remains optimal regardless of the initial condition. We further characterize the set of perfect best replies in $\CM(\CI)$ to player $j$'s Markovian randomized stopping time $m^j$ (Proposition \ref{prop:cacrt_pbr}), and, among those, the set of pure-strategy best replies, whose associated optimal stopping regions form a lattice (Lemma \ref{lem:sigma_i}). This part of our analysis falls in the domain of the general theory of optimal stopping. We next introduce a correspondence $\Phi$ whose values are subsets of the perfect best-reply correspondence and that is more amenable to topological methods than the latter. Our key results about $\Phi$ is that it has a closed graph (Proposition \ref{prop:closed_graph}) and contractible values (Proposition \ref{prop:contractible_values}), from which the existence of an Mpe follows by a direct application of the Eilenberg--Montgomery theorem, establishing Theorem \ref{thm:main}. The intermediary results proven along the way have independent interest for a class of optimal stopping problems that can be studied in its own right, in which the decision maker must choose an optimal stopping time anticipating that his reward function may jump upward according to a Markovian randomized stopping time. Special cases of such models have been considered in the real- option literature on investment under technological and cash-flow uncertainty (see, e.g., \cite{ChronopoulosSiddiqui, HuismanKort, Murto}), in which good news, in the form of technological breakthroughs, arise with constant intensity. Our analysis of perfect best replies generalizes the insights of that literature to a state-dependent intensity of technological breakthroughs. It notably entails that the lattice of optimal investment regions varies upper hemicontinuously (in the Painlev\'e--Kuratowski sense) with respect to the measure $m \in \CM(\CI)$ associated with a Markovian randomized breakthrough time (Corollary \ref{breakthrough}).\footnote{It should be noted that not every optimal stopping model with good news fits into our framework. Such is for instance the case of the incomplete-information model of D\'ecamps, Gensbittel, and Mariotti \cite{DGM1}, in which good news arise when $X$ exceeds a threshold value whose value is ex-ante unknown to the decision maker. The corresponding randomized stopping time is not Markovian, reflecting that the decision maker learns about the unknown threshold value over time. As a result, the appropriate state variable for optimal stopping is $X$ together with its running maximum.}

We hope that our general approach can be extended to cover a wider range of games and environments. This issue is discussed at the end of the paper, where a generalization of Theorem \ref{thm:main} to more than two players is presented in Theorem \ref{pg}.

In contemporaneous independent work, Christensen and Schultz \cite{ChristensenShultz} derive an analogous existence theorem for Mpes of lBwas using a different method. They first consider a family of auxiliary games in which the players are only allowed to stop over increasingly finer finite subsets of the state space $\CI$. In these discretized games, the best-reply sets are convex and the existence of an Mpe can be directly proved using Kakutani's fixed-point theorem \cite{Kakutani}. The existence of an Mpe for the primary game is then obtained as the limit of a convergent sequence of Mpes of these auxiliary games. The analysis requires the introduction of two distinct topologies. The first one allows one to use Kakutani's fixed-point theorem in the auxiliary discretized games. The second, based on the distribution of stopped processes, defines an appropriate notion of convergence allowing one to pass from a sequence of Mpes of the discretized games to an Mpe of the primary game.

By using the Eilenberg-Montgomery fixed-point theorem, our approach is more direct, avoids convexity issues, and only requires us to define a natural topology under which the set of Markovian randomized stopping times, identified to $\CM(\CI)$, is a compact absolute retract. Interestingly, when the discretization is locally finite, the topology used in the auxiliary games introduced by \cite{ChristensenShultz} actually corresponds to the restriction to a finite or countable subspace of the state space of the topology we define on $\CM(\CI)$.

To the best of our knowledge, \cite{ChristensenShultz} and the present paper are the only papers proving the existence of an Mpe in the lBwa under weak assumptions on the reward functions. Our approach and that in \cite{ChristensenShultz} are complementary in that the method of proof in the latter paper enables one to show that some Mpe of the continuous-time game can be obtained as the limit of a sequence of Mpes of suitably discretized games. Whether this is the case of all Mpes remains an open question.}

\section{Model and Main Results}\label{sec:model}

\subsection{A Brownian Model of the War of Attrition}

Consider a one-dimensional time-homogeneous diffusion process $X := (X_t)_{t \geq 0}$ defined on the canonical space $(\Omega,\CF,\mathbb P_x)$ of continuous trajectories with $X_0=x$ under $\mathbb P_x$, that is solution in law to the SDE
\begin{align} \label{eq1}
\mathrm dX_t = b(X_t) \, \mathrm dt  \, + \, \sigma(X_t) \, \mathrm dW_t, \quad t \geq 0,
\end{align}
driven by some Brownian motion $W := (W_t)_{t \geq 0}$. The state space for $X$ is an interval $\mathcal I := (\alpha, \beta)$, with $- \infty \leq \alpha < \beta \leq \infty$, and $b$ and $\sigma$ are continuous functions, with $\sigma>0$ on $\mathcal I$. We assume that $\alpha$ and $\beta$ are natural endpoints for the diffusion. Therefore, $X$ is regular on $\mathcal I$ and the SDE \eqref{eq1} admits a weak solution that is unique in law.

The process $X$ is defined on the canonical space $(\Omega,\mathcal F)$ of continuous trajectories endowed with the usual family of shift operators $(\theta_t)_{t \geq 0}$. Let $\mathbb P_\mu$ be the law of the process $X$ with initial distribution $\mu \in \Delta(\mathcal I)$, where $\Delta(\mathcal I)$ is the space of probability measures on the Borel $\sigma$-field $\mathcal B(\mathcal I)$. We denote by $(\CF_t^0)_{t \geq 0}$ the natural filtration $(\sigma( X_s; s \leq t))_{t \geq 0}$ generated by $X$, and we let $\CF_\infty^0 := \sigma(\bigcup_{t \geq 0} \mathcal F_t^0)$. For each $\mu$, we denote by $\CF_\infty^\mu$ the completion of $\CF _\infty^0$ with respect to $\mathbb P_\mu$, and, for each $t\geq 0$, we let $\CF_t^\mu$ be the augmentation of $\CF_t^0$ by the $\mathbb P_\mu$-null, $\CF_\infty ^\mu$-measurable sets. The usual augmented filtration $(\CF_t)_{t \geq 0}$ is then defined by $\CF_t := \bigcap_{\mu\in \Delta(\mathcal I)} \CF_t^\mu$ for all $t\geq 0$ and it is right-continuous (see, e.g., \cite[Chapter III, \S2, Proposition 2.10]{RevuzYor}). As usual, we say that a property of the trajectories $\omega \in \Omega$ is satisfied a.s.\! if, for each $x\in \CI$, it is satisfied for $\mathbb P_x$-a.e.\! $\omega \in \Omega$.

The game is played as follows. Player 1 chooses a stopping  time $\tau ^1$ and player 2 chooses a stopping time $\tau^ 2$ in the set $\CT$ of all stopping times of $(\CF_t)_{t \geq 0}$. Both players discount future payoffs at a constant rate $r \geq 0$. For each $i =1,2 $, the expected payoff of player $i$ is\footnote{By convention, we let $\mathrm e^{-r \tau}f(X_{ \tau}) := 0$ on $\{\tau = \infty\}$ for any Borel function $f$ and any random time $\tau$, see Assumption A2 below.}
\begin{align}
\label{core} J^i(x,\tau^i,\tau^j) : = \EE _x\! \left[\indic_{\tau^i \leq \tau^j} \,\mathrm e^{-r \tau^i}R^i (X_{\tau^i}) + \indic_{\tau^i>\tau^j}\, \mathrm e^{-r \tau^j} G^i(X_{\tau^j}) \right]\hskip -1mm .
\end{align}
For each $i=1,2,$ we assume
\begin{itemize}
\item[\bf A0] The functions $R^i$ and $G^i$ in (\ref{core}) are continuous on ${\cal I}$ and $R^i \leq G^i$.
\end{itemize}
For each $i = 1, 2$ and every function $f=R^i,G^i$, we assume
\begin{itemize}
\item[\bf A1]
$\mathbb E _x \hskip 0.3mm  [ \sup_{t \geq 0} \mathrm e^{-rt} | f(X_t) | ] < \infty$ for all $x \in \mathcal I$.
\item[\bf A2]
$\lim_{t \to \infty} \mathrm e^{-rt} f(X_t) = 0$ a.s.
\end{itemize}
Assumption A1 guarantees that the family $(\mathrm e^{-r  \tau} f(X_{ \tau}))_{\tau \in {\cal T}}$  is uniformly integrable, that is, the process $(\mathrm e^{-r t} f(X_{t}))_{t \geq 0}$ is of class (D).

A game satisfying the above assumptions is hereafter generically referred to as a lBwa.

\subsection{Randomized Stopping Times}

In this section, we briefly recall some definitions and results that are standard in the literature; we refer to \cite{DGM2} for the missing proofs. For every player $i= 1,2$, consider the enlarged probability space $(\Omega ^i ,\mathcal F^i) := (\Omega\times [0,1], \mathcal F \otimes \mathcal B([0,1]))$, endowed with the product probability $\mathbb P^i_x := \mathbb P_x \otimes \lambda$, where $\lambda$ denotes Lebesgue measure.

\begin{definition}\label{def:randomized_stopping_time}
For each $i=1,2,$ a randomized stopping time for player $i$ is an $\CF \otimes \CB([0,1])$- measurable function $\gamma^i: \Omega^i \to \mathbb R_+$ such that$,$ for $\lambda$-a.e.\! $u^i\in [0,1],$ $\gamma^i(\cdot,u^i) \in \CT$. The process $\Gamma^i := (\Gamma^i_t)_{t \geq 0}$ defined by
\begin{align} \label{ccdf'}
\Gamma^i_t (\omega) := \int_{[0,1]} \indic_{\gamma^i(\omega,u^i) \leq t} \, \mathrm du^i, \quad (\omega,t) \in \Omega \times \mathbb R_+,
\end{align}
is the conditional cumulative distribution function (ccdf) of the randomized stopping time $\gamma^i$. Likewise$,$ the process $\Lambda^i := (\Lambda^i_t)_{t \geq 0}$ defined by
\begin{align*}
\Lambda^i_t(\omega) := 1- \Gamma^i_t(\omega) , \quad (\omega,t) \in \Omega \times \mathbb R_+,
\end{align*}
is the conditional survival function (csf) of the randomized stopping time $\gamma^i$.
\end{definition}

We denote by $\CT_r$ the set of randomized stopping times. The process $\Gamma^i$ defined by (\ref{ccdf'}) takes values in $[0,1]$ and has nondecreasing and right-continuous trajectories.

\begin{lemma}[{\cite[Lemma 2]{DGM2}}]\label{lemma_properties_lambda}
The ccdf process $\Gamma^i$ is $(\CF_t)_{t\geq 0}$-adapted and$,$ for $\mathbb P_x$-a.e.\! $\omega \in \Omega,$
\begin{align*}
\Gamma^i_t (\omega)= \mathbb P^i_x \hskip 0.3mm [\gamma^i \leq t \! \mid \! \CF_t] (\omega)
\end{align*}
for all $x\in \CI$ and $t \geq 0$.
\end{lemma}

By convention, we let $\Gamma^i_{0-}:= 0$ and thus $\Lambda^i_{0-}:=1$. This allows us in what follows to interpret integrals of the form $\int_{[0, \tau)} \cdot \, \mathrm d \Gamma^i_t$ or $\int_{[0, \tau)} \cdot \, \mathrm d \Lambda^i_t$ in the Stieltjes sense for any ccdf $\Gamma^i$ and any csf $\Lambda^i$. Notice for further reference that, for any sufficiently integrable process $Z$,
\begin{align*}
\int_{[0,\tau)} Z_s \, \mathrm d\Gamma^i_s = \Gamma^i_0 Z_0 +\int_{(0,\tau)}Z_s \, \mathrm d\Gamma^i_s .
\end{align*}
If the players use randomized stopping times $\gamma^1$ and $\gamma^2$, then their expected payoffs are defined on the product probability space $\Omega\times [0,1] \times [0,1]$ with canonical element $(\omega,u^1,u^2)$, endowed with the product probability $\overline{\mathbb P}_x := \mathbb P_x \otimes \lambda \otimes \lambda$. Specifically, we have
\begin{align*}
J^i(x,\gamma^1,\gamma^2) := \overline {\mathbb E }_x\! \left[\indic_{\gamma^i\leq \gamma^j} \,\mathrm e^{-r \gamma^i}R^i (X_{\gamma^i}) + \indic_{\gamma^i>\gamma^j}\, \mathrm e^{-r \gamma^j} G^i(X_{\gamma^j}) \right]\hskip -1mm,
\end{align*}
where $\gamma^1 :=\gamma^1(\omega,u^1)$ and $\gamma^2 := \gamma^2(\omega,u^2)$, reflecting that player $1$ and player $2$ use the independent randomization devices $u^1$ and $u^2$, respectively. The next lemma shows that we may equivalently work with the family of ccdf processes $\Gamma^i$.

\begin{lemma}[{\cite[Lemma 3]{DGM2}}] \label{ccdf}
If the players use randomized stopping times with ccdfs $\Gamma^1$ and $\Gamma^2,$ then their expected payoffs write as
\begin{align}
J^i(x,\Gamma^1,\Gamma^2) =  \mathbb E _x \! \left[ \int_{[0,\infty)} \mathrm e^{-r t} R^i(X_t)\Lambda^j_{t-} \, \mathrm d\Gamma^i_t+ \int_{[0,\infty)} \mathrm e^{-r t}G^i(X_t)\Lambda^i_t \, \mathrm d\Gamma^j_t\right] \hskip -1mm. \label{forbpr}
\end{align}
Moreover$,$ any nondecreasing$,$ right-continuous$,$ $(\CF_t)_{t \geq 0}$-adapted$,$ $[0,1]$-valued process $\Gamma^i$ is the ccdf of the randomized stopping time $\hat\gamma^i$ defined by
\begin{align*}
\hat\gamma^i(u^i) := \inf \hskip 0.5mm \{ t \geq 0 : \Gamma^i_t > u^i \}.
\end{align*}
\end{lemma}

\subsection{Markovian Randomized Stopping Times} \label{MRST}

We now recall the definition of a Markovian randomized stopping time used in \cite{DGM2}.

\begin{definition} \label{marksur}
A randomized stopping time for player $i=1,2$ with csf $\Lambda^i: \Omega \times \mathbb R_+ \to [0,1]$ is \textit{Markovian} if$,$ for all $x \in \mathcal I,$ $\tau \in \mathcal T,$ and $s\geq 0,$
\begin{align}\label{mark}
\mbox{$\Lambda^i_{\tau+s}=\Lambda^i_{\tau}(\Lambda^i_s \circ \theta_{\tau})$ on $\{\tau<\infty\}$ $\mathbb P_x$-a.s.}
\end{align}
\end{definition}

Processes satisfying (\ref{mark}) are known as multiplicative functionals of the Markov process $X$, see, e.g., \cite{BlumenthalGetoor}. Combining a result in \cite{Sharpe} with the classical representation result for additive functionals of regular diffusions \cite[Part I, Chapter II, Section 4, \S23]{BorodinSalminen} yields the following representation result.

\begin{theorem}[{\cite[Theorem 1]{DGM2}}] \label{representation}
For each $i =1,2,$ $\Lambda^i : \Omega \times \mathbb R_ +\to [0,1]$ is the csf of a Markovian randomized stopping time for player $i$ if and only if there exists a closed set $S^i \subset \mathcal I$ and a Radon measure\footnote{\label{perigord}
{A Radon measure on a locally compact Hausdorff space $E$ is a nonnegative Borel measure that is locally finite (in the sense that every point of $E$ has a neighborhood having finite measure), inner regular with respect to compact sets, and outer regular with respect to open sets (see, e.g., \cite[Chapter 7, \S 1]{Folland}).}} $\mu^i$ on $\mathcal I \setminus S^i$ such that$,$ for all $x \in \mathcal I$ and $t \geq 0,$
\begin{align*}
\mbox{$\displaystyle \Lambda_t^i= \indic_{t < \tau_{S^i}} \, \mathrm e^{- \int_{\mathcal I \setminus S^i} L_t^y \, \mu^i(\mathrm dy)}$ $\mathbb P_x$-a.s.$,$}
\end{align*}
where $L_t^y := \lim_{\varepsilon \downarrow 0}\, {1 \over 2 \varepsilon} \int_0^t \indic_{(y- \varepsilon, y + \varepsilon)}(X_s) \sigma^2( X_s) \, \mathrm ds$ is the local time of $X$ at $(y,t),$ and $\tau_{S^i}:= \inf \hskip 0.5mm \{t \geq 0: X_t \in S^i\}$ is the hitting time by $X$ of $S^i$. In particular, the mapping $t \mapsto \Lambda_t^i$ is continuous on $[0,\tau_{S^i})$ $\mathbb P_x$-a.s.
\end{theorem}

In the following, we refer to a Markov strategy as a pair $(\mu^i, S^i)$, a ccdf $\Gamma^i$, or a csf $\Lambda ^i$, based on the relations established in Theorem \ref{representation}. Three special cases discussed in \cite{DGM2} are worth mentioning.
\begin{enumerate}

\item

The pure stopping case: If $\mu^i = 0$, then the Markov strategy $(0,S^i)$ is just the hitting time $\tau_{S^i}$ by $X$ of $S^i$.

\item

The absolutely continuous case: If $\mu^i = g^i \cdot \lambda$, then, from the occupation time formula \cite[Chapter VI, \S1, Corollary 1.6]{RevuzYor},
\begin{align*}
\Lambda_t ^i=\indic_{ t<\tau_{S^i}} \, \mathrm e^{- \int_{\mathcal I\setminus S^i} L_t^y g^i(y) \, \mathrm dy} = \indic_{t < \tau_{S^i}} \, \mathrm e^{- \int_0^t g^i(X_s)\sigma^2(X_s)\, \mathrm ds}.
\end{align*}
This absolutely continuous strategy therefore amounts for player $i$ to conceding with intensity $\lambda^i(X_t) := g^i(X_t)\sigma^2(X_t)$ outside $S^i$.

\item

The singular case: If, e.g., $\mu^i = a^i \delta_{x^i}$, where $a^i >0$ and $\delta_{x^i}$ is the Dirac mass at $x^i \in \mathcal I\setminus S^i$, then
\begin{align*}
\Lambda^i_t= \indic_{t<\tau_{S^i}} \, \mathrm e ^{-a^i L^{x^i}_t}.
\end{align*}
Such discrete singular strategies are the building blocks of all mixed-strategy Mpes in the model studied in \cite[Theorems 2--3]{DGM2} when players are asymmetric.

\end{enumerate}

\subsection{Markov-Perfect Equilibrium and Properties of Best Replies}

We recall the definition and some properties of best replies.

\begin{lemma}[{\cite[Lemma 4]{DGM2}}] \label{usepure}
For each $x \in {\mathcal I}$ and for any pair of randomized stopping times with ccdfs $(\Gamma^1,\Gamma^2),$ $J^i (x,  \Gamma^i, \Gamma^j)  \leq \sup_{\tau^i \in \mathcal T} \,J^i(x, \tau^i,\Gamma^j)$.
\end{lemma}

\begin{definition}\label{def:mpe}
For each $i=1,2,$ $(\mu^i,S^i)$ is a perfect best reply (pbr) for player $i$ to $(\mu^j,S^j)$ and $\bar J^i(\cdot, (\mu^j, S^j))$ is player $i$'s best-reply value function (brvf) to $(\mu^j, S^j)$ if
\begin{align*}
\forall x\in \CI, \; J^i(x, (\mu^i, S^i), (\mu^j, S^j)) &= \bar J^i(x, (\mu^j, S^j)) := \sup_{\tau^i \in {\mathcal T}} \, J^i(x, \tau^i, (\mu^j, S^j)).
\end{align*}
The set of pbrs of player $i$ against $(\mu^j,S^j)$ is
\begin{align*}
PBR^i(\mu^j,S^j) := \{ (\mu^i,S^i) : \forall x\in \CI,\, J^i(x, (\mu^i, S^i), (\mu^j, S^j)) = \bar J^i(x, (\mu^j, S^j))\} ,
\end{align*}
and the pbr correspondence is defined by
\begin{align}\label{eq:def_pbr_corr}
 PBR((\mu^1,S^1),(\mu^2,S^2)):=PBR^1(\mu^2,S^2) \times PBR^2(\mu^1,S^1) .
\end{align}
A Markov-perfect equilibrium (MPE) is a profile $((\mu^1 ,S^1),(\mu^2,S^2))$ of Markov strategies such that$,$ for each $i=1,2,$ $(\mu^i,S^i)$ is a pbr for player $i$ to $(\mu^j,S^j)$.
\end{definition}

When no confusion can arise as to the strategy of player $j$, we write $\bar J^i$ instead of $\bar J^i (\cdot, (\mu^j, S^j))$. It follows from Definition \ref{def:mpe} that a pair of Markovian randomized stopping times is an Mpe if and only if it is a fixed point of the pbr correspondence, i.e.,
\begin{align*}
((\mu^1,S^1),(\mu^2,S^2))\in PBR\hskip 0.3mm((\mu^1,S^1),(\mu^2,S^2))).
\end{align*}
The following proposition provides useful general properties of pbrs and brvfs.

\begin{proposition} \label{geneprop}
Given $(\mu^j,S^j),$ the corresponding brvf $\bar J^i$ satisfies
\begin{itemize}

\item[(a)]

$R^i\leq \bar J^i$ on $\CI;$

\item[(b)]

$\bar J^i=G^i$ on $S^j;$

\item[(c)]

for each $x\in S^j,$ if $G^i(x) > R^i(x),$ then $\bar J^i >R^i$ on a neighborhood of $x$.

\end{itemize}
Furthermore$,$ if $(\mu^i,S^i)$ is a pbr to $(\mu^j,S^j),$ then
\begin{itemize}

\item[(i)]

$S^i\cap S^j\cap \{G^i >R^i\}=\emptyset;$

\item[(ii)]

$S^i \subset \overline S\,\!^i:= \{\bar {J}^i=R^i\};$

\item[(iii)]

$\mathrm {supp}\,\mu^i  \setminus S^j \subset \overline S\,\!^i$ and $\mathrm {supp}\,\mu^i \cap  S^j \subset \{\bar {J}^i=G^i\};$

\item[(iv)]

$(0,S^i)$ is also a pbr to $(\mu^j,S^j);$ more generally$,$ $(\tilde \mu^i,S^i)$ is a pbr to $(\mu^j,S^j)$  for any  Radon measure $\tilde\mu^i$ on $\CI \setminus S^i$  such that $\mathrm {supp} \, \tilde\mu^i\subset \overline \!\,\overline S\,\!^i \cup S^j.$

\end{itemize}
\end{proposition}

Except for point (c), the proof of Proposition \ref{geneprop} essentially follows along the lines of \cite[Proposition 1]{DGM2}, and is therefore postponed to the Appendix. Notice that points (i)--(iv) assume that a pbr to $(\mu^j,S^j)$ exists.

\subsection{Main Results}

We are now ready to state our two main results, Theorems \ref{thm:compact_AR} and \ref{thm:main}.

Let $\CM(\CI)$ be the set of nonnegative, regular, but non necessarily finite measures $m : \CB(\CI) \to [0,\infty]$. Our first main result, which may be of independent interest, introduces a convenient topological structure on $\CM(\CI)$.

\begin{theorem}\label{thm:compact_AR}
Let $\vartheta$  be the coarsest topology on $\CM(\CI)$ such that
\begin{enumerate}

\item

for all $a,b \in \CI \cap \QQ$ such that $a<b,$ the mapping $\CM(\CI) \to [0,\infty] : m  \mapsto m((a,b))$ is lower semicontinuous (lsc)$;$

\item

for all $a,b \in \CI \cap \QQ$ such that $a\leq b,$  the mapping $\CM(\CI) \to [0,\infty] : m  \mapsto m([a,b])$ is upper semicontinuous (usc).

\end{enumerate}
Then $(\CM(\CI),\vartheta)$ is a compact absolute retract.
\end{theorem}

Our second main result encapsulates our central existence claim.

\begin{theorem}\label{thm:main}
Any lBwa admits an MPE.
\end{theorem}

Notice that the Mpe whose existence is asserted in Theorem \ref{thm:main} may well have to involve randomized stopping times. Indeed, we provide in Section \ref{sec:example} an example of a lBwa that admits no pure-strategy MPE.

\subsection{An Overview of the Argument} \label{unsurvol}

The proof of Theorem \ref{thm:main} is based on a fixed-point theorem for correspondences applied to a slightly modified version of the pbr correspondence \eqref{eq:def_pbr_corr}. In this section, we outline the main steps of the proof, emphasizing the central role played by Theorem \ref{thm:compact_AR}.

\paragraph{An Alternative Representation of Markov Strategies}

First, it is useful to identify a pair $(\mu,S)$, where $S \subset \CI$ is a closed set and $\mu$ is a Radon measure on $\CI\setminus S$, with a measure in $\mathcal M(\mathcal I)$ that is identically $\infty$ on $S$. Precisely, given such a pair $(\mu,S)$, let $m : \CB(\CI) \to [0, \infty]$ be the measure defined by
\begin{align}\label{eq:muS_to_m}
m(A):= \begin{cases} \mu(A) & \text{if } A\cap S = \emptyset \\ \infty & \text{if } A\cap S \neq \emptyset \end{cases}, \quad  A \in \CB(\CI).
\end{align}
That $m$ is regular and, hence, belongs to $\mathcal M(\mathcal I)$, follows directly from \eqref{eq:muS_to_m} and from $\mu$ being a Radon measure on $\mathcal I \setminus S$. Conversely, given $m\in \CM(\CI)$, let $e(m)$ be the explosion set of $m$, defined as
\begin{align*}
e(m):=\{ x \in \CI : \forall \varepsilon>0, \, m(N_\varepsilon(x))=\infty \},
\end{align*}
where $N_\varepsilon(x):=(x-\varepsilon,x+\varepsilon)\cap \CI$.

\begin{lemma}\label{lem:mtomu}
For each $m \in \CM(\CI),$ the set $e(m)$ is closed and $m|_{\CI \setminus e(m)}$ is a Radon measure on $\CI\setminus e(m)$. Moreover$,$ if $A\in \CB(\CI)$ is such that $A\cap e(m) \neq \emptyset,$ then $m(A)=\infty$.
\end{lemma}

\begin{proof}
First, if $x_n \rightarrow x$ with $x_n \in e(m)$ for all $n\geq 0$, then, for each $\varepsilon>0$, $|x_n-x|<{\varepsilon \over 2}$ for any sufficiently large $n$, so that
\begin{align*}
m(N_{\varepsilon}(x)) \geq m(N_{{\varepsilon \over 2}}(x_n)) = \infty,
\end{align*}
proving that $e(m)$ is closed. Next, by definition, every point $x \in \CI \setminus e(m)$ has a neighborhood with finite $m$-measure, which implies the second assertion. Finally, the last assertion is a direct consequence of the regularity of $m$ and of the definition of $e(m)$. The result follows.
\end{proof}

Using Lemma \ref{lem:mtomu}, we can define a mapping
\begin{align*}
m \mapsto (\mu,S):= (m|_{\CI \setminus e(m)},e(m))
\end{align*}
that associates to each $m\in \CM(\CI)$ a pair $(\mu,S)$ such that $S$ is a closed subset of $\CI$ and $\mu$ is a Radon measure on $\CI \setminus S$. By \eqref{eq:muS_to_m}, this mapping is one-to-one and onto, which allows us to identify a pair $(\mu,S)$ with the corresponding measure $m$, and thus the set of Markovian randomized stopping times with $\CM(\CI)$. With some abuse of notation, we will accordingly write $(\mu,S)\in \CM(\CI)$.

\paragraph{A Fixed-Point Theorem}

Proving that an Mpe exists in any lBwa requires applying an appropriate fixed-point theorem to the pbr correspondence. The main difficulty is that the domain $\CM(\CI)$ of this correspondence is not convex for the two natural vector-space structures we can think of. First, the set of csfs (or ccdfs) associated to Markovian randomized strategies is not convex.\footnote{This can be seen by considering the average of the csfs associated to the hitting times $\tau_{x}$ and $\tau_{y}$ for two points $x\neq y$ in $\CI$: the average csf jumps from $1$ to $\frac{1}{2}$ at $\tau=\tau_{x}\wedge \tau_{y}$, which contradicts \eqref{mark} applied at $\tau$ with $s=0$.} Second, because we allow the measures in $\CM(\CI)$ to take the value $\infty$ on compact sets, $\CM(\CI)$ is not a subset of the vector space of signed locally finite measures. Therefore, we cannot easily apply standard results such as Glicksberg's \cite{Glicksberg} infinite-dimensional generalization of Kakutani's \cite{Kakutani} fixed-point theorem, which requires a convex structure. Our proof is instead based on a more general fixed-point theorem due to Eilenberg and Montgomery \cite{EilenbergMontgomery}.\footnote{Debreu \cite{Debreu} and Reny \cite{Reny} use this theorem to prove the existence of a social equilibrium in an abstract economy and of a pure-strategy equilibrium in a class of static Bayesian games, respectively.}

Let us first recall the definition of an absolute retract appearing in Theorem \ref{thm:main} as well as the definition of a contractible space appearing in the fixed-point theorem we will use.

\begin{definition}\label{def:AR}
A metric space $(E,d)$ is an absolute retract (AR) if$,$ for any continuous map $f:E \rightarrow E'$ into a metric space $(E',d')$ such that $f$ is an homeomorphism between $E$ and $f(E)$ and $f(E)$ is closed in $E',$ there exists a continuous map $g: E' \rightarrow f(E)$ such that for all $x\in f(E),$ $g(x)=x$ (i.e.$,$ $f(E)$ is a retract of $E'$).
\end{definition}

\begin{definition}\label{def:contractible}
A metric space $(E,d)$ is contractible if there exists a continuous map $H:E\times [0,1] \rightarrow E$ and $x_0\in E$ such that $H(\cdot,0)={\rm Id}_E$ and $H(\cdot,1)= x_0$ (i.e.$,$ the identity map is homotopic to a constant map).
\end{definition}

The following result is a corollary of Eilenberg--Montgomery's fixed-point theorem.

\begin{theorem}[{\cite[Theorem 14.3]{McLennan}}] \label{thm:fixed_point}
If $(E,d)$ is a compact AR and $\Phi: E \twoheadrightarrow E$ is a correspondence with a closed graph and nonempty contractible values$,$ then $\Phi$ has a fixed point$,$ i.e.$,$ there exists $e^* \in E$ such that $e^* \in \Phi(e^*)$.
\end{theorem}

The importance of Theorem \ref{thm:compact_AR} is now clear. Theorem \ref{representation} and Lemma \ref{lem:mtomu} enable us to identify the set of Markovian randomized stopping times with $\CM(\CI)$, and Theorem \ref{thm:compact_AR} shows that $\CM(\CI)$ is a compact AR. This provides in turn the required foundation for applying Theorem \ref{thm:fixed_point}.

\paragraph{The Main Steps of the Proof}

The remainder of the paper is organized as follows:
\begin{enumerate}

\item

In Section \ref{sec:BR}, we show that there exists a pbr to any Markovian strategy (Proposition \ref{prop:existence_pbr}) and we provide a characterization of pbrs (Proposition \ref{prop:cacrt_pbr}). We also introduce a correspondence $\Phi$ in \eqref{defPhi}--\eqref{eq:def_phi_product} whose values are subsets of the pbr correspondence \eqref{eq:def_pbr_corr}, and to which we will eventually apply Theorem \ref{thm:fixed_point}.

\item

In Section \ref{sec:compact}, we show that the topology $\vartheta$ on $\CM(\CI)$ is compact and metrizable and extends the classical vague topology for Radon measures to the whole set $\CM(\CI)$ (Proposition \ref{prop:compact_metric}). We also show that convergence for this topology implies almost sure weak convergence of the associated csfs (Proposition \ref{prop:cvg_lambda}).

\item

In Section \ref{sec:closed_graph}, we show that the correspondence $\Phi$ has a closed graph.

\item

In Section \ref{sec:contractibility}, we prove Theorem \ref{thm:compact_AR} and we show that the correspondence $\Phi$ has contractible values. Except for the tools from general topology we use, the proof is relatively elementary, and is based on classical convolutions and orthogonal projections. Applying Theorem \ref{thm:fixed_point} to $\Phi$ finally concludes the proof of Theorem \ref{thm:main}.

\item

In Section \ref{sec:example}, we present an example of game which does not admit any Mpe in pure stopping times, but admits an Mpe in randomized stopping times that has a similar structure as the Mpes identified in \cite{DGM2} in a more specific framework.

\item

Finally, in Section \ref{sec:discussion}, we present a generalization of Theorem \ref{thm:main} to lBwas with more than two players, in which the game is over as soon as one player stops and the players' rewards do not depend on who stops first. We also discuss extending our analysis to more general classes of games and environments.

\end{enumerate}

\section{Existence and Characterization of Pbrs}\label{sec:BR}

\subsection{Existence of Pure Pbrs}

Let us fix $(\mu^j,S^j) \in \CM(\CI)$ and consider the following optimal stopping problem:
\begin{align}\label{eq:bestreply}
\bar J^i(x) : =\sup_{\gamma^i \in \CT_r} J^i(x,\gamma^i,(\mu^j,S^j))=\sup_{\tau^i \in \CT} J^i(x,\tau^i,(\mu^j,S^j)),\quad  x\in \CI,
\end{align}
where the second equality follows from Lemma \ref{usepure}. Our goal is to characterize the Markovian randomized stopping times that are optimal in \eqref{eq:bestreply} for all $x$, which by Definition \ref{def:mpe},  are the pbrs to $(\mu^j,S^j)$. By \eqref{forbpr}, we have
\begin{align*}
\forall \tau^i \in \CT, \; J^i(x,\tau^i,(\mu^j,S^j))= \EE_x \hskip 0.3mm [Y_{\tau^i}],
\end{align*}
where
\begin{align}\label{eq:defY}
Y_t:=\int_{[0,t)}\mathrm e^{-rs}G^i(X_s)\, \mathrm d\Gamma^j_s+ \mathrm e^{-rt} R^i(X_t) \Lambda^j_{t-}, \quad t\geq 0.
\end{align}
Notice that $Y_\infty=Y_{\infty-}=\int_{[0,\infty)}\mathrm e^{-rs}G^i(X_s)\, \mathrm d\Gamma^j_s$. Therefore, the study of player i's brvf to $(\mu^j, S^j)$ is tantamount to the study of problem
\begin{align}\label{optstop1}
\bar J^i(x)= \sup_{\tau^i \in \CT}\,\EE_x\hskip 0.3mm [Y_{\tau^i}]
\end{align}
and falls into the general theory of optimal stopping, from which we will borrow several results below. It follows from Proposition \ref{geneprop} and $\overline S\,\! ^i=\{\bar J^i=R^i\}$ that
\begin{align}\label{eq:inclusions}
S^j \cap \overline S\,\!^i \subset \{G^i=R^i\}\, \text{ and } \, S^j \setminus \overline S\,\!^i \subset \{ G^i>R^i\}.
\end{align}
Notice that the set $\overline S\,\!^i$ may be empty. We now prove that a pure pbr exists.

\begin{proposition}\label{prop:existence_pbr}
$\overline S\,\! ^i$ is closed and $(0,\overline S\,\!^i)$ is a pbr to $(\mu^j,S^j)$.
\end{proposition}

\begin{proof}
The proof consists of five steps.

\subparagraph{Step 1}

Observe first from $G^i \geq R^i$ that the process $Y:=(Y_t)_{t \geq 0}$ defined in \eqref{eq:defY} has c\`agl\`ad and lsc trajectories; specifically, the only potential discontinuity is at $\tau_{S^j}$ whenever $\tau_{S^j}< \infty$. Observe also that the value function of problem \eqref{optstop1} is not modified if we replace $Y$ with its right-continuous modification $\bar Y$, defined by
\begin{align}\label{eq:def_barY}
\bar Y_t:=Y_{t+}=\int_{[0,t]}\mathrm e^{-rs}G^i(X_s)\, \mathrm d\Gamma^j_s+ \mathrm e^{-rt}R^i(X_t)\Lambda^j_{t}, \quad t\geq 0.
\end{align}
Indeed, $\bar Y \geq Y$ and thus
\begin{align}\label{eq:zzbar1}
\bar J^i(x)\leq \sup_{\tau \in \CT}\,\EE_x\hskip 0.3mm[\bar Y_\tau],
\end{align}
and the reverse inequality follows from the fact that, by dominated convergence,
\begin{align}\label{eq:zzbar2}
\EE_x[\bar Y_\tau] =\lim_{n\to \infty} \EE_x \hskip 0.3mm [Y_{\tau+\frac{1}{n}}].
\end{align}
Notice that $\bar Y$ has c\`adl\`ag and usc trajectories. These remarks lead us to consider the problem
\begin{align}\label{optstop2}
\bar J^i(x)= \sup_{\tau \in \CT}\,\EE_x \hskip 0.3mm [\bar Y_\tau].
\end{align}
By Assumption A1, the processes $Y$ and $\bar Y$ are of class (D). Therefore, $\bar J^i$ is analytically measurable \cite[Proposition 2.4]{EKLM}, and thus universally measurable. In particular, for each $\tau\in \CT$, $\bar J^i(X_\tau)$ defines a random variable on $\Omega$. For each $x\in \CI$, let $Z^x$ denote the Snell envelope of $\bar Y$ on the stochastic basis $(\Omega,\CF,(\CF_t)_{t \geq 0},\PP_x)$. It is known (see \cite[Appendix 1, \S22]{DMB} and \cite[Theorem 2.28 and Proposition 2.29]{EK}) that $Z^x$ is a strong optional supermartingale of class (D) with almost surely c\`{a}dl\`{a}g paths and that
\begin{align*}
\forall \tau \in \CT, \; Z^x_\tau= \mathop{\mathrm {ess \, sup}} \limits _{\rho \geq \tau , \, \rho \in \CT} \, \EE_x \hskip 0.3mm[ \bar Y_{\rho} \! \mid \! \CF_{\tau}] .
\end{align*}
Using the same argument as in \eqref{eq:zzbar1}--\eqref{eq:zzbar2} with conditional expectations, one can check that $Z^x$ is also the Snell envelope of $Y$.

\subparagraph{Step 2}

We first claim that the Snell envelope $Z^x$ is indistinguishable under $\PP_x$ from the process $\widehat Z$ defined by
\begin{align} \label{eq:def_hatZ}
\widehat Z_t:=\int_{[0,t]}\mathrm e^{-rs}G^i(X_s)\, \mathrm d\Gamma^j_s+ \Lambda^j_{t}\,\mathrm e^{-rt}\bar J^i(X_t), \quad t \geq 0.
\end{align}
First, it follows from \cite{EKLM} that, for every stopping time $\tau$ of the canonical filtration $(\CF^0_t)_{t\geq 0}$,
\begin{align} \label{eq:Z=hatZ}
Z^x_{\tau}=\widehat Z_{\tau} \; \text{$\PP_x$-a.s.}
\end{align}
The proof of \eqref{eq:Z=hatZ} is detailed in the Appendix for the sake of completeness.
Given \eqref{eq:Z=hatZ}, to prove that $Z^x$ and $\widehat Z$ are indistinguishable, it is sufficient to show that  $\hat Z$ has c\`adl\`ag trajectories. To this end, we prove that  $\bar J^i$ is continuous on $\CI \setminus S^j$.

As for the continuity of $\bar J^i$ on $\CI \setminus S^j$, let $x\notin S^j$ and, for each $n\in \mathbb N$, consider the $(\CF^0_t)_{t \geq 0}$-stopping time $\tau_n$ defined as the first exit time by $X$ of an interval $(x-\delta,x+\varepsilon_n)$ for an arbitrary sequence $\varepsilon_n \to 0$ in $\mathbb R_+$ and a fixed $\delta>0$. We have $Z^x_0=\bar J^i(x)$ $\PP_x$-a.s., and, as $Z^x$ is $\PP_x$-a.s.\! right-continuous and of class (D),
\begin{align*}
\bar J^i(x)&=Z^x_0
\\
&= \lim_{n \to \infty} \EE_x\hskip 0.3mm[Z^x_{\tau_n}]
\\
&=\lim_{n \to \infty} \EE_x\hskip 0.3mm[\widehat Z_{\tau_n}]
\\
&=\lim_{n \to \infty} \EE_x\! \left[\int_{[0,\tau_n]}\mathrm e^{-rs}G^i(X_s)\, \mathrm d\Gamma^j_s+ \Lambda^j_{\tau_n}\mathrm e^{-r\tau_n}\bar J^i(X_{\tau_n})\right] \hskip -1mm,
\end{align*}
where the third equality follows from \eqref{eq:Z=hatZ}. Using that $\bar J^i(X_{\tau_n})=\bar J^i(x+\varepsilon_n)$ on $\{X_{\tau_n}=x+\varepsilon _n\}$, we have, for some constant $C>0$,
\begin{align*}
\left| \EE_x \! \left[ \int_{[0,\tau_n]}\mathrm e^{-rs}G^i(X_s)\, \mathrm d\Gamma^j_s+ \Lambda^j_{\tau_n}\mathrm e^{-r\tau_n}\bar J^i(X_{\tau_n})\right] - \bar J^i(x+ \varepsilon_n) \right|  \leq C \EE_x \hskip 0.3mm [\Gamma^j_{\tau_n}+ \indic_{X_{\tau_n}=x- \delta}] \to  0.
\end{align*}
This implies that $\lim_{n \to \infty} \bar J^i(x+\varepsilon_n)=\bar J^i(x)$, and thus that $\bar J^i$ is right-continuous at $x$ as the sequence $(\varepsilon_n)_{n \geq 0}$ in $\mathbb R_+$ is arbitrary. The proof of the left-continuity of $\bar J^i$ on $\CI \setminus S^j$ is similar and is thus omitted.

The continuity of $\bar J^i$ on $\CI \setminus S^j$ implies that $\widehat Z$ has c\`{a}dl\`{a}g trajectories, with a single potential discontinuity at $\tau_{S^j}$. Therefore, the processes $Z^x$ and $\widehat Z$ are indistinguishable. The claim follows. As the Snell envelope is defined up to an evanescent set, for all $x\in \CI$,  $\widehat Z$ is the Snell envelope of $\bar Y$ on the stochastic basis $(\Omega,\CF,(\CF_t)_{t \geq 0},\PP_x)$. We will use this fact in the subsequent steps.

\subparagraph{Step 3}

We now prove that $\overline S\,\!^i$ is closed. Consider a sequence $(x_n)_{n \geq 0}$ in $\overline S\,\!^i$ converging to $x \in\CI$. We need to show that $x\in \overline S\,\!^i$. If $x\notin S^j$, then $x\in \overline S\, \! ^i$ as $\bar J^i$ is continuous on $\CI \setminus S^j$ by Step 2. If $x\in S^j$, then it must be that $G^i(x)=R^i(x)$; otherwise, by Proposition  \ref{geneprop}(c), there would exist a neighborhood of $x$ which does not intersect $\overline S\!\,^i$, a contradiction. By Proposition \ref{geneprop}(b), $\bar J^i(x)=G^i(x)$ as $x\in S^j$. Thus $\bar J^i(x)=G^i(x)=R^i(x)$, which implies that $x\in \overline S\!\,^i$. This concludes the proof that $\overline S\,\!^i$ is closed.

\subparagraph{Step 4}

We next claim that $\tau_{\overline S\,\!^i}\wedge \tau_{S^j}$ is optimal for \eqref{optstop2}. As $\bar Y$ has usc trajectories and is of class (D), it follows from \cite[Theorem 2.41]{EK} that
\begin{align*}
\tau^* : =\inf \hskip 0.5mm \{ t \geq 0 : \widehat Z_t=\bar Y_t \}
\end{align*}
is the smallest optimal stopping time for \eqref{optstop2}. In turn \cite[Theorem 2.31]{EK}, the optimality of $\tau^*$ in \eqref{optstop2} is equivalent to the facts that:
\begin{itemize}

\item

$\widehat Z$ is a martingale up to $\tau^*$;

\item

$\widehat Z_{\tau^*}=\bar Y_{\tau^*}$.

\end{itemize}
Let us focus on the second condition. From \eqref{eq:def_barY} and \eqref{eq:def_hatZ}, $\widehat Z_t=\bar Y_t$ is equivalent to
\begin{align*}
\Lambda^j_t [\bar J^i(X_t)-R^i(X_t)]=0.
\end{align*}
Recalling that $\inf \hskip 0.5mm \{ t \geq 0 : \Lambda^j_t=0\}=\tau_{S^j}$, we deduce that
\begin{align*}
\mbox{$\tau^*= \inf \hskip 0.5mm\{t \geq 0 : \bar J^i(X_t)=R^i(X_t)\} \wedge \tau_{S^j} =\tau_{\overline S\!\,^i}\wedge \tau_{S^j}$ a.s.}
\end{align*}
The claim follows.

\subparagraph{Step 5}

We finally prove that $\tau_{\overline S\,\!^i}$ is optimal for \eqref{optstop1}. For each $x \in \CI$,
\begin{align*}
\bar J^i(x) &=\EE_x\hskip 0.3mm[ \bar Y_{\tau^*}] \allowdisplaybreaks
\\
&= \EE_x\hskip 0.3mm[ \widehat Z_{\tau^*}]
\\
&=\EE_x\! \left[ \int_{[0,\tau^*]}\mathrm e^{-rs}G^i(X_s)\, \mathrm d\Gamma^j_s+ \Lambda^j_{\tau^*} \,\mathrm e^{-r\tau^*}\bar J^i(X_{\tau^*})\right]
\\
&=\EE_x \Bigg[ \indic_{\tau_{\overline S\,\!^i}< \tau_{S^j}}\! \left[\int_{[0,\tau_{\overline S\!\,^i}]}\mathrm e^{-rs}G^i(X_s)\, \mathrm d\Gamma^j_s+ \Lambda^j_{\tau_{\overline S\,\!^i}}\mathrm e^{-r\tau_{\overline S\,\!^i}}R^i(X_{\tau_{\overline S\,\!^i}}) \right]
\\
& \hskip 0.78cm +\indic_{\tau_{\overline S\,\!^i}=\tau_{S^j}} \! \left[\int_{[0,\tau_{\overline S\,\! ^i}]}\mathrm e^{-rs}G^i(X_s)\, \mathrm d\Gamma^j_s+ \Lambda^j_{\tau_{\overline S\!\,^i}}\mathrm e^{-r\tau_{\overline S\,\!^i}}R^i(X_{\tau_{\overline S\!\,^i}}) \right]
\\
&  \hskip 0.78cm +\indic_{\tau_{\overline S\,\! ^i}> \tau_{S^j}}\! \left[ \int_{[0,\tau_{S^j}]}\mathrm e^{-rs}G^i(X_s)\, \mathrm d\Gamma^j_s+ \Lambda^j_{\tau_{S^j}}\mathrm e^{-r\tau_{S^j}}\bar J^i(X_{\tau_{S^j}}) \right] \!  \Bigg]\hskip -0.3mm ,
\end{align*}
where the first and second equalities follow from the optimality of $\tau^*$ in \eqref{optstop2}, and the fourth equality follows from $\bar J^i(X_{\tau_{\overline S\,\! ^i}})=R^i(X_{\tau _{\overline S\,\! ^i}})$ as $\overline S\!\,^i$ is closed.

Let us examine the three terms in the last expression separately. For the first one, using that $\Lambda^j$ is continuous on $[0, \tau_{S_j})$, we obtain
\begin{align*}
&\indic_{\tau_{\overline S\,\!^i}< \tau_{S^j}}\! \left[\int_{[0,\tau_{\overline S\!\,^i}]}\mathrm e^{-rs}G^i(X_s)\, \mathrm d\Gamma^j_s+ \Lambda^j_{\tau_{\overline S\,\!^i}}\mathrm e^{-r\tau_{\overline S\,\!^i}}R^i(X_{\tau_{\overline S\,\!^i}}) \right]
\\
&=\indic_{\tau_{\overline S\,\!^i}< \tau_{S^j}}\! \left[\int_{[0,\tau_{\overline S\!\,^i})}\mathrm e^{-rs}G^i(X_s)\, \mathrm d\Gamma^j_s+ \Lambda^j_{\tau_{\overline S\,\!^i}-}\mathrm e^{-r\tau_{\overline S\,\!^i}}R^i(X_{\tau_{\overline S\,\!^i}}) \right]
\\
&= \indic_{\tau_{\overline S\,\!^i}< \tau_{S^j}}Y_{\tau_{\overline S\,\!^i}}.
\end{align*}
For the second one, using that $G^i=R^i$ on $S^j \cap \overline S\,\!^i$, we obtain
\begin{align*}
&\indic_{\tau_{\overline S\,\!^i}= \tau_{S^j}}\! \left[\int_{[0,\tau_{\overline S\!\,^i}]}\mathrm e^{-rs}G^i(X_s)\, \mathrm d\Gamma^j_s+ \Lambda^j_{\tau_{\overline S\,\!^i}}\mathrm e^{-r\tau_{\overline S\,\!^i}}R^i(X_{\tau_{\overline S\,\!^i}}) \right] \allowdisplaybreaks
\\
&=\indic_{\tau_{\overline S\,\!^i} = \tau_{S^j}}\! \left[ \int_{[0,\tau_{\overline S\,\!^i})}\mathrm e^{-rs}G^i(X_s)\, \mathrm d\Gamma^j_s+ \mathrm e^{-r\tau_{\overline S\,\!^i}}(\Lambda^j_{\tau_{\overline S\,\!^i}-}-\Lambda^j_{\tau_{\overline S\!\,^i}})\,G^i(X_{\tau_{\overline S\,\!^i}})+ \Lambda^j_{\tau_{\overline S\,\!^i}}\mathrm e^{-r\tau_{\overline S\!\,^i}}R^i(X_{\tau_{\overline S\,\!^i}})\right]
\\
&=\indic_{\tau_{\overline S\,\!^i}= \tau_{S^j}} \! \left[\int_{[0,\tau_{\overline S\,\!^i})}\mathrm e^{-rs}G^i(X_s)\, \mathrm d\Gamma^j_s+ \Lambda^j_{\tau_{\overline S\!\,^i}-}\mathrm e^{-r\tau_{\overline S\,\!^i}}R^i(X_{\tau_{\overline S\,\!^i}})\right]
\\
&= \indic_{\tau_{\overline S\,\!^i}= \tau_{S^j}}Y_{\tau_{\overline S\,\!^i}}.
\end{align*}
For the third one, using that $\Lambda^j=0$ on $[\tau_{S^j},\infty)$, we obtain
\begin{align*}
& \indic_{\tau_{\overline S\,\!^i}> \tau_{S^j}} \! \left[\int_{[0,\tau_{S^j}]}\mathrm e^{-rs}G^i(X_s)\, \mathrm d\Gamma^j_s+ \Lambda^j_{\tau_{S^j}} \mathrm e^{-r\tau_{S^j}}R^i(X_{\tau_{S^j}})\right]
\\
&=\indic_{\tau_{\overline S\,\!^i}> \tau_{S^j}} \int_{[0,\tau_{S^j}]}\mathrm e^{-rs}G^i(X_s)\, \mathrm d\Gamma^j_s
\\
&= \indic_{\tau_{\overline S\,\! ^i}> \tau_{S^j}} \! \left[ \int_{[0,\tau_{\overline S\,\!^i})}\mathrm e^{-rs}G^i(X_s)\, \mathrm d\Gamma^j_s+ \Lambda^j_{\tau_{\overline S\,\! ^i}-}\mathrm e^{-r\tau_{\overline S\,\! ^i}}R^i(X_{\tau_{S^i}+})\right]
\\
&= \indic_{\tau_{\overline S\,\! ^i}> \tau_{S^j}}Y_{\tau_{\overline S\,\!^i}}.
\end{align*}
Gathering these three equalities, we obtain:
\begin{align*}
\bar J^i(x) =\EE_x \! \left[ \left(\indic_{\tau_{\overline S\, \! ^i}<\tau_{S^j}}+\indic_{\tau_{\overline S\!\, ^i}=\tau_{S^j}}+\indic_ {\tau_{ \overline S\,\!^i}>\tau_{S^j}} \right) \! Y_{\tau_{\overline S\,\!^i}}\right] =\EE_x \! \left[ Y_{\tau_{\overline S\,\!^i}}\right]\!,
\end{align*}
from which it follows that $\tau_{\overline S\,\!^i}$ is optimal in problem \eqref{optstop1} and thus that $(0,\overline S\,\!^i)$ is a pbr to $(\mu^j,S^j)$. Hence the result.
\end{proof}

It follows from this result and the definition of $\overline S\,\!^i$ that $\overline S\,\!^i$ is the largest set on which it is optimal for player $i$ to stop in a pbr to $(\mu^j, S^j)$. Equilibrium may however require that player $i$ stop on a smaller set, and possibly mix on the complement of this set. Thus we need to characterize all pbrs to $(\mu^j, S^j)$, a task to which we now turn.

\subsection{Characterization of Pbrs}

Define the set
\begin{align}
\Sigma^i : =\{ S^i \subset \CI \text{ closed}: (0,S^i) \text{ is a pbr to } (\mu^j,S^j) \}. \label{purepbrcorr}
\end{align}
The characterization of pbrs relies on the following lemma.

\begin{lemma}\label{lem:sigma_i}
The set $\Sigma^i$ is a nonempty lattice$,$ with smallest element $\underline S^i$ and largest element $\overline S\,\!^i,$ and
\begin{align*}
\Sigma^i=\{ S^i \subset \CI \text{ closed} : \underline S^i \subset S^i \subset \overline S \,\! ^i\}.
\end{align*}
\end{lemma}

\begin{proof}
By Proposition \ref{prop:existence_pbr}, $\overline S\!\,^i \in \Sigma^i$, so that $\Sigma^i$ is nonempty. Moreover, by Proposition \ref{geneprop}(ii), every element of $\Sigma^i$ is a subset of $\overline S\,\!^i$. Now, let $S^i \subset \overline S\, \! ^i$ be closed and recall that $R^i=G^i$ on $S^j \cap  \overline S\,\!^i$ by \eqref{eq:inclusions}. We have a.s.
\begin{align}
\widehat Z_{\tau_{S^i}}&=\int_{[0,\tau_{S^i}]}\mathrm e^{-rs}G^i(X_s)\, \mathrm d\Gamma^j_s+ \Lambda^j_{\tau_{S^i}}\mathrm e^{-r\tau_{S^i}}\bar J^i(X_{\tau_{S^i}}) \nonumber
\\
&=\int_{[0,\tau_{S^i}]}\mathrm e^{-rs}G^i(X_s)\, \mathrm d\Gamma^j_s+ \Lambda^j_{\tau_{S^i}}\mathrm e^{-r\tau_{S^i}}R^i(X_{\tau_{S^i}}) \nonumber
\\
&=\int_{[0,\tau_{S^i})}\mathrm e^{-rs}G^i(X_s)\, \mathrm d\Gamma^j_s+ \Lambda^j_{\tau_{S^i}-}\mathrm e^{-r\tau_{S^i}}R^i(X_{\tau_{S^i}}) \nonumber
\\
&= Y_{\tau_{S^i}}, \label{eq:eqZY}
\end{align}
where the first equality follows from \eqref{eq:def_hatZ}--\eqref{eq:Z=hatZ}, the second equality follows from $X_{\tau_{S^i}} \in \overline S\,\! ^i$, and the third equality follows from the fact that either $\Delta \Gamma^j_{\tau_{S^i}}=0$ or $\tau_{S^i}=\tau_{S^j}$, the latter implying that $R^i(X_{\tau_{S^i}})=G^i(X_{\tau_{S^i}})$ as in the proof of Proposition \ref{prop:existence_pbr}. The remainder of the proof consists of two steps.

\subparagraph{Step 1}

We first show that $\Sigma^i$ is stable by intersection, i.e., that, given two subsets $S^i$ and $\widehat S^i$ of $\Sigma^i$, the stopping time $\tau_{S^i\cap\widehat S^i}=\tau_{S^i}\vee \tau_{\widehat S^i}$ is optimal in \eqref{optstop1}. By \cite[Theorem 2.31]{EK}, it is sufficient to prove that $Y_{\tau_{S^i\cap\widehat S^i}}=\widehat Z_{\tau_{S^i\cap\widehat S^i}}$ and that $\widehat Z$ is a martingale up to $\tau_{S^i\cap\widehat S^i}$. The first property follows from \eqref{eq:eqZY}. As for the martingale property, because $\widehat Z$ is a supermartingale, we only need to verify that for each $x \in \CI$, $\EE_x \! \left[\widehat Z_{\tau_{S^i \cap \widehat S^i}}\right]=\EE_x \hskip 0.3mm[\widehat Z_0]$. We have
\begin{align*}
\EE_x\!\left[ \widehat Z_{\tau_{S^i \cap \widehat S^i}}\right]&= \EE_x\!\left[ \widehat Z_{\tau_{S^i}}\indic_{\tau_{\widehat S^i} \leq \tau_{S^i}} + \widehat Z_{\tau_{\widehat S^i}}\indic_{\tau_{\widehat S^i} > \tau_{S^i}}\right]
\\
&=\EE_x\!\left[ \widehat Z_{\tau_{\widehat S^i}}\indic_{\tau_{\widehat S^i} \leq \tau_{S^i}} +\widehat Z_{\tau_{\widehat S^i}}\indic_{\tau_{\widehat S^i} > \tau_{S^i} }\right]
\\
&=\EE_x\!\left[ \widehat Z_{\tau_{\widehat S^i}}\right]
\\
&=\EE_x\hskip 0.3mm[\widehat Z_0]
\end{align*}
where the second and fourth equalities follow from the fact that $\widehat Z$ is a martingale up to $\tau_{S^i}$ and $\tau_{\widehat  S^i}$, respectively. We conclude that $\tau_{S^i\cap\widehat S^i}=\tau_{S^i}\vee \tau_{\widehat S^i}$ is indeed optimal in \eqref{optstop1} and  thus that $S^i \cap \widehat S^i \in \Sigma^i$.

\subparagraph{Step 2}

Now, let us define $\underline S^i: =\bigcap_{S^i \in \Sigma^i} S^i$. Because $\CI\setminus \underline S^i$ is the union of the open sets $\CI \setminus S^i$ for $S^i \in \Sigma^i$, which admits a countable subcover as any open subset of the real line is a Lindel\"of space, there exists a sequence $(S^i_n)_{n \geq 0}$ in $\Sigma^i$ such that $\underline S^i=\bigcap_{n \geq 0} S^i_n$. As a result, $\underline S^i$ is the intersection of the nonincreasing sequence of closed sets $(\widehat S^i_n)_{n \geq 0} := (\bigcap_{p=1}^n S^i_p)_{n \geq 0}$ in $\Sigma^i$. The sequence $(\tau_{ \widehat S^i_n})_{n \geq 0}$ is nondecreasing, and thus $\lim_{n \to \infty} \tau_{\widehat S^i_n}$ is well-defined. We claim that $\tau_{\underline S^i}=\lim_{n \to \infty} \tau_{\widehat S^i_n}$. It is clear that $\tau_{\underline S^i}\geq \lim_{n \to \infty} \tau_{\widehat S^i_n}$. If this limit is infinite, then the equality holds. If this limit is finite, then $X_{\lim_{n \to \infty} \tau_{\widehat S^i_n}}$ belongs to $\underline S^i$ and thus $\tau_{ \underline S^i}\leq \lim_{n \to \infty} \tau_{\widehat S^i_n}$, so that the equality again holds. The claim follows. Because $\tau_{\widehat S^i_n}$ is optimal in \eqref{optstop1}, it follows from \eqref{eq:eqZY} that $\bar J^i(x)= \EE_x \big[\widehat Z_{\tau_{\widehat S^i_n}}\big]= \EE_x \big[Y_{\widehat S^i_n}\big]$ for all $n \geq 0$. On the other hand, \eqref{eq:eqZY} implies $\widehat Z_{\tau_{\underline S^i}}= Y_{\tau _{\underline S^i}}$ a.s.\! and we conclude that $\bar J^i(x)= \EE_x \big[\widehat Z_{\tau_{\underline S^i}}\big]= \EE_x \big[Y_{\underline S^i}\big]$ by dominated convergence using that $Y$ is left-continuous. In particular, $\underline S^i\in \Sigma^i$ and $\underline S^i$ is the smallest element of $\Sigma^i$. It follows that $\Sigma^i\subset \{ S^i \subset \CI \text{ closed} : \underline S^i \subset S^i \subset \overline S\,\!^i\}$. To prove the reverse inclusion, it suffices to notice that, if $\underline S^i \subset S^i \subset \overline S\,\!^i$, then \eqref{eq:eqZY} implies $\widehat Z_{\tau_{S^i}}= Y_{\tau_{S^i}}$ and that $\widehat Z$ is a martingale up to $\tau_{S^i}$ as it is a martingale up to $\tau_{\underline S^i} \geq \tau_{S^i}$. The result follows.
\end{proof}

We are now ready to characterize the set of pbrs to $(\mu^j,S^j)$.

\begin{proposition}\label{prop:cacrt_pbr}
$(\mu^i,S^i)$ is a pbr to $(\mu^j,S^j)$ if and only if $\underline S^i \subset S^i \subset \overline S\,\!^i$ and $\mu^i$ is a Radon measure on $\CI \setminus S^i$ that is concentrated on $(\overline S\,\!^i \setminus S^i) \cup S^j$.
\end{proposition}

\begin{proof}
If $(\mu^i,S^i)$ is a pbr to $(\mu^j,S^j)$, then, by Proposition \ref{geneprop}(iv), $(0,S^i)$ is also a pbr to $(\mu^j,S^j)$ and thus $S^i \in \Sigma^i$, which proves the inclusions by Lemma \ref{lem:sigma_i}. The second point follows directly from Proposition \ref{geneprop}-(iii--iv).
Conversely, if $\underline S^i \subset S^i \subset \overline S\,\!^i$ and $\mu^i$ is a Radon measure on $\CI \setminus S^i$ that is concentrated on $(\overline S\!\,^i \setminus S^i) \cup S^j$, then $(0,S^i)$ is a pbr to $(\mu^j,S^j)$ by Lemma \ref{lem:sigma_i}, and $(\mu^i,S^i)$ is a pbr to $(\mu^j,S^j)$ by Proposition \ref{geneprop}-(iv). Hence the result.
\end{proof}

\subsection{The Correspondence $\Phi$}

We now consider a correspondence $\Phi^i$ whose values are nonempty subsets of the values of $PBR^i$. Specifically, for each $(\mu^j,S^j)\in\CM(\CI)$, let
\begin{align} \label{defPhi}
\Phi^i(\mu^j,S^j):= \{ (\mu^i,S^i) \in \CM(\CI) : \underline S^i \subset S^i \subset \overline S\,\!^i \text{ and $\mu^i$ is concentrated on $\overline S\!\,^i\setminus S^i$} \}.
\end{align}
We will apply the fixed-point Theorem \ref{thm:fixed_point} to the correspondence $\Phi : \CM(\CI) \times \CM(\mathcal I) \twoheadrightarrow \mathcal M(\mathcal I) \times \mathcal M(\mathcal I)$ defined by
\begin{align} \label{eq:def_phi_product}
\Phi((\mu^1,S^1),(\mu^2,S^2)):=\Phi^1(\mu^2,S^2)\times \Phi^2(\mu^1,S^1).
\end{align}
Our approach is justified by the fact that $\Phi^i$ takes values in the set of pbrs of player $i$, as the following result shows.

\begin{lemma}\label{lem:caractPhi}
For all $(\mu^j,S^j)\in \CM(\CI),$
\begin{align} \label{egalPhi}
\Phi^i(\mu^j,S^j)= \{ (\mu^i,S^i) \in PBR^i(\mu^j,S^j) : \mu^i(S^j \cap\{G^i >R^i\})=0 \}.
\end{align}
\end{lemma}
\begin{proof}
Let $(\mu^i, S^i) \in \Phi^i(\mu^j,S^j)$.
By Proposition \ref{prop:cacrt_pbr} and Definition \ref{defPhi}, we have $\Phi^i(\mu^j,S^j) \subset PBR^i(\mu^j,S^j)$.
By \eqref{eq:inclusions}, $S^j \cap \overline S\,\!^i  \subset \{G^i=R^i\}$, and thus
\begin{align*}
\mu^i(S^j \cap\{G^i >R^i\})=\mu^i( S^j \cap \overline S\,\!^i \cap \{G^i >R^i\})=0.
\end{align*}
Conversely, let $(\mu^i,S^i) \in PBR^i(\mu^j,S^j)$ such that $\mu^i(S^j \cap\{G^i >R^i\})=0$. By Proposition \ref{prop:cacrt_pbr}, $\underline S^i \subset S^i \subset \overline S\,\!^i$ and $\mu^i$ is concentrated on $(\overline S\,\!^i\setminus S^i) \cup S^j$. By \eqref{eq:inclusions}, $S^j \setminus \overline S\,\!^i  \subset S^j \cap \linebreak \{ G^i>R^i\}$, and thus
\begin{align*}
\mu^i(S^j\setminus \overline S\,\!^i) \leq \mu^i(S^j \cap\{G^i >R^i\})=0.
\end{align*}
Hence, $\mu^i$ is concentrated on $\overline S\!\,^i\setminus S^i$, so that $(\mu^i,S^i)\in \Phi^i(\mu^j,S^j)$. The result follows.
\end{proof}

\section{A Compact Topology on $\CM(\CI)$}\label{sec:compact}

Recall that $\CM(\CI)$ denotes the set of nonnegative regular measures $m:\CB( \CI)\to[0,\infty]$, i.e., such that 
\begin{align*}
m(A)=\inf \hskip 0.5mm\{ m(O): A\subset O, \text{ $O$ open}\} = \sup\hskip 0.5mm \{ m(K): K\subset A, \text{$K$ compact}\}, \quad A \in \CB(\CI).
\end{align*}
The proof of the following result follows along more or less standard lines (see, e.g., \cite[chapter 4]{Kallenberg}) and is therefore postponed to the Appendix.

\begin{proposition}\label{prop:compact_metric}
The topology $\vartheta$ on $\CM(\CI)$ defined in Theorem \ref{thm:compact_AR} is metrizable and \linebreak compact. Moreover$,$
\begin{enumerate}

\item

for every open set $O\subset \CI,$ the mapping $\CM(\CI) \to [0,\infty] : m  \mapsto m(O)$ is lsc$;$

\item

for every compact set $K \subset \CI,$ the mapping $\CM(\CI) \to [0,\infty] : m  \mapsto m(K)$ is usc$;$

\item

a sequence $(m_n)_{n \geq 0}$ converges to $m$ if and only if
\begin{itemize}

\item

for every open set $O$ such that $O \cap e(m) \neq \emptyset,$ $m_n(O) \rightarrow \infty;$

\item

letting $L_\phi(m):= \int_\CI \phi \, \mathrm dm,$ $L_\phi( m_n) \rightarrow L_\phi(m)$ for all $\phi\in \CC_c^+(\CI \setminus e(m))$.

\end{itemize}
\end{enumerate}
\end{proposition}

\begin{remark}\label{remark:vague}
Observing that $m(e)= \emptyset$ for any $m \in {\CM}_{\rm{loc}}(\CI)$, the second part of point (3) implies that the topology induced by $\vartheta$ on ${\CM}_{\rm{loc}}(\CI)$ coincides with the usual vague topology (see$,$ e.g.$,$ \cite[Chapter 4]{Kallenberg}).
\end{remark}

The next result is very important as it will allow us to prove convergence of expected payoffs under appropriate assumptions.

\begin{proposition}\label{prop:cvg_lambda}
Suppose that $(\mu_n,S_n) \to (\mu,S)$ in $\CM(\CI),$ and let $\Lambda^n,\Lambda$ denote the csfs associated with $(\mu_n,S_n)$ and $(\mu,S),$ respectively. Then
\begin{align*}
\forall t\neq \tau_S, \;  \Lambda^n_t \rightarrow \Lambda_t\; a.s.
\end{align*}
\end{proposition}
\begin{proof}
For each $t \geq 0$, we have
\begin{align}
\Lambda^n_t= \indic_{t < \tau_{S_n}}\, \mathrm e^{- \int_{\mathcal I \setminus S_n} L^y_t \, \mu_n(\mathrm d y)} \, \text{ and } \, \Lambda_t =\indic_{t < \tau_{S}} \, \mathrm e^{- \int_{\mathcal I \setminus S} L^y_t \, \mu(\mathrm d  y)} . \label{crnp}
\end{align}
Recall that by convention $\Lambda^n_\infty=\Lambda_\infty:=0$, so that $\Lambda^n,\Lambda$ define for each $\omega \in \Omega$ the survival function of a probability measure on $[0,\infty]$. Let $x\in \CI$. We distinguish two cases.

\subparagraph{Case 1: $t<\tau_S$}

Notice that there exists a set $\Omega_1$ of $\PP_x$-probability $1$ such that, for each $\omega \in \Omega_1$, the mapping $(t,y)\mapsto L^y_t(\omega)$ is continuous (see, e.g., \cite[Chapter VI, \S1, Theorem 1.7]{RevuzYor}). Define $M_t=\max_{0\leq s \leq t} X_s$ and $m_t=\min_{0\leq s \leq t}X_s$. Using the occupation time formula \cite[Chapter VI, \S1, Corollary 1.6]{RevuzYor}), for every interval $A\subset \CI $,
\begin{align*}
\int_0^t \indic_A(X_s)\sigma^2(X_s)\, \mathrm ds = \int_\CI \indic_A(y) L^y_t \,\mathrm dy \;\,\PP_x\text{-a.s.}
\end{align*}
Therefore, there exists a set $\Omega_2$ of $\PP_x$-probability $1$ such that the above equality holds for all $A$ with rational endpoints. Now, notice that $\{t <\tau_S\}=\{ [m_t,M_t] \subset \CI \setminus S\}$. Fix $\omega \in \Omega_1\cap \Omega_2$ such that $\tau_S(\omega)>0$. Then, for $t<\tau_S(\omega)$, we have, for every interval $A\subset \CI \setminus [m_t(\omega),M_t(\omega)]$ with rational endpoints,
\begin{align*}
0=\int_0^t \indic_A(X_s(\omega))\sigma^2(X_s(\omega))\, \mathrm ds = \int_\CI \indic_A(y) L^y_t(\omega) \, \mathrm dy
\end{align*}
as $X_s \in [m_t,M_t]$ for all $s\in [0,t]$. Using the continuity of $y\mapsto L^y_t(\omega)$, we deduce that $L^y_t(\omega)=0$ for all $y\in \CI\setminus [m_t(\omega),M_t(\omega)]$, and therefore that $y\mapsto L^y_t(\omega)$ is a continuous function with compact support $K_t(\omega) \subset \CI\setminus S$. Now, for each $n\geq 0$, let $m_n$ denote the measure in $\CM(\CI)$ associated to $(\mu_n,S_n)$. Because the restriction of $m_n$ to $\CI \setminus S$ converges vaguely to $\mu$ by Theorem \ref{prop:compact_metric}(3), $K_t(\omega) \cap S_n= \emptyset$ for any sufficiently large $n$, and thus the restrictions of $\mu_n$ and $m_n$ to $K_t(\omega)$ coincide. As a result, for each $\omega \in \Omega_1\cap \Omega_2 \cap\{t <\tau_S\}$, $\int_{\CI \setminus S^n} L^y_t(\omega) \, \mu_n(\mathrm d y) = \int_{K_t(\omega)} L^y_t(\omega) \, \mu_n(\mathrm d y) \to \int_{K_t( \omega)} L^y_t(\omega) \, \mu(\mathrm d y) = \int_{\CI\setminus S} L^y_t(\omega) \, \mu(\mathrm dy)$ and thus $\Lambda^n_t (\omega) \to \Lambda_t( \omega) $ by \eqref{crnp}.

\subparagraph{Case 2: $t>\tau_S$}

We first claim that, for each $z\in \CI$, letting $\tau_z$ denote the hitting time of $z$, we have, for every event $A\in \CF_{\tau_z}$
\begin{align*}
\PP_x \hskip 0.3mm[ L^z_t >0, t>\tau_z, A]=\PP_x\hskip 0.3mm[t>\tau_z, A].
\end{align*}
First, because $t\mapsto L_t^z$ is a strongly additive functional of the diffusion process $X$ \cite[Part I, Chapter II, Section 2, \S13, and Section 4, \S21]{BorodinSalminen}, we have, with $\PP_x$-probability $1$ on $\{\tau_z<t\}$,
\begin{align*}
L^z_{t}(\omega)= L^z_{\tau_z(\omega)}(\omega)+ L^z_{t-\tau_z(\omega)}(\theta_{\tau_z(\omega)}(\omega))=L^z_{t-\tau_z (\omega)}(\theta_ {\tau_z (\omega)} (\omega)).
\end{align*}
Then, denoting by $\tilde \Omega$ a copy of the canonical space endowed with the probabilities $\tilde \PP_y=\PP_y$ for $y\in \CI$,
\begin{align*}
\PP_x\hskip 0.3mm[ L^z_{t} >0,t>\tau_z, A]= \EE_x\hskip 0.3mm[\tilde\PP_z\hskip 0.3mm[ L^z_{t-\tau_z(\omega)}(\tilde \omega)>0]\indic_{t>\tau_z (\omega)} \indic_A(\omega)]=\PP_x\hskip 0.3mm[ t>\tau_z, A],
\end{align*}
where the first equality follows from the Markov property, and the second equality follows from the fact that $\PP_y\hskip 0.3mm[L_t^y>0]=1$ for all $y\in \CI$ and $t>0$ (see, e.g., \cite[Chapter VI, \S2, Proof of Proposition 2.5]{RevuzYor}). The claim follows. Now, if $t>\tau_S$, then it must be that either $X_{\tau_S}=x$ if $x\in S$ or $X_{\tau_S} \in\{a,b\}\subset S$ if $x \notin S$, where $(a,b)$ denotes the largest open interval containing $x$ in $\CI \setminus S$.
We claim that
\begin{align*}
\PP_x \hskip 0.3mm \big[L^{X_{\tau_S}}_t >0, t>\tau_S\big]=\PP_x \hskip 0.3mm[t>\tau_S].
\end{align*}
If $x\in S$, $X_{\tau_S}=x$ and $\tau_S=0$, so both sides are equal by the same reasoning as above. If $x\notin S$, we have, again by the same reasoning,
\begin{align*}
\PP_x \hskip 0.3mm\big[L^{X_{\tau_S}}_t >0, t>\tau_S\big]&= \PP_x\hskip 0.3mm[L^a_t >0, t>\tau_a, X_{\tau_S}=a] +\PP_x\hskip 0.3mm[L^b_t >0, t>\tau_b, X_{\tau_S}=b]
\\
&=\PP_x\hskip 0.3mm[t>\tau_a, X_{\tau_S}=a] +\PP_x\hskip 0.3mm[t>\tau_b, X_{\tau_S}=b]
\\
&= \PP_x\hskip 0.3mm[t>\tau_S].
\end{align*}
The claim follows. Let $\Omega_3$ be a set of $\PP_x$-probability $1$ such that $L_t^{ X_{\tau_S}(\omega)}(\omega)>0$ for all $\omega \in \Omega _3 \cap \{t>\tau_S\}$. For $\omega \in \Omega_1\cap \Omega_2\cap \Omega_3 \cap\{t >\tau_S\}$, the mapping $y \mapsto L_t^y(\omega)$ is continuous, vanishes outside of $[m_t(\omega),M_t(\omega)]$, and $L_t^{ X_{\tau_S}(\omega)}(\omega)>0$. By continuity, it must be that $m_t(\omega)<X_{\tau_S} (\omega) <M_t(\omega)$, and there exist $\varepsilon(\omega),\eta(\omega)>0$ such that $L^y_t(\omega) \geq \eta(\omega)$ for all $y\in (X_{\tau_S}(\omega)-\varepsilon(\omega),X_{\tau_S}(\omega)+\varepsilon(\omega)) \subset [m_t(\omega),M_t(\omega)]$. Because $(\mu_n,S_n) \to (\mu,S)$ and $X_{\tau_S}(\omega)\in S$, it must be that $m_n((X_{\tau_S}(\omega)-\varepsilon(\omega),X_{\tau_S}(\omega)+\varepsilon(\omega))) \to \infty$, where $m_n$ denotes the measure in $\CM(\CI)$ associated to $(\mu_n,S_n)$. Notice that $\indic_{t<\tau_{S_n}(\omega)} \neq 0$ if and only if $S_n\cap [m_t(\omega),M_t (\omega)] = \emptyset$, which implies $(X_{\tau_S}(\omega)-\varepsilon(\omega),X_{\tau_S}(\omega)+\varepsilon(\omega)) \subset \CI\setminus S_n$ and $m_n((X_{\tau_S}(\omega)-\varepsilon(\omega),X_{\tau_S}(\omega)+\varepsilon(\omega)))=\mu_n((X_ {\tau_S}(\omega)- \varepsilon (\omega),X_{\tau_S}(\omega)+\varepsilon(\omega)))$. We deduce that
\begin{align*}
0 \leq \Lambda^n_t(\omega) =\indic_{t<\tau_{S_n}(\omega)} \, \mathrm e^{- \int_{\CI \setminus S_n} L^y_t(\omega) \, \mu_n(\mathrm d y)} \leq \indic_{t<\tau_{S_n}( \omega)} \,  \mathrm e^{-  \eta(\omega) \mu_n((X_ {\tau_S}(\omega)- \varepsilon (\omega),X_{\tau_S}(\omega)+\varepsilon(\omega)))}\rightarrow 0,
\end{align*}
which concludes the proof because $\Lambda_t(\omega) = 0$ by \eqref{crnp} as $t>\tau_S(\omega)$. Hence the result. \end{proof}

\section{Closedness of the Graph of $\Phi$}\label{sec:closed_graph}

We will use the following classical result.

\begin{theorem}[{\cite[Theorems 5.1 and 5.4]{Billingsley}}]\label{thm:bill}
Let $E$ be a Polish space and $(\nu_n)_{n \geq 0}$ a sequence of probability measures on $E$ that converges weakly to $\nu$. Suppose that $f: E\rightarrow \RR$ is a measurable function such that $\nu(D)=0,$ where $D$ is the set of discontinuity points of $f,$ and that the variables of law $\nu_n\circ f^{-1}$ are uniformly integrable$,$ i.e.$,$
\begin{align*}
\lim_{M \to \infty}\sup_{n \geq 0} \int_E |f(x)|\indic_{|f(x)|\geq M} \, \nu_n(\mathrm d x) =0.
\end{align*}
Then
\begin{align*}
\int_E f \, \mathrm d\nu_n \to  \int_E f \, \mathrm d\nu.
\end{align*}
\end{theorem}

\medskip

The following result then holds.

\begin{proposition}\label{prop:closed_graph}
The correspondence $\Phi$ defined by \eqref{eq:def_phi_product} has a closed graph.
\end{proposition}

\begin{proof}
Because $\Phi$ is defined as a cartesian product, it is sufficient to prove that, for each $i=1,2$, $\Phi^i$ has a closed graph in $\CM(\CI)\times \CM(\CI )$. Because $\CM(\CI)$ is metrizable, it is sufficient to prove that the graph of $\Phi^i$ is sequentially closed. Let us therefore consider a sequence $((\mu^i_n, S^i_n ), (\mu^j_n,S^j_n))_{n \geq 0}$ in $\CM(\CI)\times \CM(\CI)$ such that, for each $n \geq 0$, $( \mu^i_n,S^i_n) \in \Phi^i(\mu^j_n,S^j_n)$. Assume further that this sequence converges to a limit $((\mu^1,S^1),(\mu^2,S^2))$. We need to prove that $(\mu^i,S^i) \in \Phi^i(\mu^j,S^j)$ or, equivalently, by Lemma \ref{lem:caractPhi}, that $(\mu^i,S^i)$ is a pbr to $(\mu^j,S^j)$ such that $\mu^i(S^j\cap \{G^i> R^i\})=0$.

By Proposition \ref{prop:cvg_lambda}, for each $t \neq \tau_{S^i}$, $\Lambda^{i}_{n,t} \to \Lambda^i_t$ a.s., where the processes $\Lambda^i$ and $\Lambda^{i}_n$ are the csfs associated to $(\mu^i,S^i)$ and $(\mu^i_n,S^i_n)$, respectively. Recall also that
\begin{align}
J^i(x,\Gamma^i,\Gamma^j) &=  \mathbb E _x \! \left[ \int_{[0,\infty)} \mathrm e^{-r t} R^i(X_t)\Lambda^j_{t-} \, \mathrm d\Gamma^i_t+ \int_{[0,\infty)} \mathrm e^{-r t}G^i(X_t)\Lambda^i_t \, \mathrm d\Gamma^j_t\right] \hskip -1mm, \label{eq:forbpr}
\\
J^i(x,\Gamma^{i}_n,\Gamma^{j}_n) &=  \mathbb E _x \! \left[ \int_{[0,\infty)} \mathrm e^{-r t} R^i(X_t)\Lambda^{j}_{n,t-} \, \mathrm d\Gamma^{i}_{n,t} + \int_{[0,\infty)} \mathrm e^{-r t}G^i(X_t)\Lambda^{i}_{n,t} \, \mathrm d\Gamma^{j}_{n,t}\right] \hskip -1mm. \label{eq:forbpr2}
\end{align}
The remainder of the proof consists of three steps.

\subparagraph{Step 1}

We first prove that $S^i\cap S^j\cap \{R^i<G^i\}=\emptyset$. Hence suppose, by way of contradiction, that $x\in S^i\cap S^j\cap \{R^i<G^i\}$. For each $n \geq 0$, because $(\mu^i_n,S^i_n) \in \Phi^i(\mu^j_n,S^j_n)$, we have $\bar J^i(y,(\mu^j_n,S^j_n))=R^i(y)$ for all $y\in S^i_n\cup \mathrm{ supp}\,\mu^i_n$, and therefore stopping immediately gives a weakly larger payoff than never stopping, i.e.,
\begin{align}\label{eq:ineqRG}
R^i(y) \geq \EE_y\! \left[\int_{[0,\infty)} \mathrm e^{-rs}G^i(X_s)\, \mathrm d\Gamma^{j}_{n,s}\right]\!=:\kappa_n(y).
\end{align}
Now, because $x \in S^i$ and $m^i_n: = (\mu_n,S_n) \to (\mu^i,S^i)$, for each $\varepsilon >0$ we have $m_n((x-\varepsilon,x+\varepsilon)) \to \infty$, and there exists $N(\varepsilon) >0$ such that $m_n((x-\varepsilon,x+\varepsilon))>0$ and thus $(x-\varepsilon,x+\varepsilon) \cap \mathrm {supp}\, m_n \not = \emptyset$ for all $n \geq N(\varepsilon)$. Hence, because $\mathrm {supp}\, m_n = S^i_n \cup \mathrm {supp}\, \mu_n$ for all $n \geq 0$, there exists a sequence $x_n \to x$ such that for each $n \geq 0$, $x_n  \in S^i_n\cup \mathrm{ supp} \, \mu^i_n$ and thus satisfies \eqref{eq:ineqRG}. Fix some $\varepsilon>0$ such that $G^i > R^i(x)+\varepsilon$ on $[x-\varepsilon,x+\varepsilon]$. By the above reasoning, we can assume that $x_n \in (x-\varepsilon,x+\varepsilon)$ for $n$ large enough.

Next, observe that the mapping $[0,\infty] \to \RR: t \to \mathrm e^{-rt}G^i(X_t)$, which is a.s.\! equal to 0 at $\infty$ by Assumption A2, is a.s.\! continuous and bounded, and that the sequence of probabilities on $[0,\infty]$ with csfs $(\Lambda^{j}_{n})_{n \geq 0}$ converges weakly to the probability $\nu$ with csf $\Lambda^j$ by Proposition \ref{prop:cvg_lambda}. Therefore, we can apply Theorem \ref{thm:bill} with $E=[0,\infty]$ to obtain
\begin{align*}
\int_{[0,\infty)} \mathrm e^{-rs}G^i(X_s)\, \mathrm d\Gamma^{j}_{n,s} \rightarrow \int_{[0,\infty)} \mathrm e^{-rs}G^i(X_s)\, \mathrm d\Gamma^{j}_s \; \, \mathrm{a.s.}
\end{align*}
Using Assumption A1 and $x \in S^j$, we conclude by dominated convergence that
\begin{align*}
\kappa_n(x)=\EE_x\! \left[\int_{[0,\infty)} \mathrm e^{-rs}G^i(X_s)\, \mathrm d\Gamma^{j}_{n,s}\right] \to \EE_x \! \left[\int_{[0,\infty)} \mathrm e^{-rs}G^i(X_s)\, \mathrm d \Gamma^{j}_s\right]=G^i(x).
\end{align*}
It follows that $\kappa_n(x)\geq R^i(x)+\varepsilon$ for $n$ large enough. Let $\tau$ denote the exit time from $(x-\varepsilon,x+\varepsilon)$ and $\tau_x$ the hitting time of $x$. Using the Markov property and \eqref{mark} as in the proof of Proposition \ref{geneprop}(c), we deduce that
\begin{align}
&\EE_{x_n} \!\left[\int_{[0,\infty)} \mathrm e^{-rs}G^i(X_s)\, \mathrm d\Gamma^{j}_{n,s}\right] \notag
\\
& = \EE_{x_n} \! \left[\int_{[0,\tau_x\wedge \tau)} \mathrm e^{-rs}G^i(X_s)\, \mathrm d\Gamma^{j}_{n,s} +\indic_{\tau_x<\tau}\,\Lambda^{j}_{n,\tau_x-} \,\mathrm e^{-r \tau_x}\kappa_n(x)  +\indic_{\tau <\tau_x}\int_{[\tau,\infty)} \mathrm e^{-rs}G^i(X_s)\, \mathrm d\Gamma^{j}_{n,s} \right] \notag \allowdisplaybreaks
\\
&\geq [R^i(x)+\varepsilon]\,\EE_{x_n} [\indic_{\tau_x<\tau}\,\mathrm e^{-r \tau_x} ] +\EE_{x_n}\!\left[\indic_{\tau <\tau_x}\int_{[\tau, \infty)} \mathrm e^{-rs}G^i(X_s)\, \mathrm d\Gamma^{j}_{n,s} \right] \hskip -1mm. \label{armoire}
\end{align}
Again, as in the proof of Proposition \ref{geneprop}(c), there exists a constant $C'>0$ such that
\begin{align}
\left|\EE_{x_n}\! \left[\indic_{\tau <\tau_x}\int_{[\tau,\infty)} \mathrm e^{-rs}G^i(X_s)\, \mathrm d\Gamma^{j}_{n,s} \right]\right| \leq C' \PP_{x_n}[\tau <\tau_x]. \label{lit}
\end{align}
We deduce from (\ref{eq:ineqRG}) applied to $x_n$ and from \eqref{armoire}--(\ref{lit}) that
\begin{align*}
R^i(x_n)\geq \EE_{x_n}\!\left[\int_{[0,\infty)} \mathrm e^{-rs}G^i(X_s)\, \mathrm d\Gamma^{j}_{n,s}\right] \geq [R^i(x)+\varepsilon]\, \EE_{x_n} \! \left[\indic_{\tau_x<\tau}\mathrm e^{-r \tau_x} \right]-C'\PP_{x_n}[\tau <\tau_x].
\end{align*}
The right-hand side of this inequality converges to $R^i(x)+\varepsilon$ as $n \to \infty$, whereas the left-hand side converges to $R^i(x)$, a contradiction. We conclude that $S^i\cap S^j\cap \{R^i<G^i\}=\emptyset$.

\subparagraph{Step 2}

We now prove that $\mu^i(S^j\cap\{R^i<G^i\})=0$. Suppose, by way of contradiction, that $\mu^i(S^j\cap\{R^i<G^i\})>0$. Then there exists $x\in S^j\cap\{ R^i<G^i\}$ such that every neighborhood $O$ of $x$ is such that $\mu^i(O)>0$. Because $x\notin S^i$ by Step 1, it must be that $[x- \varepsilon,x+\varepsilon] \cap S^i =\emptyset$ for any sufficiently small $\varepsilon>0$, and thus $\mu^i([x-\varepsilon,x+\varepsilon])<\infty$.
Because $(\mu^i_n,S^i_n) = m_n^i \to m^i=(\mu^i,S^i)$, it follows that
\begin{align*}
\limsup_{n \to \infty} m^i_n([x-\varepsilon,x+\varepsilon]) \leq m^i([x-\varepsilon,x+\varepsilon])=\mu^i([x-\varepsilon,x+\varepsilon])< \infty,
\end{align*}
so that $S^i_n \cap [x-\varepsilon,x+\varepsilon] = \emptyset$ for $n$ large enough. We deduce that
\begin{align*}
\liminf_{n \to \infty} \mu^i_n((x-\varepsilon,x+\varepsilon))&=\liminf_{n \to \infty} m^i_n((x-\varepsilon,x+\varepsilon))
\\
&\geq m^i((x-\varepsilon ,x+\varepsilon))
\\
&=\mu^i((x-\varepsilon,x+\varepsilon))
\\
&>0.
\end{align*}
As this is true for any sufficiently small $\varepsilon >0$, there exists a sequence $x_n \to x$ such that for all $n$, $x_n \in \mathrm {supp} \, \mu^i_n$. Because $(\mu^i_n,S^i_n) \in \Phi^i(\mu^j_n,S^j_n)$, we have $R^i(x_n)=\bar J^i(x_n,(\mu^j_n,S^j_n))$. Thus, for any large enough $n$, inequality \eqref{eq:ineqRG} holds with $y=x_n$, which leads to a contradiction as in Step 2. We conclude that $\mu^i(S^j\cap\{R^i<G^i\})=0$.

\subparagraph{Step 3}

We finally prove that $(\mu^i,S^i)$ is a pbr to $(\mu^j,S^j)$. By Assumptions A1--A2, the random function $f:[0,\infty]^2 \rightarrow \RR$ defined by
\begin{align*}
f(t,t') := \indic_{t\leq t'} \, \mathrm e^{-rt}R^1(X_t)+ \indic_{t'<t} \,\mathrm e^{-rt'}G^1(X_{t'}),
\end{align*}
with $f(\infty,\infty):=0$, is a.s.\! bounded, and the set of discontinuities of $f$ is the set
\begin{align*}
\{ (t,t)\in [0,\infty)^2 : R^1(X_t)<G^1(X_t)\}.
\end{align*}
Letting $\nu^i_n$ and $\nu^i$ denote the probabilities on $[0,\infty]$ with csfs $\Lambda^{i}_{n}$ and $\Lambda^i$ respectively, and similarly for player $j$, we have
\begin{align}\label{eq:forbpr3}
\int_{[0,\infty]^2} f(t,t')\, \nu^i\otimes \nu^j(\mathrm d t,\mathrm d t')=\int_{[0,\infty)} \mathrm  e^{-r t} R^i(X_t)\Lambda^j_{t-} \, \mathrm d\Gamma^i_t+ \int_{[0,\infty)} \mathrm e^{-r t}G^i(X_t)\Lambda^i_t \, \mathrm d\Gamma^j_t.
\end{align}
Because $S^i\cap S^j\cap \{R^i<G^i\}=\emptyset$, the probability $\nu^1\otimes \nu^2$ does not charge the set of discontinuities of $f$; indeed, the conditional probability that $t'=t$ given $t$ is 0 unless $t=\tau_{S^j}$, the probability that $t=\tau_{S^j}$ is 0 unless $\tau_{S^i}=\tau_{S^j}$, and it cannot be that $\tau_{S^i}=\tau_{S^j}$ and $R^i(X_{\tau_{S^i}})<G^i(X_{\tau_{S^i}})$. Proposition \ref{prop:cvg_lambda} thus implies that the sequence $(\nu^i_n\otimes \nu^j_n)_{n \geq 0}$ converges weakly to $\nu^i\otimes \nu^j$. We can thus apply Theorem \ref{thm:bill} to obtain
\begin{align*}
\int_{[0,\infty]^2} f(t,t')\, \nu^{i}_n\otimes \nu^{j}_n(\mathrm dt,\mathrm dt') \rightarrow \int_{[0,\infty]^2} f(t,t')\, \nu^i\otimes \nu^j
(\mathrm dt,\mathrm dt')\;\, \text{a.s.}
\end{align*}
By Assumption A1, the random variables $(f(t,t'))_{(t,t') \in [0,\infty]^2}$ are uniformly integrable. As a result, the sequence of random variables $(\int _{[0,\infty]^2} f(t,t')\, \nu^{i}_n\otimes \nu^{j}_n(\mathrm dt,\mathrm dt'))_{n\geq 0}$ is also uniformly integrable. Therefore, the above convergence also holds in expectation, which leads by \eqref{eq:forbpr}--\eqref{eq:forbpr2} and \eqref{eq:forbpr3} to
\begin{align*}
J^i(x,\Gamma^{i}_{n},\Gamma^{j}_{n}) \rightarrow J^i(x,\Gamma^{i},\Gamma^{j}).
\end{align*}
Let now $\tau \in \mathcal T$ be an arbitrary stopping time for player $i$ such that $\tau\neq \tau_{S^j}$ on $R^i(X_\tau)<G^i(X_\tau)$ a.s. Replacing $\Lambda^{i}_{n}$ and $\Lambda^i$ in the preceding proof by $\Lambda_t=\indic_{t<\tau}$, we obtain
\begin{align*}
J^i(x,\tau,\Gamma^{j}_{n}) \to J^i(x,\tau,\Gamma^{j}).
\end{align*}
Because, for each $n \geq 0$, $(\mu^i_n,S^i_n) \in PBR^i(\mu^j_n,S^j_n)$, we have, for each $x\in \CI$,
\begin{align*}
J^i(x,\tau,\Gamma^j_n) \leq J^i(x,\Gamma^{i}_{n},\Gamma^{j}_{n}).
\end{align*}
Taking the limit on both sides, it follows that, for each $\tau \in \mathcal T$ such that $\tau\neq \tau_{S^j}$ a.s.\! and for each $x\in \CI$,
\begin{align}
J^i(x,\tau,\Gamma^{j}) \leq J^i(x,\Gamma^{i},\Gamma^{j}). \label{this is the end}
\end{align}
To conclude, notice that a stopping time $\tau$ such that $\tau=\tau_{S^j}$ and $R^i(X_\tau)<G^i(X_\tau)$ with positive probability cannot be optimal, as player $1$ would prefer, conditionally on this event, to wait indefinitely so as to let player $j$ stop first. Hence, \eqref{this is the end} holds without restriction for all $\tau \in \mathcal T$. We conclude that $(\mu^i,S^i) \in PBR^i(\mu^j,S^j)$ and therefore that the graph of $\Phi^i$ is closed. Hence the result. \end{proof}

{Proposition \ref{prop:closed_graph} has an important corollary, which pertains to pure pbrs as defined by \eqref{purepbrcorr}. As explained in the introduction, the optimal stopping problem of a decision maker $i$ whose reward function may jump upwards according to a Markovian randomized stopping time $m^j$ is of independent interest. According to Proposition \ref{prop:existence_pbr}, this problem admits an optimal stopping time and, according to Lemma \ref{lem:sigma_i}, the corresponding optimal stopping regions forms a lattice. We now study how this set varies with $m^j$.

To this end, we may identify the family $\CC$ of all closed subsets $S \subset \CI$ with a subset of $\CM(\CI)$ by identifying $S$ with the measure $m_S$ associated to the pair $(0,S)$. Recall that a sequence $(S_n)_{n \geq 0}$ of closed sets converges to $S$ in the Painlev\'{e}--Kuratowski sense ($S_n \mathop{\to} \limits_{\rm{PK}} S$ hereafter) if:
\begin{itemize}

\item[(i)]

For each $x\in S$ and every open neighborhood $O$ of $x$, $O \cap S_n \neq \emptyset$ for any sufficiently large $n$;

\item[(ii)]

For each $x\notin S$, there exists an open neighborhood $O$ of $x$ such that $O \cap S_n = \emptyset$ for any sufficiently large $n$.

\end{itemize}
(See, e.g., \cite[Section 5.2]{Beer}). The following result then holds.

\begin{lemma} \label{lastlemma}
$\CC$ is closed in $\CM(\CI)$ and $S_n \to S$ in $\CC$ if and only if $S_n \mathop{\to} \limits_{\rm{PK}} S$.
\end{lemma}
\begin{proof}
The proof relies on the characterization given in Proposition \ref{prop:compact_metric}(3).

Assume that $m_{S_n} \to m$ in $\CM(\CI)$. Then $\int_{\CI} \phi \, \mathrm dm_{S_n}\in\{0,\infty\}$ for all $\phi \in \CC_c^+(\CI)$, and hence $m$ must vanish on $e(m)^c$. It follows that $m=m_{e(m)} \in\CC$.

Assume that $S_n\to S$ in $\CC$. First, for every open set $O$ such that $O\cap S \neq \emptyset$, we must have $O \cap S_n \neq \emptyset$ for any sufficiently large $n$ as $m_{S_n} (O) \to \infty$. Second, for each $x\notin S$, one can find $\phi \in \CC_c^+(\mathcal I \setminus S)$ which is positive on some open neighborhood $O$ of $x$. Because $\int_\CI \phi \, \mathrm dm_{S_n} \to \int_\CI \phi \, \mathrm dm_{S}=0$, it must be that $O \cap S_n = \emptyset$ for any sufficiently large $n$. We deduce from this that $S_n \mathop{\to} \limits_{\rm{PK}} S$. The proof of the converse is similar and is thus omitted. The result follows.
\end{proof}

Thanks to Lemma \ref{lastlemma}, we deduce from Proposition \ref{prop:closed_graph} a stability result for pure pbrs.

\begin{corollary}\label{breakthrough}
The correspondence of pure pbrs of player $i,$ defined by
\begin{align*}
\CM(\CI)\twoheadrightarrow \CM(\CI): m^j \mapsto \Sigma^i(m^j):=\{ S\in \CC : m_S\in PBR^i(m^j) \},
\end{align*}
has a compact graph and hence is upper hemicontinuous.
\end{corollary}

\begin{proof} Notice that any pure pbr of player $i$ against a measure $m^j \in \CM(\CI)$ belongs to $\Phi^i(m^j)$ by Proposition \ref{geneprop}(ii). We deduce that $\Sigma^i(m^j)= \CC \cap \Phi^i(m^j)$. As $\CC$ is closed, we can then apply Proposition \ref{prop:closed_graph}. The result follows. \end{proof}

\begin{remark}
Given a sequence $(m^j_n)_{n \geq 0}$ with limit $m^j$ in $\CM(\CI),$ let $\overline S\,\!^i_n$ and $\overline S\,\!^i$ denote the largest closed sets in $\Sigma^i(m^j_n)$ and $\Sigma^i(m^j),$ respectively$,$ as in Lemma \ref{lem:sigma_i}. The sequence $(\overline S\,\!^i_n)_{n \geq 0}$ admits a convergent subsequence with limit $S\in \Sigma^i(m^j),$ but we can only conclude in general that $S \subset \overline S\,\!^i$---unless $\Sigma^i(m^j)$ is a singleton$,$ in which case we have an equality.
\end{remark}}

\section{Contractibility and the AR Property}\label{sec:contractibility}

The aim of this section is twofold. First, we establish in Proposition \ref{prop:contraction} that the space $(\CM(\CI),\vartheta)$ is contractible, which, together with Proposition \ref{prop:compact_metric}, completes the proof of Theorem \ref{thm:compact_AR}. Second, we establish in Proposition \ref{prop:contractible_values} that the correspondence $\Phi$ defined in \eqref{defPhi} has contractible values, which together with the results of the preceding sections completes the proof of Theorem \ref{thm:main}. Both proofs rely on an explicit construction of a contraction of the space $\CM(\CI)$ using convolutions.

Let us first introduce more tools from general topology of metric spaces.

\begin{definition}\label{def:ANR}
A metric space $(E,d)$ is an absolute neighborhood retract (ANR) if$,$ for any continuous map $f:E \rightarrow E'$ into a metric space $(E',d')$ such that $f$ is an homeomorphism between $E$ and $f(E)$ and $f(E)$ is closed in $E',$ there exists an open set $U$ such that $f(E)\subset U$ and a continuous map $g: U \rightarrow f(E)$ such that for all $x\in f(E),$ $g(x)=x$ (i.e.$,$ $f(E)$ is a retract of some neighborhood $U$).
\end{definition}

From Definition \ref{def:AR}, it is clear that an AR is an ANR, and we have the following characterization of ARs.

\begin{proposition}[{\cite[Theorem 8.2]{McLennan}}]\label{prop:ar=contractible_anr}
$(E,d)$ is an AR if and only if it is a contractible ANR.
\end{proposition}

This equivalence is useful as there are sufficient conditions for a metric space to be an ANR. The first one can be stated as follows.

\begin{proposition}[{\cite[Proposition 8.3]{McLennan}}]\label{prop:convex_anr}
A metrizable convex subset of a locally convex Hausdorff topological vector space is an ANR.
\end{proposition}

The second sufficient condition we will use states that the closure (in the strong sense of homopotopy-denseness defined below) of an ANR is still an ANR.

\begin{definition}
Let $(E,d)$ be a metric space and $A\subset E$. $A$ is said to be homotopy-dense in $E$ if there exists a continuous map $H:E\times [0,1] \rightarrow E$ such that $H(\cdot,0)={\rm Id}_E$ and $H(E\times (0,1])\subset A$.
\end{definition}

\begin{proposition}[{\cite[Corollary 6.6.7]{Sakai}}] \label{prop:sakai}
Let $(E,d)$ be a metric space and $A$ a homotopy-dense subset of $E$. Then $E$ is an ANR if and only if $A$ is an ANR.
\end{proposition}

The first main result of this section can be stated as follows.

\begin{proposition}\label{prop:contraction}
There exists a continuous map $H: \CM(\CI)\times [0,1] \rightarrow \CM(\CI)$ such that
\begin{enumerate}

\item

for each $m \in \CM(\CI),$ $H(m,0)=m$ and $H(m,1)=0;$

\item

for all $\varepsilon \in (0,1]$ and $m \in \CM(\CI),$ $H(m,\varepsilon) \in \CM_{\rm{loc}}(\CI)$.
\end{enumerate}
\end{proposition}

Before proving Proposition \ref{prop:contraction}, let us first show how this result enables us to complete the proof of Theorem \ref{thm:compact_AR}.

\begin{proof}[Proof of Theorem \ref{thm:compact_AR}]
By Proposition \ref{prop:compact_metric}, $\CM(\CI)$ is compact and metrizable. Proposition \ref{prop:contraction} implies that $\CM(\CI)$ is contractible and that $\CM_{\rm{loc}}(\CI)$ is homotopy-dense in $\CM(\CI)$. The topology induced by $\vartheta$ on $\CM_{\rm{loc}}(\CI)$ coincides with the topology of vague convergence by Proposition \ref{prop:compact_metric}(3), and thus $\CM_{\rm{loc}}(\CI)$ can be identified with a convex subset of the vector space of linear functionals on $\mathcal C^+_c(\CI)$ endowed with the vague topology. Therefore, $\CM_{\rm{loc}}(\CI)$ is a convex subset of a locally convex Hausdorff topological vector space and is metrizable. We deduce from Proposition \ref{prop:convex_anr} that $\CM_{\rm{loc}}(\CI)$ is an ANR.  Because $\CM_{\rm{loc}}(\CI)$ is homotopy-dense in $\CM(\CI)$ and an ANR, we conclude by Proposition \ref{prop:sakai} that $\CM(\CI)$ is an ANR.  As $\CM(\CI)$ is also contractible, we conclude from Proposition \ref{prop:ar=contractible_anr} that it is an AR. Hence the result.
\end{proof}

Let us come back to the proof of  Proposition \ref{prop:contraction}.

\begin{proof}[Proof of Proposition \ref{prop:contraction}]
The proof consists of two steps.

\subparagraph{Step 1}

We first show that it is sufficient to prove the result for $\CI=\RR$. Let $\psi: \CI \rightarrow \RR$ denote a $\mathcal C^1$-diffeomorphism and assume that a function $H$ satisfying properties $1$ and $2$ exists with $\CI=\RR$. Define
\begin{align*}
\widehat H(m,\varepsilon):= H(m\circ \psi^{-1},\varepsilon ) \circ \psi, \quad(m,\varepsilon )\in \CM(\CI)\times [0,1],
\end{align*}
where $m \circ \psi^{-1}$ denotes the image of the measure $m$ by $\psi$, defined for each $B\in \CB(\RR)$ by $m \circ \psi^{-1}(B):= m(\psi^{-1}(B) )$, and $\nu \circ \psi$ the image of the measure $\nu$ by $\psi^{-1}$, so that
\begin{align*}
\forall (m,\varepsilon ,A) \in \CM(\CI)\times [0,1] \times \CB(\CI), \; \widehat H(m,t)(A) = H(m\circ \psi^{-1},\varepsilon ) (\psi(A)).
\end{align*}
We just need to check that $\widehat H$ is continuous and satisfies Properties $1$ and $2$ in Proposition \ref{prop:contraction}. Property $1$ is immediate and Property $2$ follows from the fact that $\psi$ preserves compact sets. To prove continuity, it is sufficient to prove that the mappings \begin{eqnarray*}
\CM(\CI) \to \CM(\RR) : m \mapsto m\circ \psi^{-1}\, \text{ and } \,\CM(\RR) \to \CM(\CI): \nu \mapsto \nu \circ \psi
\end{eqnarray*}
are continuous. As the arguments for the two mappings are similar, we only consider the first one. Thus suppose that $m_n \rightarrow m $ in $\CM(\CI)$. If $O$ is open in $\RR$, then $\psi^{-1}(O)$ is open in $\CI$ and
\begin{align*}
\liminf_{n \to \infty} m_n\circ \psi^{-1} (O) = \liminf_{n \to \infty} m_n (\psi^{-1} (O)) \geq m (\psi^{-1} (O))=m\circ \psi^{-1}(O).
\end{align*}
Similarly, if $K$ is compact in $\RR$, then $\psi^{-1}(F)$ is compact in $\CI$ and
\begin{align*}
\limsup_{n \to \infty} m_n\circ \psi^{-1} (K) = \limsup_{n \to \infty} m_n (\psi^{-1} (K)) \leq m (\psi^{-1} (K))=m\circ \psi^{-1}(K).
\end{align*}
We conclude that $m_n\circ \psi^{-1} \rightarrow m\circ \psi^{-1} $ in $\CM(\RR)$, as desired. Notice for later use that, by choosing $\psi$ as a $\mathcal C^1$-diffeomorphism, we ensure that, if, for some $\varepsilon \in [0,1]$, $H(m\circ \psi^{-1},\varepsilon )$ is absolutely continuous with respect to Lebesgue measure, then so is $\widehat H(m, \varepsilon )$.

\paragraph{Step 2}

We now prove the result for $\CI=\RR$. Define, for all $(m,\varepsilon,x)\in \CM(\RR)\times (0,1] \times \RR$,
\begin{align}\label{eq:defh}
h(m,\varepsilon,x):=  \min \left\{ \int_\RR \rho_{\varepsilon}(x-y)\,m(\mathrm d y) , \frac{1}{\varepsilon^2} \right\}\!,
\end{align}
where $\rho_\varepsilon$ is a continuous function with compact support $[-\varepsilon,\varepsilon]$ defined by
\begin{align*}
\rho_\varepsilon(x)=\frac{1}{\varepsilon} \hskip 0.3mm \max\left \{1- \frac{|x|}{\varepsilon},0\right\} \! , \quad x \in \RR.
\end{align*}
One can easily verify the following properties for all $\varepsilon \in (0,1]$ and $c\in (0,1)$:
\begin{align}\label{eq:prop_rho}
0 \leq \rho_\varepsilon \leq \frac{1}{\varepsilon} \, \indic_{(-\varepsilon,\varepsilon)}, \;\; \int_\RR \rho_\varepsilon(x) \, \mathrm dx=1, \;\;\lim_{\delta \rightarrow \varepsilon} \| \rho_\delta- \rho_\varepsilon \|_\infty =0, \;\; \rho_\varepsilon  \geq \frac{(1-c)}{\varepsilon} \,\indic_{(-c \varepsilon,c\varepsilon)} .
\end{align}
Define then, for all $(m,\varepsilon)\in \CM(\RR) \times (0,1]$,
\begin{align*}
H(m,\varepsilon):= (1-\varepsilon)h(m,\varepsilon,.)\cdot \lambda,
\end{align*}
where $\lambda$ denotes Lebesgue measure, and $H(m,0): =m$. Notice that, for all $(m,\varepsilon)\in \CM(\RR) \times (0,1]$, the measure $H(m,\varepsilon)$ is absolutely continuous with respect to Lebesgue measure and has a bounded density, so that $H(m,\varepsilon)\in \CM_{\rm{loc}}(\RR)$.

Because $H(\cdot,0)= \mathrm{Id} _{\CM(\RR)}$ and $H(\cdot,1)= 0$, we only need to check that $H$ is jointly continuous on $\CM(\RR)\times [0,1]$. Thus consider a sequence $(\varepsilon_n,m_n) \to (\varepsilon,m)$ in $\CM(\RR)\times [0,1]$.

\subparagraph{Case 1: $\varepsilon>0$}

With no loss of generality, we can assume that $\varepsilon_n>0$ for all $n \geq 0$. Denoting $d(.,C)$ the usual distance to a set $C$ in $\RR$, define the sets
\begin{align*}
E^+(m,\varepsilon):=\{ x \in \RR : d(x,e(m))>\varepsilon \},
\\
E^-(m, \varepsilon):=\{ x \in \RR : d(x,e(m))<\varepsilon \},
\\
E^0(m,\varepsilon):=\{ x \in \RR : d(x,e(m))=\varepsilon \}.
\end{align*}
We examine these three sets separately in cases numbered 1.1, 1.2, 1.3, respectively.

\subsubparagraph{Case 1.1}

For each $x\in E^+(m,\varepsilon)$, we have $d(x,e(m))>\varepsilon$, so that $\rho_\varepsilon(x-\cdot) \in \mathcal C_c^+(\RR\setminus e(m))$. It then follows from Proposition \ref{prop:compact_metric}(3) that
\begin{align*}
\int_\RR \rho_\varepsilon (x-y) \,m_n(\mathrm dy) \to  \int_\RR \rho_\varepsilon (x-y)\,m(\mathrm dy) <\infty.
\end{align*}
Moreover, $\rho_{\varepsilon_n}(x-\cdot) \to \rho_\varepsilon(x-\cdot)$ uniformly, and, any sufficiently large $n$, all the supports of these functions are contained in a compact subset $K$ of $\RR\setminus e(m)$. As $\limsup_{n\to \infty} m_n(K) \leq m(K)<\infty$ and $\varepsilon>0$, we have, by \eqref{eq:prop_rho},
\begin{align*}
\int_\RR | \rho_{\varepsilon_n}(x-y)-\rho_{\varepsilon}(x-y)| \,m_n(\mathrm dy) \leq \|\rho_{\varepsilon_n}-\rho_\varepsilon \|_\infty\, m_n(K) \to 0.
\end{align*}
This implies
\begin{align*}
\int_\RR \rho_{\varepsilon_n} (x-y) \,m_n(\mathrm dy) \to \int_\RR \rho_\varepsilon (x-y) \, m(\mathrm dy)
\end{align*}
and therefore, by \eqref{eq:defh},
\begin{align*}
h(m_n,\varepsilon_n,x) \to  h(m,\varepsilon,x),
\end{align*}
as desired.

\subsubparagraph{Case 1.2}

The set $E^-(m,\varepsilon)= e(m)+(-\varepsilon,\varepsilon)$ is open. Let $x\in E^-(m,\varepsilon)$ and $z \in e(m)$ such that $|x-z|< \varepsilon$. Letting $c \in (\frac{|x-z|}{\varepsilon},1)$, we have, by \eqref{eq:prop_rho},
\begin{align*}
\rho_{\varepsilon}(x-\cdot) \geq \frac{(1-c)}{\varepsilon}\,\indic_{(x-c\varepsilon,x+c\varepsilon)}.
\end{align*}
For any sufficiently large $n$, $\|\rho_\varepsilon-\rho_{\varepsilon_n}\|_\infty \leq \frac{c}{\varepsilon}$, and thus
\begin{align*}
\rho_{\varepsilon_n}(x-\cdot) \geq \frac{(1-2c)}{\varepsilon}\, \indic_{(x-c\varepsilon,x+c\varepsilon)}.
\end{align*}
Using that $z\in (x-c\varepsilon,x+c\varepsilon) \cap e(m)$, we have $\lim_{n \to \infty} m_n((x-c\varepsilon,x+c\varepsilon))=\infty$ and thus
\begin{align*}
\int_\RR \rho_{\varepsilon_n}(x-y) \,m_n(\mathrm d y) \geq \frac{1-2c}{\varepsilon} \,m_n((x-c\varepsilon,x+c\varepsilon)) \to \infty.
\end{align*}
It follows that, for any sufficiently large $n$, $h(m,\varepsilon_n,x)=\frac{1}{\varepsilon_n^2}$ by \eqref{eq:defh}. On the other hand, $ \int_\RR \rho_\varepsilon (x-y) \,m(\mathrm dy) =\infty$, so that $h(m,\varepsilon,x)=\frac{1}{\varepsilon^2}$ by \eqref{eq:defh} again, and we conclude that
\begin{align*}
h(m_n,\varepsilon_n,x) \to  h(m,\varepsilon,x),
\end{align*}
as desired.

\subsubparagraph{Case 1.3}

Observe first that the set $E^0(m,\varepsilon)$ is countable and thus has Lebesgue measure 0; indeed, $e(m)^c$ is open and thus is a countable union of disjoint open intervals, each of which contains at most two points in $E^0(m,\varepsilon)$. Given $\phi \in \mathcal C^+_c(\CI)$, thanks to the analysis of Cases 1.1 and 1.2, we can apply the bounded convergence theorem to deduce that
\begin{align*}
L_\phi(H(m_n,\varepsilon_n))=\int_\RR \phi(x)h(m_n,\varepsilon_n,x)\, \mathrm dx \to \int_\RR \phi(x)h(m,\varepsilon,x)\, \mathrm dx= L_{\phi}(H(m, \varepsilon)).
\end{align*}
Because $e(H(m,\varepsilon))=\emptyset$, we conclude by Proposition \ref{prop:compact_metric}(3) that
\begin{align*}
H(m_n,\varepsilon_n) \to H(m,\varepsilon),
\end{align*}
as desired.

\subparagraph{Case 2: $\varepsilon=0$}

For each $n$ such that $\varepsilon_n=0$, we have $H(m_n,0)=m_n$, so we may assume with no loss of generality that $\varepsilon_n >0$ for all $n$.
To prove that $H(m_n,\varepsilon_n)\rightarrow H(m,0)=m$, we use Proposition \ref{prop:compact_metric}(3). We check each property in turn.

\subsubparagraph{Property 1}

Given an open subset $O$ of $\RR$ such that $O \cap e(m)\neq \emptyset$, we have to prove that $ \lim_{n \to \infty} H(m_n,\varepsilon_n)( O) =\infty$. Let us fix some $c\in (0,1)$, and let $z \in O \cap e(m)$ and $\delta>0$ be such that $(z-\delta,z+\delta) \subset O$. There exists some $n_0$ such that, for each $n\geq n_0$, $(y-c\varepsilon_n,y+c\varepsilon_n) \subset O$ for all $y\in (z-\delta,z+\delta)$. By \eqref{eq:defh}--\eqref{eq:prop_rho}, we have
\begin{align}
H(m_n,\varepsilon_n)(O)&= (1-\varepsilon_n) \int_\RR \indic_O(x)h(m_n,\varepsilon _n,x) \, \mathrm dx \nonumber
\\
&= (1-\varepsilon_n) \int \indic_O(x)\min\left\{\int_\RR \rho_{\varepsilon_n}(x-y)dm_n(y), \frac{1}{\varepsilon_n^2}\right\} \hskip 0.2mm \mathrm dx \nonumber
\\
&\geq (1-\varepsilon_n) \int_\RR \indic_O(x)\min\left\{\frac{1-c}{\varepsilon_n}\,m_n((x-c\varepsilon_n,x+c\varepsilon_n)), \frac{1}{\varepsilon_n^2}    \right\} \hskip 0.2mm \mathrm dx. \label{eq:ineq_Hlambda}
\end{align}
In turn, letting $A_n:=\{ x\in O :  (1-c)m_n((x-c\varepsilon_n,x+c\varepsilon_n))<\frac{1}{\varepsilon_n} \}$, we have
\begin{align}
(1-\varepsilon_n) &\int_\RR \indic_O(x)\min\left\{\frac{1-c}{\varepsilon_n}\,m_n((x-c\varepsilon_n,x+c\varepsilon_n)), \frac{1}{\varepsilon_n^2}    \right\} \hskip 0.2mm \mathrm dx \notag
\\
&= \frac{1-\varepsilon_n}{\varepsilon_n} \int_\RR \indic_O(x)\min\left\{(1-c)m_n((x-c\varepsilon_n,x+c\varepsilon_n)), \frac{1}{\varepsilon_n}    \right\} \hskip 0.2mm \mathrm dx  \notag
\\
&= \frac{1-\varepsilon_n}{\varepsilon_n} \int_\RR  \left[ \indic_{A_n}(x)(1-c)m_n((x-c\varepsilon_n,x+c\varepsilon_n)) + \indic_{O\setminus A_n}(x) \frac{1}{\varepsilon_n}    \right] \mathrm dx  \notag
\\
&=\frac{1-\varepsilon_n}{\varepsilon_n}  \left[ \int_\RR \int_\RR \indic_{A_n}(x)(1-c)\indic_{(x-c\varepsilon_n,x+c\varepsilon_n)}(y) \,m_n(\mathrm dy)\, \mathrm dx + \frac{1}{\varepsilon_n}\,\lambda(O\setminus A_n) \right] \notag
\\
&=\frac{1-\varepsilon_n}{\varepsilon_n}  \left[ \int_\RR (1-c)\lambda(A_n \cap(y-c\varepsilon_n,y+c\varepsilon_n)) \, m_n(\mathrm dy) + \frac{1} {\varepsilon_n}\,\lambda(O\setminus A_n) \right]\hskip -1mm, \label{meremichel}
\end{align}
where the last equality follows from Fubini's theorem. Suppose, by way of contradiction, that, along some subsequence, $H(m_n,\varepsilon_n) (O)$ is bounded above by a constant $M$. Then, from \eqref{eq:ineq_Hlambda}--\eqref{meremichel}, it must be that $\frac{1-\varepsilon_n}{ \varepsilon _n^2}\, \lambda (O\setminus A_n) \leq  M$. Because $(y-c\varepsilon_n,y+c\varepsilon_n) \subset O$, we deduce that, for each $n \geq n_0$ in the subsequence and each $y\in(z-\delta,z+\delta)$,
\begin{align*}
\lambda(A_n \cap(y-c\varepsilon_n,y+c\varepsilon_n)) &= \lambda((y-c\varepsilon_n,y+c\varepsilon_n)) - \lambda((O\setminus A_n)\cap(y-c \varepsilon_n , y+c\varepsilon_n))
\\
& \geq 2c\varepsilon_n -M\, \frac{\varepsilon_n^2}{1-\varepsilon_n}
\\
&\geq c\varepsilon_n
\end{align*}
for any sufficiently large $n$. It follows from \eqref{eq:ineq_Hlambda}--\eqref{meremichel} that, along this subsequence,
\begin{align*}
H(m_n,\varepsilon_n)(O)& \geq (1-\varepsilon_n)c(1-c) m_n((z-\delta,z+\delta)) \to \infty,
\end{align*}
a contradiction. We conclude that $H(m_n,\varepsilon_n)(O) \to \infty$, as desired.

\subsubparagraph{Property 2}

Given $\phi \in \mathcal C^+_c(\RR\setminus e(m))$, we have to prove that $L_\phi(H(m_n,\varepsilon_n)) \to L_\phi(m).$ Let $K:=\mathrm{supp} \, \phi$ and, for each $x \in \mathbb R$, $K_x:=K+[-x,x]$. Because $K$ is a compact subset of the open set $\RR\setminus e(m)$, there exists $\delta>0$ such that $K_\delta \cap e(m)=\emptyset$. Because $\lim_{n \to \infty} m_n = m$, we have $\limsup_{n \to \infty} m_n(K_\delta) <\infty$, and hence $\limsup_{n\to \infty} m_n(K_{\varepsilon_n}) <\infty$. Moreover,
\begin{align*}
H(m_n,\varepsilon_n)(K) &\leq \int_K \int_\RR \rho_{\varepsilon_n}(x-y)\, m_n(\mathrm dy) \, \mathrm dx
\\
&=  \int_\RR \int_K \rho_{\varepsilon_n}(x-y) \, \mathrm dx \, m_n(\mathrm dy)
\\
& \leq \int_{\RR} \indic_{K_{\varepsilon_n}}(y)  \int_\RR \rho_{\varepsilon_n}(x-y) \, \mathrm dx \, m_n(\mathrm dy)
\\
&= m_n(K_{\varepsilon_n}),
\end{align*}
where the first inequality follows from \eqref{eq:defh}, the first equality follows from Fubini's theorem, the second inequality follows from the definition of $\rho_\varepsilon$, and the last equality follows from \eqref{eq:prop_rho}. We deduce from this that $\limsup_{n \to \infty} H(m_n, \varepsilon_n)(K) <\infty$. Now, recall that
\begin{align*}
L_{\phi}(H(m_n,\varepsilon_n))&= \int_\RR  \phi(x) h(m_n,\varepsilon_n,x)\, \mathrm dx  = \int_\RR  \phi(x) \min\left\{ \int_\RR \rho_{\varepsilon_n}(x-y)\,m_n(\mathrm dy), \frac{1}{\varepsilon_n^2} \right\}  \mathrm dx.
\end{align*}
By \eqref{eq:prop_rho}, for sufficiently large $n$,
\begin{align*}
\int_\RR \rho_{\varepsilon_n}(x-y)\, m_n(\mathrm dy) \leq \frac{1}{\varepsilon_n} \,m_n((x-\varepsilon_n,x+\varepsilon_n)) \leq \frac{1}{ \varepsilon_n} \, m_n(K_\delta) \leq \frac{1}{\varepsilon_n^2}
\end{align*}
for all $x\in K$, where the last inequality follows from $\limsup_{n \to \infty} m_n(K_\delta)<\infty$. Therefore, for any sufficiently large $n$, we have, by Fubini's theorem,
\begin{align*}
L_{\phi}(H(m_n,\varepsilon_n))&= \int_\RR  \phi(x) \int_\RR \rho_{\varepsilon_n}(x-y) \,m_n(\mathrm dy)\, \mathrm dx
\\
&= \int_\RR   \int_\RR  \phi(x) \rho_{\varepsilon_n}(x-y)\, \mathrm dx \, m_n(\mathrm dy)  \allowdisplaybreaks
\\
&= \int_\RR   \int_\RR  \phi(u+y) \rho_{\varepsilon_n}(u)\, \mathrm du \, m_n(\mathrm d y).
\end{align*}
Using that $\int_\RR \rho_{\varepsilon_n}(u)\, \mathrm du=1$, we obtain
\begin{align}
|L_{\phi}(H(m_n,\varepsilon_n))-L_\phi(m_n)|&= \left| \int_\RR   \int_\RR  [\phi(u+y)-\phi(y)] \rho_{\varepsilon_n}(u)\, \mathrm du \,m_n(\mathrm dy) \right|  \notag
\\
&\leq  \int_\RR   \int_\RR  |\phi(u+y)-\phi(y)| \rho_{\varepsilon_n}(u)\, \mathrm du \, m_n(\mathrm d y) \notag
\\
&= \int_{K_{\varepsilon_n}}\int_\RR  |\phi(u+y)-\phi(y)| \rho_{\varepsilon_n}(u)\, \mathrm du \, m_n(\mathrm dy)  \notag
\\
& \leq \omega_\phi(\varepsilon_n)m_n(K_{\varepsilon_n}) \notag
\\
& \to 0, \label{Kperdusonchat}
\end{align}
where, recalling that $\phi \in \CC^+_c(\CI\setminus e(m))$ and thus is uniformly continuous, $\omega_\phi$ denotes the modulus of continuity of $\phi$. As $m_n \to m$ and $\mathrm{supp}\, \phi\cap e(m)=\emptyset$, we have $L_\phi(m_n) \rightarrow L_\phi(m)$ by Proposition \ref{prop:compact_metric}(3). From this and \eqref{Kperdusonchat}, we conclude that
\begin{align*}
L_\phi(H(m_n,\varepsilon_n)) \to  L_\phi(m),
\end{align*}
as desired. Hence the result.
\end{proof}

Because the correspondence $\Phi$ defined by \eqref{eq:def_phi_product} and characterized by \eqref{egalPhi} has nonempty values by Proposition \ref{prop:existence_pbr} and a closed graph by Proposition \ref{prop:closed_graph}, and because $\CM(\CI) \times \CM(\CI)$ is a compact AR as the product of two compact ARs \cite[Exercise 8.4]{McLennan}, the following result enables us to apply Theorem \ref{thm:fixed_point} to $\Phi$ and thereby to complete the proof of Theorem \ref{thm:main}.

\begin{proposition}\label{prop:contractible_values}
The correspondence $\Phi$ has contractible values.
\end{proposition}

\begin{proof}
Recall that $\Phi$ take values in $\CM(\CI)\times \CM(\CI)$ and that, denoting $m^i \in \CM(\CI)$ the measure associated to the pair $(\mu^i,S^i)$, we have
\begin{align*}
\forall (m^1,m^2) \in \CM(\CI)\times \CM(\CI), \; \Phi(m^1,m^2)=\Phi^1(m^2)\times \Phi^2(m^1).
\end{align*}
Because the product of two contractible spaces is contractible, it is therefore sufficient to prove that, for all $i=1,2$ and $m^j \in \CM(\CI)$, $\Phi^i(m^j)$ is contractible in $\CM(\CI)$. The measure $m^j$ being fixed, we have, by \eqref{defPhi},
\begin{align*}
\Phi^i(m^j)=\{ m \in \CM(\CI) : \underline S^i \subset e(m) \subset \overline S\,\!^i, \, m(\CI\setminus \overline S\!\,^i)=0\},
\end{align*}
where $\underline S^i$ and $\overline S\,\! ^i$ are the closed subsets of $\CI$ defined in Section \ref{sec:BR}. The open set $O=\CI\setminus \underline S^i$ can be written as a countable union $O=\bigcup_{k\geq 0} O_k$ of disjoint open intervals $O_k \subset \CI$. For each $k\geq 0$, let $F_k : =\overline S\,\!^i \cap O_k$, which is closed in $O_k$. Every measure $m\in \Phi^i(m^j)$ can in turn be written as
\begin{align}\label{eq:decomposition_m_1}
 m= \overline m + \sum_{k \geq 0} m_k,
\end{align}
where $\overline m$ is defined by
\begin{align*}
\overline m (A)= \begin{cases} \infty & \text{if } A\cap \underline S^i \neq \emptyset \\ 0 & \text{if } A\cap \underline S^i = \emptyset \end{cases}, \quad A \in \CB(\CI),
\end{align*}
and, for each $k\geq 0$, $m_k$ is the restriction of $m$ to $O_k$ that we identify with a (not necessarily regular) measure on $\CI$ through the formula
\begin{align*}
m_k(A)= m(A\cap O_k), \quad A \in \CB(\CI).
\end{align*}
Reciprocally, given any sequence $(m_k)_{k \geq 0}$ of regular measures on $O_k$, formula \eqref{eq:decomposition_m_1} defines a regular measure on $\CI$.

The proof consists of two steps. In Step 1, we prove that the contraction $H$ constructed in Proposition \ref{prop:contraction} can be modified to obtain, for each $k \geq 0$, a contraction of the set of measures in $\CM(O_k)$ concentrated on $F_k$. To do so, we simply compose, up to a diffeomorphism, $H$ with the projection on $F_k$. In Step 2, we paste together a family of such contractions using \eqref{eq:decomposition_m_1} to obtain a contraction of $\Phi^i(m^j)$.

\subparagraph{Step 1}

We first prove that, for each $k\geq 0$, the set
\begin{align*}
C_k:= \{ m \in \CM(O_k) : m(O_k\setminus F_k)=0\}
\end{align*}
is contractible for the topology induced by $\vartheta$ on $\CM(O_k)$. Notice that $C_k$ is closed in $\CM(O_k)$ as the mapping $m \mapsto  m(O_k\setminus F_k)$ is lsc and nonnegative. Fix some $k\geq 0$ and let $\psi_k : O_k \rightarrow \RR$ be a $\mathcal C^1$-diffeomorphism. Then the map $\widehat \psi_k : \CM(O_k) \to \CM(\RR)$ defined by
\begin{align*}
\widehat \psi_k(m)=m\circ \psi_k^{-1}, \quad m \in \CM(O_k),
\end{align*}
is a homeomorphism (see the proof of Proposition \ref{prop:contraction}). Letting $\widehat F_k :=\psi_k(F_k)$ and $\widehat C_k:= \widehat \psi_k(C_k)$, consider then the map $p_k : \RR \rightarrow \widehat F_k$ defined by
\begin{align*}
p_k(x)= \max \left\{ y \in \widehat F_k :  |x-y|= \inf_{z \in \widehat F_k}|x-z| \right\}\! , \quad x\in \RR,
\end{align*}
which is a right-continuous version of the orthogonal projection on $\widehat F_k$. The map $p_k$ is nondecreasing and continuous outside of an at most countable set $D_k$. It is easy to verify that for any interval $U \subset \RR$, $p_k^{-1}(U)$ is an interval and that
\begin{align*}
p_k^{-1}(U)=p_k^{-1}(U\cap \widehat F_k) \, \text{ and } \, p_k^{-1}(U) \cap \widehat F_k =U\cap \widehat F_k.
\end{align*}
Define a map $\widehat H_k: \CM(\RR)\times [0,1] \to \CM(\RR)$  by
\begin{align*}
\widehat H_k(m,\varepsilon)=H(m,\varepsilon) \circ p_k^{-1}, \quad (m,\varepsilon)\in \CM(\RR)\times [0,1],
\end{align*}
where $H : \CM(\RR)\times [0,1] \to \CM(\RR)$ denotes the contraction of $\CM(\RR)$ explicitly constructed in Step 2 of the proof of Proposition \ref{prop:contraction}. By construction, $\widehat H_k(m,0)= m \circ p_k^{-1}$ and $\widehat H_k(m,1)=0$ for all $m\in \CM(\RR)$. In particular, $\widehat H_k(m,0)=m$ for all $m \in \widehat C_k$. Therefore, to prove that $\widehat H_k$ is a contraction of $\widehat C_k$, we only need to check that it is jointly continuous on $\widehat C_k\times [0,1]$ and takes values in $ \widehat C_k$. Thus consider a sequence $(m_n, \varepsilon_n)\to (m, \varepsilon)$ in $\widehat C_k \times [0,1]$. We distinguish two cases.

\subparagraph{Case 1: $\varepsilon >0$}

With no loss of generality, we can assume that $\varepsilon_n>0$ for all $n \geq 0$. First, let $U \subset \RR$ be an open interval. Then $p_k^{-1}(U)$ is an interval, whose interior we denote by $U'$. For each $n\geq 0$, we have
\begin{align*}
\widehat H_k(m_n,\varepsilon_n)(U)=H(m_n,\varepsilon_n) (p_k^{-1}(U))= H(m_n,\varepsilon_n) (U') ,
\end{align*}
where the second equality follows from the fact that $H(m_n,\varepsilon_n)$ is absolutely continuous and that $p_k^{-1}(U)\setminus U'$ consists of at most one point, which, whenever it exists, is the left-endpoint of $p_k^{-1}(U)$. Using that $H$ is continuous, we conclude by the same argument that
\begin{align}
\liminf_{n \to \infty} \widehat H_k(m_n,\varepsilon_n)(U) \geq H(m,\varepsilon)(U')= \widehat H_k(m,\varepsilon)(U). \label{Henri}
\end{align}
Next, let $K \subset \RR$ be a compact interval. Then $p_k^{-1}(K)$ is an interval, but it is not necessarily bounded. However, because the measures $m_n$ belong to $\widehat C_k$, using the definition of $H$ and letting $G_k$ be the convex hull of $F_k+[-1,1]$, we have $H(m_n,\varepsilon_n)(\RR\setminus G_k)=0$ and therefore
\begin{align*}
H(m_n,\varepsilon_n)(p_k^{-1}(K))=H(m_n,\varepsilon_n)(p_k^{-1}(K) \cap G_k)
\end{align*}
for all $n \geq 0$. We claim that $p_k^{-1}(K) \cap G_k$ is a (possibly empty) bounded interval. That it is an interval follows from the fact that it is the intersection of two intervals. Now, suppose, by way of contradiction, that it is not bounded. Then there exists an unbounded monotone sequence $(x_m)_{m \geq 0}$ in $p_k^{-1}(K) \cap G_k$. With no loss of generality, assume that this sequence is increasing. Because $p_k(x_m) \in K$ for all $m\geq 0$, we have, for any sufficiently large $m$,
\begin{align*}
p_k(x_m)=x^*:=\max \hskip 0.5mm K\cap \widehat F_k\, \text{ and } \, (x^*,x_m) \cap \widehat F_k =\emptyset.
\end{align*}
Letting $m \to \infty$, this implies $(x^*, \infty)\cap \widehat F_k =\emptyset$ and in turn that $G_k \subset (-\infty,x^*+1]$, a contradiction as $x_m \in G_k$ for all $m\geq 0$. The claim follows. Define $K'$ as the closure of $p_k^{-1}(K) \cap G_k$, we have
\begin{align*}
\widehat H_k(m_n,\varepsilon_n)(K)=H(m_n,\varepsilon_n) (p_k^{-1}(K)\cap G_k)= H(m_n,\varepsilon_n) (K') ,
\end{align*}
where the second equality follows from the fact that $H(m_n,\varepsilon_n)$ is absolutely continuous and that $K' \setminus(p_k^{-1}(K)\cap G_k)$ is finite. Using that $H$ is continuous, we conclude by the same argument that
\begin{align}
\limsup_{n \to \infty} \widehat H_k(m_n,\varepsilon_n)(K) \leq H(m,\varepsilon)(K')= \widehat H_k(m,\varepsilon)(K). \label{Michaux}
\end{align}
It follows from \eqref{Henri}--\eqref{Michaux} that $\widehat H_k$ is continuous at $(m, \varepsilon)$.

\subparagraph{Case 2: $\varepsilon =0$}

For each $n$ such that $\varepsilon_n=0$, we have $H_k(m_n,0)=m_n$, so we may assume with no loss of generality that $\varepsilon_n >0$ for all $n$. First, let $U \subset \RR$ be an open interval. Then $p_k^{-1}(U)$ is an interval, whose interior we denote by $U'$. As in Case 1, we have
\begin{align*}
\liminf_{n \to \infty} \widehat H_k(m_n,\varepsilon_n)(U) \geq H(m,0)(U')=m(U').
\end{align*}
Notice that $p_k^{-1}(U)\setminus U'$ consists of at most one point, which, whenever it exists, is the left-endpoint of $p_k^{-1}(U)$ and does not belong to $\widehat F_k$. Thus $m(U')= m(p_k^{-1}(U))$ as $m$ is concentrated on $\hat F_k$. Using the properties of $p_k$ and the fact that $m \in \widehat C_k$, it follows that
\begin{align*}
m(U')= m(p_k^{-1}(U))=m(U\cap \widehat F_k)=m(U).
\end{align*}
We conclude that
\begin{align}
\liminf_{n \to \infty} \widehat H_k(m_n,\varepsilon_n)(U) \geq m(U')=m(U)=\widehat H_k(m,0)(U). \label{David}
\end{align}
Next, let $K\subset \RR$ be a compact interval, and let $K'$ be the closure of $p_k^{-1}(K) \cap G_k$. As in Case 1, we have
\begin{align*}
\limsup_{n\to \infty} \widehat H_k(m_n,\varepsilon_n)(K) \leq H(m,0)(K')=m(K').
\end{align*}
Using the properties of $p_k$ and the fact that $m \in \widehat C_k$, we have
\begin{align*}
m(K)=m(K\cap \widehat F_k)=m(p_k^{-1}(K))=m(p_k^{-1}(K) \cap G_k).
\end{align*}
Because $G_k$ is a closed interval, any point in $K' \setminus (p_k^{-1}(K) \cap G_k)$ must be an endpoint of $p_k^{-1}(K)$ which does not belong to $p_k^{-1}(K)$ and thus cannot belong to $\widehat F_k$. It follows that
\begin{align*}
m(K)=m(p_k^{-1}(K) \cap G_k)=m(K').
\end{align*}
We conclude that
\begin{align}
\limsup_{n \to \infty} \widehat H_k(m_n,\varepsilon_n)(K) \leq m(K')=m(K)=\widehat H_k(m,0)(K). \label{Kreps}
\end{align}
It follows from \eqref{David}--\eqref{Kreps} that $\widehat H_k$ is continuous at $(m, 0)$.

\vskip 3mm

The analyses of Cases 1--2 above imply that $\widehat H_k$ is a contraction of $\widehat C_k$. Define then
\begin{align*}
H^*_k(m,\varepsilon) :=\widehat H_k(m\circ\psi_k^{-1},\varepsilon) \circ \psi_k, \quad (m,\varepsilon) \in C_k  \times[0,1].
\end{align*}
By composition, $H^*_k$ is continuous and, from the properties of $\widehat H_k$, we have
\begin{align*}
\forall m \in C_k, \;  H^*_k(m,0)=m \, \text{ and } \, H^*_k(m,0)=0.
\end{align*}
Hence $H^*_k$ is a contraction of $C_k$, as desired.

\subparagraph{Step 2} We now prove that $\Phi^i(m^j)$ is contractible. Define the map $H^*: \Phi^i(m^j)\times [0,1]  \to \Phi^i(m^j)$ by
\begin{align}
H^*(m,\varepsilon):= \overline m + \sum_{k \geq 0} H^*_k(m_k,\varepsilon), \quad (m,\varepsilon) \in \Phi^i(m^j) \times [0,1], \label{AA}
\end{align}
where $\overline m$ and the measures $(m_k)_{k\geq 0}$ are defined in \eqref{eq:decomposition_m_1} and, for each $k \geq 0$, $H^*_k$ is the map constructed in Step 1. We only need to prove that $H^*$ is continuous; indeed, that
\begin{align*}
\forall m \in \Phi^i(m^j), \; H^*(m,0)=m \,\text{ and } \,H^*(m,1)=\overline m
\end{align*}
follows directly from \eqref{eq:decomposition_m_1}, \eqref{AA}, and the properties of the maps $H^*_k$. Thus consider a sequence $(m_n,\varepsilon_n) \to (m,\varepsilon)$ in $\Phi^i(m^j)\times [0,1]$. Notice first that, for each $k\geq0$, the sequence $(m_{k,n})_{n \geq 0}$ converges to $m_k$ in $\CM(O_k)$, so that
\begin{align*}
H^*_k(m_{k,n},\varepsilon_n) \to H^*_k(m_k,\varepsilon).
\end{align*}
First, let $U$ be an open subset of $\CI$. If $U \cap \underline S^i \neq \emptyset$, then
\begin{align*}
H^*(m_n,\varepsilon_n)(U)=\infty  \to \infty =   H^*(m,\varepsilon)(U).
\end{align*}
If $U \cap \underline S^i= \emptyset$, then $U$ is the disjoint union of the open sets $U \cap O_k$, and we have
\begin{align*}
\liminf_{n\to \infty} H^*(m_n,\varepsilon_n)(U) &= \liminf_{n \to \infty} \,\sum_{k \geq 0} H^*_k(m_n,\varepsilon_n)(U\cap O_k)
\\
& \geq  \sum_{k \geq 0} \liminf_{n\to \infty} H^*_k(m_{k,n},\varepsilon_n)(U\cap O_k)
\\
&\geq  \sum_k H^*_k(m_k,\varepsilon)(U\cap O_k)
\\
&= H^*(m,\varepsilon)(U).
\end{align*}
Next, let $K$ be a compact subset of $\CI$. If $K\cap \underline S^i \neq \emptyset$, then
\begin{align*}
H^*(m_n,\varepsilon_n)(K)=\infty  \to \infty =   H^*(m,\varepsilon)(K).
\end{align*}
If $K\cap \underline S^i =\emptyset$, then $K$ is the disjoint union of the compact sets $K\cap O_k$, and we have
\begin{align*}
\limsup_{n \to \infty} H^*(m_n,\varepsilon_n)(K) &= \limsup_{n \to \infty} \, \sum_{k \geq 0} H^*_k(m_n,\varepsilon_n)(K\cap O_k)
\\
& \leq  \sum_{k \geq 0} \limsup_n H^*_k(m_{k,n},\varepsilon_n)(K\cap O_k)
\\
&\leq  \sum_{k \geq 0} H^*_k(m_k,\varepsilon)(K\cap O_k)
\\
&= H^*(m,\varepsilon)(K).
\end{align*}
It follows that $H^*$ is continuous at $(m, \varepsilon)$. Hence the result.
\end{proof}

\section{An Example}\label{sec:example}

We consider in this section the diffusion $X$ with state space $\CI=(0,1)$ solution of the SDE
\begin{align}\label{eq:SDE_example}
dX_t= X_t(1-X_t) \, \mathrm dW_t.
\end{align}
This process satisfies the assumptions of Section \ref{sec:model} and is a martingale appearing in filtering equations (see, e.g., \cite{LipsterShiryaev}) that satisfies $X_\infty:=\lim_{t \rightarrow \infty}X_t \in \{0,1\}$ a.s. We let the discount rate $r$ be equal to $0$. In order to satisfy Assumption A2, the payoff functions $R^i$ and $G^i$, $i=1,2,$ must converge to 0 at both boundaries $0$ and $1$.

The payoff functions, which we consider as functions on $[0,1]$ equal to 0 at 0 and 1, are represented in Figure 1. Assumption A1 is satisfied as all these functions are bounded.

\smallskip
\bigskip

\begin{tikzpicture}

\node at (1,-0.75) [blue]{$R^1$};
\node at (1,1.5) [red]{$G^1$};

  \draw[->] (0, 0) -- (12.5, 0) node[right] {$x$};
  \draw[->] (0,-2) -- (0,4.2) node[above] {};
\draw[scale=1, domain=0:2, smooth, variable=\x, blue]  plot ({\x}, {0.125*\x*\x*\x-1.5*\x});
\draw[scale=1, domain=2:6, smooth, variable=\x, blue] plot ({\x}, {-0.125*(\x-4)*(\x-4)*(\x-4)+1.5*(\x-4)});
\draw[scale=1, domain=6:10, smooth, variable=\x, blue] plot ({\x}, {0.125*(\x-8)*(\x-8)*(\x-8)-1.5*(\x-8)});
\draw[scale=1, domain=10:12, smooth, variable=\x, blue]  plot ({\x}, {-0.125*(\x-12)*(\x-12)*(\x-12)+1.5*(\x-12)});
\draw[scale=1, domain=0:12, smooth, dashed, variable=\x, blue]  plot ({\x}, {2});
\node at (6,2) [circle,fill,inner sep=1pt,blue]{};
\draw[scale=1, domain=0:2, smooth, variable=\x, red]  plot ({\x}, {-0.25*\x*\x*\x+3*\x});
\draw[scale=1, domain=2:4, smooth, variable=\x, red] plot ({\x}, {0.75*(\x-3)*(\x-3)*(\x-3)-2.25*(\x-3)+2.5});
\draw[scale=1, domain=4:6, smooth, variable=\x, red] plot ({\x}, {-0.5*(\x-5)*(\x-5)*(\x-5)+1.5*(\x-5)+2});
\draw[scale=1, domain=6:8, smooth, variable=\x, red] plot ({\x}, {+0.5*(\x-7)*(\x-7)*(\x-7)-1.5*(\x-7)+2});
\draw[scale=1, domain=8:10, smooth, variable=\x, red] plot ({\x}, {-0.75*(\x-9)*(\x-9)*(\x-9)+2.25*(\x-9)+2.5});
\draw[scale=1, domain=10:12, smooth, variable=\x, red]  plot ({\x}, {+0.25*(\x-12)*(\x-12)*(\x-12)-3*(\x-12)});
\draw[thick] (2, 0.1) -- (2, -0.1) node[below] {$\frac{1}{6}$};
\draw[thick] (3, 0.1) -- (3, -0.1) node[below] {$\frac{1}{4}$};
\draw[thick] (4, 0.1) -- (4, -0.1) node[below] {$\frac{1}{3}$};
\draw[thick] (6, 0.1) -- (6, -0.1) node[below] {$\frac{1}{2}$};
\draw[thick] (12, 0.1) -- (12, -0.1) node[below] {$1$};
\draw[thick] (0, 0.1) -- (0, -0.1) node[left] {$0$};


\end{tikzpicture}

\begin{tikzpicture}
\node at (1,0.75) [blue]{$R^2$};
\node at (1,3) [red]{$G^2$};

  \draw[->] (0, 0) -- (12.5, 0) node[right] {$x$};
  \draw[->] (0,-2) -- (0,4.2) node[above] {};
\draw[scale=1, domain=0:4, smooth, variable=\x, blue]  plot ({\x}, {0.00625*(5*\x*\x*\x-62*\x*\x+256*\x)});
\draw[scale=1, domain=4:6, smooth, variable=\x, blue]  plot ({\x}, {0.5*((\x-5)*(\x-5)*(\x-5)-3*(\x-5))+0.00625*192});
\draw[scale=1, domain=6:8, smooth, variable=\x, blue]  plot ({\x}, {0.5*(-(\x-7)*(\x-7)*(\x-7)+3*(\x-7))+0.00625*192});
\draw[scale=1, domain=8:12, smooth, variable=\x, blue]  plot ({\x}, {0.00625*(5*(12-\x)*(12-\x)*(12-\x)-62*(12-\x)*(12-\x)+256*(12-\x))});

\draw[scale=1, domain=2:4, smooth, dashed, variable=\x, blue]  plot ({\x}, {0.00625*(304+(\x-2)*68)});
\node at (2,0.00625*304) [circle,fill,inner sep=1pt,blue]{};
\draw[scale=1, domain=3:8, smooth, dashed, variable=\x, blue]  plot ({\x}, {0.00625*(345+(\x-3)*19)});
\node at (3,0.00625*345) [circle,fill,inner sep=1pt,blue]{};
\draw[scale=1, domain=4:8, smooth, dashed, variable=\x, blue]  plot ({\x}, {1+0.00625*192});
\node at (4,1+0.00625*192) [circle,fill,inner sep=1pt,blue]{};
\node at (8,1+0.00625*192) [circle,fill,inner sep=1pt,blue]{};

\draw[scale=1, domain=2:10, smooth, variable=\x, red]  plot ({\x}, {0.00625*(345+5*19)});
\node at (4,{0.00625*(345+5*19)}) [circle,fill,inner sep=1pt,red]{};
\node at (8,{0.00625*(345+5*19)}) [circle,fill,inner sep=1pt,red]{};

\draw[scale=1, domain=0:2, smooth, variable=\x, red]  plot ({\x}, {0.00625*(345+5*19)*0.25*(-0.25*\x*\x*\x+3*\x)});
\draw[scale=1, domain=10:12, smooth, variable=\x, red]  plot ({\x}, {0.00625*(345+5*19)*0.25*(0.25*(\x-12)*(\x-12)*(\x-12)-3*(\x-12))});
\draw[thick] (2, 0.1) -- (2, -0.1) node[below] {$\frac{1}{6}$};
\draw[thick] (3, 0.1) -- (3, -0.1) node[below] {$\frac{1}{4}$};
\draw[thick] (4, 0.1) -- (4, -0.1) node[below] {$\frac{1}{3}$};
\draw[thick] (6, 0.1) -- (6, -0.1) node[below] {$\frac{1}{2}$};
\draw[thick] (8, 0.1) -- (8, -0.1) node[below] {$\frac{2}{3}$};
\draw[thick] (12, 0.1) -- (12, -0.1) node[below] {$1$};
\draw[thick] (0, 0.1) -- (0, -0.1) node[left] {$0$};
\end{tikzpicture}
\vskip -1cm
\begin{center}
Figure 1: The players' payoff functions.
\end{center}

We shall not give explicit formulas for these functions, as this does not help for the proof. The properties of these functions that will be useful are the following:
\begin{enumerate}

\item

For $R^1$, $R^2$, $G^1$, and $G^2$ are symmetric around $1 \over 2$;

\item

$G^1>R^1$ on $(0,1)$, $R^1<0$ on $(0, \frac{1}{3}) \cup (\frac{2}{3},1)$, and $R^1>0$ on $(\frac{1}{3},\frac{2}{3})$;

\item

$G^1$ is decreasing on $[\frac{1}{6},\frac{1}{3}]$, and $G^1(\frac{1}{4})> R^1(\frac{1}{2})> G^1 (\frac{1}{3})$;

\item

$G^2>R^2$ on $(0,1)$, and $G^2$ is concave on $(0,1)$ and constant on $[\frac{1}{6},\frac{5}{6}]$;

\item

$R^2$ is strictly concave and $\mathcal C^2$ on $(0,\frac{1}{3}]$, $(R^2)'(\frac{1}{3})=0$, $R^2<R^2(\frac{1}{3})$ on $(\frac{1}{3},\frac{2}{3})$, and
\begin{align}\label{eq:tangents_R2}
R^2(\tfrac{1}{6})+(R^2)'(\tfrac{1}{6})(\tfrac{1}{3}-\tfrac{1}{6})=G^2(\tfrac{1}{3})\, \text{ and } \,
R^2(\tfrac{1}{4})+(R^2)'(\tfrac{1}{4})(\tfrac{2}{3}-\tfrac{1}{4})=G^2(\tfrac{2}{3}).
\end{align}
\end{enumerate}
The central result of this section can then be stated as follows.

\begin{proposition}\label{prop:nopureMPE}
Consider the lBwa with underlying diffusion process solution to \eqref{eq:SDE_example} and payoff functions illustrated in Figure 1 and satisfying Properties 1--5. Then$,$
\begin{itemize}

\item[(i)]

there exists no pure-strategy MPE$;$

\item[(ii)]

the randomized stopping times $(\mu^1,S^1):=(\alpha \delta_{\frac{1}{2}},\emptyset)$ and $(\mu^2,S^2):=(0,(0,x^*]\cup[1-x^*,1))$ form an MPE$,$ where $x^*$ is the unique solution in $(\tfrac{1}{4},\tfrac{1}{3})$ of $G^1(x^*)=R^1(\tfrac{1}{2})$ and
\begin{align*}
\alpha= \frac{(R^2)'(x^*)}{G^2(\tfrac{1}{2})- R^2(x^*)-(R^2)'(x^*)(\tfrac{1}{2}-x^*)}>0.
\end{align*}
\end{itemize}
\end{proposition}

To prove the Proposition \ref{prop:nopureMPE}(i), we will use a semi-harmonic characterization of best replies that can be found in \cite{Attard18}, proven in a more general framework. To deduce the statement below from \cite{Attard18}, we use that the fine topology associated to $X$ coincides with the usual topology in $(0,1)$, that all points of $(0,1)$ are regular for $X$, and that super-harmonic functions are just concave functions because $X$ is a martingale and $r=0$.

\begin{theorem}[{\cite[Theorem 5.3]{Attard18}}]\label{thm:attard}
Let $\bar J^i$ denote the pbr value function to some pure strategy $(0,S^j)$. Then $\bar J^i$ is continuous and is the pointwise minimum of the family of continuous functions $u: (0,1) \rightarrow \RR$ satisfying $R^i  \leq u \leq \mathrm{cav} \,G^i,$ $u=G^i$ on $S^j,$ and $u$ is concave on each connected component of $(0,1)\setminus S^j,$ where $\mathrm{cav} \,G^i$ is the smallest concave function bounded below by $G^i$.
\end{theorem}

Notice that, for one-dimensional continuous diffusions, the characterization given in Theorem \ref{thm:attard} has a local character, in the sense that the restriction of $\bar J^i$ to any connected component $(a,b)$ of $(0,1)\setminus S^j$ is the smallest concave function bounded below by $R^i$ that is equal to $G^i$ at $a$ and $b$ (where, if $a=0$, this equality means that the limit at $0+$ is $0$, as implied by the inequalities $R^i  \leq u \leq \mathrm{cav} \,G^i$ together with the fact that $\mathrm{cav} \,G^i(0)=0$, and similarly if $b=1$).

\begin{proof}[Proof of Proposition \ref{prop:nopureMPE}]
(i) Suppose, by way of contradiction, that a pure-strategy Mpe $((0,S^1),(0,S^2))$ exists. We will use Theorem \ref{thm:attard} several times during the proof, as well as the fact that, if $(0,S^i)$ is a pbr to $(0,S^j)$, then $S^i \subset \overline S\,\!^i=\{\bar J^i=R^i\}$ (Proposition \ref{geneprop}), where $\bar J^1,\bar J^2$ are the players' equilibrium brvfs.

We first claim that
\begin{align}\label{eq:pure_eq1}
S^1 \subset [\tfrac{1}{3},\tfrac{2}{3}].
\end{align}
Let $x \in (0,\frac{1}{3})$. If $x\in S^2$, then $x\notin S^1$ as $\bar J^1(x)= G^1(x)>R^1(x)$ on $(0,1)$. If $x\notin S^2$, let $(a,b)$ denote the connected component of $(0,1)\setminus S^2$ containing $x$. By Theorem \ref{thm:attard}, $\bar J^1$ is concave on $(a,b)$, bounded below by $R^1$, and equal to $G^1$ at $a$ and $b$. As $G^1 \geq 0$, if follows that $\bar J^1(x) \geq 0 > R^1(x)$ and thus $x\notin S^1$. A symmetric result holds for $(\frac{2}{3},1)$. Hence \eqref{eq:pure_eq1}.

We next claim that
\begin{align}\label{eq:pure_eq2}
S^2 \subset  (0,\tfrac{1}{3}] \cup [\tfrac{2}{3},1).
\end{align}
The proof of \eqref{eq:pure_eq2} is similar to that of \eqref{eq:pure_eq1} and is thus omitted.

Now, we claim that, if $S^1=S^1 \cap [\frac{1}{3},\frac{2}{3}] \neq \emptyset$, then it must be that
\begin{align}\label{eq:pureeq3}
\exists (x_0,x_1) \in  [\tfrac{1}{6},\tfrac{1}{4}] \times [\tfrac{3}{4},\tfrac{5}{6}], \; S^2=(0,x_0] \cup [x_1,1).
\end{align}
Let $z_0:=\min S^1 \in [\frac{1}{3},\frac{2}{3}]$. Using \eqref{eq:tangents_R2} along with the fact that $R^2$ is strictly concave and increasing on $[\frac{1}{6},\frac{1}{4}]$, we obtain that the mapping $x\mapsto R^2(x)+(R^2)'(x)(z_0-x)$ is decreasing on $[\frac{1}{6},\frac{1}{4}]$, {and satisfies
$R^2(1/4)+(R^2)'(1/4)(z_0-1/4)
 \leq G^2 (2/3) = G^2(z_0)$ and  $G^2(z_0)= G^2(1/3) \leq R^2(1/6)+(R^2)'(1/6)(z_0-1/6),$ where we have used that $G^2$ is constant on $[\frac{1}{6}, \frac{5}{6}]$.}
It follows that there exists a unique point $x_0 \in [\frac{1}{6},\frac{1}{4}]$ such that
\begin{align*}
R^2(x_0)+(R^2)'(x_0)(z_0-x_0)=G^2(z_0).
\end{align*}
By Theorem \ref{thm:attard}, the restriction of $\bar J^2$ to $(0,z_0]$ is the smallest concave function bounded below by $R^2$ and bounded above by $G^2$ which is equal to $G^2$ at $0$ and at $z_0$. It follows that $\bar J^2=R^2$ on $[0,x_0]$ and that $\bar J^2(x)=R^2(x_0)+(R^2)'(x_0)(x-x_0)$ for all $x\in [x_0,z_0]$. A symmetric argument on the interval $[z_1,1)$ with $z_1:= \max S^1$ shows that there exists $x_1\in [\tfrac{3}{4},\tfrac{5} {6}]$ such that $\bar J^2=R^2$ on $[x_1,1]$ and $\bar J^2(x)=R^2(x_1)+(R^2)'(x_1)(x-x_1)$ for all $x\in [z_1,x_1]$. Finally, it must be that $\bar J^2=G^2$ on $[z_0,z_1]$: first, on $[z_0,z_1]$, $\bar J^2 \leq G^2$, with equality on $S^1$, and $G^2$ is constant; second, $\bar J^2$ is concave on any connected component $(a,b)$ of $(z_0,z_1) \setminus S^1$ and equal to $G^2$ at $a$ and $b$. Therefore $\bar J^2$ is constant on any such interval, and thus $\bar J^2=G^2$ on $[z_0,z_1]$. Hence \eqref{eq:pureeq3}.

Now, if \eqref{eq:pureeq3} holds, then it must be that $S^1=\emptyset$. Indeed, any continuous function that is equal to $G^1$ on $S^2$ and is concave on $(x_0,x_1)$ is strictly larger than $R^1$, so that $\bar J^1 > R^1$ using again Theorem \ref{thm:attard}. Because $S^1 \not = \emptyset$ implies \eqref{eq:pureeq3}, we deduce that it must be that $S^1=\emptyset$. We deduce from this that $S^2=(0,\frac{1}{3}]\cup[\frac{2} {3},1)$, and hence that $\bar J^1({1 \over 2}) = G^1({1 \over 3} )= G^1({2\over 3})$ because $r=0$ and $X$ is a martingale such that $X_\infty \in \{0,1\}$ a.s., a contradiction as $G^1(\tfrac{1}{3})=G^1(\tfrac{2}{3}) < R^1(\tfrac{1} {2})$ and $\bar J^1 \geq R^1$. We conclude that no pure-strategy Mpe exists.

\bigskip

(ii) Let $\bar J^i$ denote the brvf to $(\mu^j,S^j)$ for the randomized stopping times defined in the statement of the proposition. First, we easily see that $\bar J^1$ is equal to $G^1$ on $S^2$ and is constant and equal to $R^1(\frac{1}{2})$ on $[x^*,1-x^*]$. It follows that $\overline S\,\!^1=\{\frac{1}{2}\}$. To show that $(0,S^1)=(0,\emptyset)$ is a best reply to $(0,S^2)$, just note that the expected payoff from not stopping starting from any point in $(x^*,1-x^*)$ is equal to $G^1(x^*)=G^1(1-x^*)=R^1(\frac{1}{2})$ because $r=0$ and $X$ is a martingale such that $X_\infty \in \{0,1\}$ a.s. From Proposition \ref{geneprop}(iv), we conclude that $(0,\alpha' \delta_{\frac{1}{2}})$ is a pbr to $(0,S^2)$ for any nonnegative $\alpha'$.

Notice then that $\alpha>0$. Indeed, $R^2$ is strictly concave and increasing on $({1 \over 4} ,\frac{1}{3}]$, so that $(R^2)'(x^*)>0$ and
\begin{align*}
R^2(x^*)+(R^2)'(x^*)(\tfrac{1}{2}-x^*) < R^2(\tfrac{1}{4})+(R^2)'(\tfrac{1}{4})(\tfrac{1}{2}-\tfrac{1}{4})<R^2(\tfrac{1}{4})+(R^2)'(\tfrac{1}{4})(\tfrac{2}{3}-\tfrac{1}{4}) = G^2(\tfrac{1} {2})
\end{align*}
by \eqref{eq:tangents_R2} as $G^2(\tfrac{1} {2})= G^2(\tfrac{2} {3})$. Let us prove that $(0,S^2)$ is a pbr to $(\alpha \delta_{\frac{1}{2}}, \emptyset)$. By Proposition \ref{prop:cacrt_pbr}, it is sufficient to prove that $\bar J^2$ is equal to $R^2$ on $S^2$ and is strictly larger than $R^2$ on $(x^*,1-x^*)$. Let $w^2 :(0,1) \rightarrow \RR$ be equal to $R^2$ on $S^2$ and such that
\begin{align*}
w^2(x)&= R^2(x^*)+(R^2)'(x^*)(x-x^*), \quad x \in [x^*,\tfrac{1}{2}],
\\
w^2(x)&= R^2(1-x^*)+(R^2)'(1-x^*)[x-(1-x^*)],\quad x \in [\tfrac{1}{2},1-x^*].
\end{align*}
Notice that $w^2> R^2$ on $(x^*,1-x^*)$, that $w^2$ is $\mathcal C^1$ on $(0, {1 \over 2}) \cup ({1 \over 2},1)$ and piecewise $\mathcal C^2$, and that $w^2$ is solution to the variational system
\begin{align*}
w^2(0+)&=w^2(1-)=0 ,
\\
w^2 &= R^2 \text{ on } (0,x^*]\cup[1-x^*,1) ,
\\
(w^2)'(x^*)&= (R^2)'(x^*) ,
\\
(w^2)'(1-x^*)&= (R^2)'(1-x^*),
\\
(w^2)''&=0 \text{ on } (x^*,\tfrac{1}{2})\cup(\tfrac{1}{2},1-x^*) ,
\\
(w^2)''&<0 \text{ on }(0,x^*)\cup(1-x^*,1) ,
\\
\alpha [ G^2(\tfrac{1}{2})- w^2(\tfrac{1}{2})]+\tfrac{1}{2}\Delta (w^2)'(\tfrac{1}{2}) &=0.
\end{align*}
Proceeding along the same lines as in \cite[Lemma A.4]{DGM2}, the proof that $\bar J^2=w^2$ now follows from a standard verification argument based on the It\^o--Tanaka--Meyer formula. First, let us observe that for $\tau \in \CT$ and denoting by $L$ the local time of $X$ at $\frac{1}{2}$, we have
\begin{align} \label{eq:payoff_ito}
J^2(x, (\mu^1, S^1), \tau) =  \EE_x\! \left[R^2(X_{\tau}) \Lambda_{\tau}^1 +\int_{[0, \tau)} G^2 (X_s) \Lambda_s^1  \, \mathrm d L_s \right] \hskip -1mm.
\end{align}
Applying the It\^o--Tanaka--Meyer formula to the process $(\Lambda^1_{t} w^2 (X_{t}))_{t \geq 0}$ yields
\begin{align}
w^2(x) = \Lambda^1_{\tau}  w^2(X_{\tau}) & -\int_{[0,{\tau} )} w^2(X_s)\, \mathrm d\Lambda_s^1 - \int_{[0,{\tau})} \Lambda_s^1 (w^2)'(X_s) \, \mathrm dX_s \notag
\\
& - \frac{1}{2} \int_{[0,{\tau} )} \Lambda_s^1 (w^2)''(X_s) X_s^2(1-X_s)^2 \, \mathrm ds  - \frac{1}{2} \,\Delta  ({w^2})' (\tfrac{1}{2}) \int_{[0,{\tau})}  \Lambda_s^1  \, \mathrm dL_s. \label{Hito}
\end{align}
Because $(w^2)'' \leq 0$ on $(0,1)\setminus\{x^*,\frac{1}{2},1-x^*\}$, with equality on $(x^*,1-x^*)\setminus\{ \frac{1}{2}\}$, we have
\begin{align}
-\int_{[0,{\tau} )} \Lambda_s^1 (w^2)''(X_s) X_s^2(1-X_s)^2\, \mathrm ds\geq 0. \label{Itoh}
\end{align}
From the last line of the variational system for $w^2$ and the properties of $L$, we have
\begin{align}
- \,\frac{1}{2} \, \Delta  (w^2)' (\tfrac{1}{2}) \int_{[0,{\tau})}\Lambda_s^1  \, \mathrm dL_s &=  \alpha \big[G^2 (\tfrac{1}{2}) - w^2 (\tfrac{1}{2}) \big] \int_{[0,{\tau})} \Lambda_s^1  \, \mathrm dL_s \notag
\\
&= \int_{[0,{\tau})} \alpha G^2(X_s) \Lambda^1_s   \, \mathrm d L_s - \int_{[0,{\tau})}  \alpha\Lambda_s^1 w^2(X_s) \, \mathrm d L_s \notag
\\
&= \int_{[0,{\tau})} G^2(X_s) \, \mathrm d\Gamma^1_s  + \int_{[0,{\tau})}w^2(X_s)\, \mathrm d\Lambda_s^1. \label{Itho}
\end{align}
We deduce that
\begin{align}
w^2(x)& \geq \EE_x \!\left[\Lambda_{\tau}^1 w^2(X_{\tau }) +  \int_{[0,{\tau})} G^2(X_s) \, \mathrm d\Gamma^1_s  \right] \notag
\\
& \geq  \EE_x \!\left[\Lambda_{\tau}^1 R^2(X_{\tau}) +  \int_{[0,{\tau})}G^2(X_s) \, \mathrm d\Gamma^1_s  \right] \notag
\\
&= J^2(x,(\mu^1, S^1), \tau), \label{lastbutnotleast}
\end{align}
where the first inequality follows from (\ref{Hito})--(\ref{Itho}) along with the fact that the stochastic integral in \eqref{Hito} is a centered integrable variable as $X$ is a bounded martingale, the second inequality follows from the fact that $w^2 \geq R^2 $ on $(0,1)$, and the equality follows from \eqref{eq:payoff_ito}. Taking the supremum over $\tau \in \mathcal T$ in \eqref{lastbutnotleast} yields $w^2 \geq \bar J^2$. It is easy to check that the above inequalities turn into equalities when $\tau=\tau_{S^2}$, which concludes the proof that $w^2=\bar J^2$. Hence the result.
\end{proof}

Let us conclude this section by explaining why the pure Nash equilibria constructed using the method of Hamadene and Zhang \cite{HamadeneZhang} need not be Markovian. In our example, the algorithm in \cite{HamadeneZhang} actually stops after two iterations and leads to the following equilibrium: Assume first that player $1$ never stops. Then, as shown in the proof of Proposition \ref{prop:nopureMPE}, a pure best reply of player $2$ is to use the hitting time $\tau_{S^2}$ with $S^2: =(0,\frac{1}{3}]\cup[ \frac{2}{3},1)$. In turn, facing the strategy $(0,S^2)$, a pure best reply of player $1$ is to use the hitting time $\tau_{S^1}$, where, letting $\bar J^1$ denote the brvf of player $1$ against $(0,S^2)$, $S^1:=\{ \bar J^1 =R^1 \}$ is a nonempty subset of $(\frac{1}{3},\frac{2}{3})$, see again the proof of Proposition \ref{prop:nopureMPE}. Define then the stopping time
\begin{eqnarray*}
\tau^1= \indic_{\tau_{S^1}<\tau_{S^2}} \,\tau_{S^1} + \indic_{\tau_{S^2}<\tau_{S^1}} \,\infty .
\end{eqnarray*}
This strategy consists for player $1$ in stopping in $S^1$ if $X$ did not visit $S^2$ before, and to never stop if $X$ visits $S^2$ before $S^1$ (one could say that player $1$ threatens to play $\infty$ if player $2$ does not stop in $S^2$). On the one hand, $\tau^1$ is a best reply to $\tau_{S^2}$ as it gives the same payoff to player $1$ as $\tau_{S^1}$ against $\tau_{S^2}$. On the other hand, whereas, as shown in the proof of Proposition \ref{prop:nopureMPE}, $\tau_{S^2}$ is not a best reply to $\tau_{S^1}$, it turns out that $\tau_{S^2}$ is a best reply to $\tau^1$ and that $(\tau^1,\tau_{S^2})$ is a Nash equilibrium. Indeed, when facing the strategy $\tau^1$, player $2$ will not stop if $S^1$ is reached before $S^2$ as $G^2>R^2$, and player $2$ will not stop before $X$ reaches $S^1$ or $S^2$ as this would give him a strictly smaller payoff than playing $\tau_{S^2}$.
However, if $X$ reaches $S^2$ before $S^1$, player 2 believes that player $1$ will never stop in the future, and thus the best player $2$ can do is to play a best reply against the stopping time $\infty$, that is, to stop in $S^2$. Notice that we may reverse the roles of the players in this construction and obtain another Nash equilibrium in which player $2$ plays a non-Markovian strategy.

{

\section{Discussion} \label{sec:discussion}

In this section, we briefly discuss the scope and limits of our analysis.

\subsection{More General Classes of Games}

Extending our analysis to lBwas with $I > 2$ players presents two difficulties.

First, the game may continue after one of the players has decided to stop---think, for instance, of firms sequentially exiting from an industry. Then the state of the game should comprise, besides the current value of $X$, also the identities of the remaining players. This calls for a recursive construction of an Mpe, starting from continuation games where only two players remain. At a minimum, we would have to verify that the continuation-equilibrium value functions of the remaining players satisfy the analogues of A0--A2. To add to the difficulty, an equilibrium-selection problem may arise because these value functions need not be uniquely determined.

Second, even if the game is over as soon as one player has decided to stop, the players' rewards may still depend on who stops first---e.g., because of externalities among players. This suggests that the very formalization of the lBwa should be amended, because, accounting for join stopping decisions, up to $2 \sum_{i=1}^I \! \binom{I}{i} = 2(2^I-1)$ different reward functions may have to be defined. The main simplification compared to the previous case is that these reward functions would be exogenously given, rather than being endogenously determined as part of a continuation Mpe.

Despite these difficulties, it should be noted that, in both cases, the challenge is less conceptual than computational. In particular, the characterization of Markovian randomized stopping times in Theorem \ref{representation} and our topological methods are likely to prove useful for equilibrium analysis. We leave these extensions for future work.

To conclude our discussion of the $I$-player lBwa, consider finally the case where the game is over as soon as one player has decided to stop and the players' rewards do not depend on who stops first. An example is Bliss and Nalebuff's \cite{BlissNalebuff} model of private provision of a public good, in which everyone benefit from provision but only the person who takes the initiative has to pay the cost (see also Bensoussan and Friedman \cite[\S1, Remark]{BensoussanFriedman}). As we now show, this gives the $I$-player lBwa the structure of an aggregative game (Selten \cite{Selten0}), which enables us to directly generalize Theorem \ref{thm:main}.

Specifically, consider a $I$-player lBwa in which the expected payoff of every player $i$ is given by
\begin{align}
J^i(x,\tau^i,(\tau^j)_{j\neq i}): = \EE _x\! \left[\indic_{\tau^i \leq  \bigwedge_{j\neq i} \tau^j} \,\mathrm e^{-r \tau^i}R^i (X_{\tau^i}) + \indic_{\tau^i>\bigwedge_{j\neq i} \tau^j}\, \mathrm e^{-r \bigwedge_{j\neq i} \tau^j} G^i\big(X_{\bigwedge_{j\neq i} \tau^j}\big) \right]\hskip -1mm ,\label{coreN}
\end{align}
where the functions $R^i$ and $G^i$ in \eqref{coreN} satisfy  A0--A2. The following result then holds.

\begin{theorem} \label{pg}
Any $I$-player lBwa with payoffs \eqref{coreN} admits an Mpe.
\end{theorem}

\begin{proof} Let $\lambda^{I-1}$ denote $(I-1)$-dimensional Lebesgue measure. For each $i=1,\ldots,I$ and $\lambda^{I-1}$-a.e.\! $(u^j)_{j\neq i}\in [0,1]^{I-1}$, $\bigwedge_{j \neq i} \gamma^j(\cdot,u^j) \in\CT $. As the players' randomization devices are independent, the corresponding conditional survival function is given by
\begin{align}
\lambda^{I-1}\big[\{(u^j)_{j\neq i} \in [0,1]^{I-1}:  \gamma^j(\cdot,u^j) >t \text{ for all } j\neq i\}\big]&= \prod_{j\neq i} \int_0^1\mathbbm{1}_{\{ \gamma^j (\cdot, u^j) >t\}}\, \mathrm du^j \label{lessommeslessommes}
\\
& = \prod_{j\neq i} \Lambda_t^j \notag
\\
&= \prod_{j\neq i} \left[ \mathbbm{1}_{t < \tau_{S^j}} \, \mathrm e^{- \int_{\mathcal I \setminus S^j} L_t^y \, \mu^j(\mathrm dy)}\right] \notag \allowdisplaybreaks
\\
&=\displaystyle \mathbbm{1}_{t < \tau_{\,\bigcup_{j \neq i} S^j} }\, \mathrm e^{- \int_{\mathcal I \setminus  \bigcup_{j \neq i} S^j}L_t^y \,  {\sum} _{j \neq i}\,\mu^j (\mathrm dy)}, \notag
\end{align}
the third equality follows from Theorem \ref{representation}, and the fourth equality follows from the fact that, if $t < \tau_{\,\bigcup_{j \neq i} S^j}$, then $L_t^y = 0$ for all $y \in \bigcup_{j \neq i} S^j$. We can identify $({\sum} _{j \neq i}\,\mu^j, \bigcup_{j \neq i} S^j)$ with a measure $m _{(\mu^j, S^j)_{j \neq i}} \in \CM(\CI)$ defined by
\begin{align*}
m_{(\mu^j, S^j)_{j \neq i}}(A):= \begin{cases}  {\sum} _{j \neq i}\,\mu^j (A)  & \text{if } A\cap \bigcup_{j \neq i} S^j = \emptyset \\ \infty & \text{if } A\cap \bigcup_{j \neq i} S^j  \neq \emptyset \end{cases}, \quad  A \in \CB(\CI).
\end{align*}
By construction, $e(m _{(\mu^j, S^j)_{j \neq i}}) = \bigcup_{j\neq i} e(m^j)$. It follows from \eqref{lessommeslessommes} that a necessary and sufficient condition for $(\mu^i,S^i)$ to be a best reply to $(\mu^j, S^j)_{j \neq i}$ in the $I$-player lBwa is that $(\mu^i,S^i)$ be a pbr to $m_{(\mu^j, S^j)_{j \neq i}}$ in the 2-player lBwa. Now, defining $\Phi^i$ as in \eqref{defPhi} for every player $i$, consider the correspondence $\Phi_I: \CM(\CI)^I \twoheadrightarrow \mathcal M(\mathcal I)^I$ defined by
\begin{align*}
\Phi_I \big((\mu^i,S^i)_{i =1}^I\big):=\bigtimes_{i=1}^I \Phi^i\big(m_{(\mu^j, S^j)_{j \neq i}}\big), \quad (\mu^i,S^i)_{i =1}^I \in  \CM(\CI)^I.
\end{align*}
It is sufficient to prove that $\Phi_I$ has a fixed point. By Proposition \ref{prop:contractible_values}, $\Phi_I$ has contractible values. Hence, by Theorem \ref{thm:fixed_point}, we only need to check that the graph of $\Phi_I$ is closed. To this end, it is sufficient to prove that the summation mapping
\begin{align}
\CM(\CI)^I\to \CM(\CI): (m^i)_{i=1}^I \mapsto \sum_{i=1}^I m^i \label{summation}
\end{align}
is continuous---i.e., that, if, for each $i= 1,\ldots, I$, $m^i_n \to m^i$, then $\sum_{i=1}^I m^i_n \to \sum_ {i=1}^I m^i $. We use Proposition \ref{prop:compact_metric}. First, let $O$ be an open set such that $O \cap e(\sum_{i=1}^I m^i)\neq \emptyset$. We have
\begin{align*}
O \cap e \!\left(\sum_{i=1}^I m^i \right) \! =  O \cap \bigcup_{i=1}^I e(m^i) = \bigcup_{i=1}^I \, [O \cap e(m^i)]
\end{align*}
so that $O \cap e(m^i)  \neq \emptyset$ for some $i$. It follows that $m_n^i (O) \to \infty$ and thus $\sum_{i=1}^I m_n^i (O) \to \infty$. Second, for each $i=1,\ldots, I$, $L_\phi( m_n^i) \to L_\phi(m^i)$ for all $\phi\in \CC_c^+(\CI \setminus e(m^i))$. Hence, for each $\phi \in \CC_c^+(\CI \setminus e(\sum_{i=1}^I m^i )) = \CC_c^+(\CI \setminus \bigcup_{i=1}^I e(m^i))$, we have $L_\phi(\sum_{i=1}^I m_n^i)  = \sum_{i=1}^I L_\phi(m_n^i ) \to \sum_{i=1}^I L_\phi(m^i ) =L_\phi(\sum_{i=1}^I m^i)$. We deduce from Proposition \ref{prop:compact_metric} that the summation mapping \eqref{summation} is continuous, as desired. Hence the result. \end{proof}

Extending our analysis beyond games with a second-mover advantage raises challenges that are of a more conceptual nature. As pointed out by Fudenberg and Tirole \cite{FudenbergTirole} in the deterministic case and Riedel and Steg \cite{RiedelSteg} in the stochastic case, the main challenge when there is a first-mover advantage lies in the treatment of coordination failures, which arise for instance when two firms simultaneously attempt to invest on a market that can profitably support only one of them. These authors argue that accounting for the risk of simultaneous moves while preserving a notion of subgame-perfectness in the spirit of \cite{Selten} requires a modification of the very concept of a strategy and of how strategy profiles determine outcomes. A plausible conjecture is that, in the case of a game with symmetric and sufficiently regular reward functions $R$ and $G$, the existence of an Mpe can be established by pasting together Mpes in Markovian randomized stopping times (in attrition regions where $G > R$) with Mpes in generalized mixed strategies in the sense of \cite{RiedelSteg} (in preemption regions where $R>G$). How to extend this construction to the case of asymmetric reward functions remains an open question.

\subsection{More General Environments}

Extending our analysis to more general Markov processes $X$ or more general payoff functions presents several difficulties.

There is a vast literature on the relation between additive or multiplicative functionals of general Markov processes and measures (see, e.g., \cite{Revuz}, \cite{Fukushima}, and the references therein). However, in order to extend our analysis to other classes of processes $X$, we first have to identify the relevant space of measures, then to find a compact topology which makes this set an absolute retract, and finally to prove the closedness of the graph of best replies for this topology. Each of these steps can be very challenging and such extensions are left for future research. To illustrate the fact that our analysis has to be adapted to each class of processes without entering too much into technicalities, let us consider examples of randomized stopping times in the three most natural cases: nonhomogeneous problems, linear diffusions with jumps, and multidimensional continuous diffusions.

Notice first that the definition of Markovian randomized stopping times given in Section \ref{sec:model} does not depend on which Markov process $X$ we consider.

If $X$ is a nonhomogeneous diffusion, or if the stopping game has a finite horizon $T$, or if the rewards are time-dependent, the natural concept of a (nonhomogeneous) Markovian randomized stopping time will be that of a multiplicative functional of the space-time process $(Z_t,X_t)$ with $Z_t=Z_0+t$ taking values in $[0,\infty)\times \CI$. To any positive measure $\mu$ on $[0,\infty)$, we can associate a multiplicative functional of the process $Z_t$ (and thus also of $(Z_t,X_t)$) through the relation
\begin{align*}
\Lambda^\mu_t(Z_\cdot) := \lim_{u \downarrow 0} \, \mathrm e^{- \int_{(Z_0,Z_t+u] }\,\mu(\mathrm d s)},\quad t \geq 0.
\end{align*}
However, if we consider a decreasing sequence $t_n \downarrow t^*$, the sequence $\Lambda^{\delta_{t_n}}$ does not converge to $\Lambda^{\delta_ {t^*}}$ if $Z_0=t^*$, showing that Proposition \ref{prop:cvg_lambda} does not hold if we use the usual vague convergence of measures. The topology has to be adapted to take care of the specific behavior of $Z_t$.

If $X$ is a homogeneous diffusion with jumps, the csf of a Markovian randomized stopping time can have several discontinuities as, for example, in
\begin{align*}
\Lambda_t(X_\cdot):=\lim_{u \downarrow 0}\prod_{0<s\leq t+u} F(X_{s-},X_s), \quad t \geq 0,
\end{align*}
where $F$ is a measurable map from $\CI\times \CI$ to $[0,1]$ that is identically equal to $1$ on the diagonal. In this case, depending on the distribution of jumps of $X$, several functions $F$ may lead to equivalent csfs (i.e., which coincide a.s.). Therefore, it is clear that any measure representing this strategy will depend on the distribution of the jumps of $X$.

If $X$ is a standard Brownian motion of dimension $d\geq 2$, given a nonnegative Borel function $f : \RR^d \rightarrow \RR$, one may define a randomized Markovian stopping time by
\begin{align*}
\Lambda_t := \lim_{u \downarrow 0} \, \mathrm e^{-\int_0^{t+u} f(X_s)\, \mathrm ds},\quad t \geq 0.
\end{align*}
Using the theory of Revuz measures (see \cite[Example 5.1.1]{Fukushima}), this strategy is naturally associated to the absolutely continuous positive measure $m:=f\cdot \lambda_d$ where $\lambda_d$ stands for the $d$-dimensional Lebesgue measure. Given a countable dense subset $S$ of $\RR^d$, one may construct\footnote{See \cite[Proposition 1.3 and Remark 1.4]{Albeverio} for the construction of $f$ and \cite[theorem 5.1.4 and Example 5.1.1]{Fukushima} for the properties of $\Lambda$, which are proved in \cite{Fukushima} for the additive functional $A_t=-\log(\Lambda_t)$.} a function $f$ such that for every open set $U\subset \RR^d$, $\int_U f(x) \, \mathrm dx=\infty$, while
\begin{align*}
\Lambda_0&=0 \text{ }\PP_x\text{-a.s. if and only if } x \in S,
\\
\Lambda_0&=1 \text{ }\PP_x\text{-a.s. if and only if } x \notin S.
\end{align*}
Moreover, for all $x \notin S$, we have $\PP_x \hskip 0.3mm [\forall t\geq 0,  \,  \Lambda_t>0]=1$. Therefore, the set $S$ on which one stops with probability $1$ is not closed and polar (i.e., for all $x \in \RR^d$, $\PP_x\hskip 0.3mm[\exists t>0\!:\, X_t \in S]=0$).\footnote{ However, $S$ is finely closed (see, e.g., \cite[Appendix A.2]{Fukushima} for the definition of the fine topology).} Extending formally the definition of $e(m)$ we used on $\CM(\CI)$ to positive measures on $\RR^d$ gives $e(m)=\RR^d$ with $m=f\cdot \lambda_d$, and we check easily that $m$ is not outer regular with respect to open sets. The set of measures we consider and the chosen topology has thus to be adapted when $X$ is a multidimensional diffusion, which seemingly requires to take care of the many subtleties appearing when going from dimension 1 to higher dimensions in the study of additive functionals or in potential theory. Notice finally that the proof of Proposition \ref{prop:cvg_lambda}, which is crucial to prove the closedness of the graph of the best reply correspondence, relies on the existence and properties of local times of $X$ at every point of $\CI$, whereas such local times are known not to exist for multidimensional Brownian motion.}

\newpage

\section*{Appendix} \label{sec:appendix}

\renewcommand{\thesection}{A}

\numberwithin{equation}{section}

\setcounter{equation}{0}
\setcounter{subsection}{0}

\newtheorem{applem}{Lemma}[section]

\subsection{Proof of Proposition \ref{geneprop}}

Point $(a)$ follows from the fact that stopping immediately is suboptimal in problem \eqref{optstop1}. Point $(b)$ follows from the fact that, for $x \in S^j$, the payoff of player $i$ is $G^i(x)$ if he does not stop immediately and $R^i(x)\leq G^i(x)$ otherwise. Let us now prove point $(c)$. Under the stated condition, by continuity, there exist $C \in \RR$ and $\varepsilon,\delta>0$ such that
\begin{align}\label{eq:GRepsilon}
\forall y\in [x-\delta,x+\delta],\; G^i(y)\geq C \geq  R^i(y)+\varepsilon.
\end{align}
Using that $\tau^i=\infty$ is suboptimal in problem \eqref{optstop1}, we have for all $y\in [x-\delta,x+\delta]$, letting $\tau_x$ and $\tau_\delta$ denote respectively the hitting time of $x$ and the exit time of $[x-\delta,x+\delta]$:
\begin{align}
\bar J^i(y) &\geq \EE_y\!\left[ \int_{[0,\infty)} \mathrm e^{-rs} G^i(X_s) \, \mathrm d\Gamma^j_s \right] \notag
\\
&=\EE_y \!\left[ \int_{[0,\tau_{x}]} \mathrm e^{-rs} G^i(X_s)\, \mathrm d\Gamma^j_s \right] \notag
\\
&= \EE_y\! \left[ \indic_{\tau_x < \tau_\delta}\int_{[0,\tau_{x}]} \mathrm e^{-rs} G^i(X_s)\, \mathrm d\Gamma^j_s \right]+\EE_y\!\left[\indic_{\tau_x > \tau_\delta} \int_{[0,\tau_{x}]}\mathrm e^{-rs} G^i(X_s)\, \mathrm d\Gamma^j_s \right]  \notag
\\
& \geq C \EE_y\hskip 0.3mm[ \mathrm e^{-r\tau_x} \indic_{\tau_x < \tau_\delta} ]+\EE_y\!\left[\indic_{\tau_x > \tau_\delta} \int_{[0,\tau_{x}]} \mathrm e^{-rs} G^i(X_s)\, \mathrm d\Gamma^j_s \right]\hskip -1mm, \label{chat}
\end{align}
where the second inequality follows from the fact that $\int_{[0,\tau_x]} \, \mathrm d\Gamma^j_s=1$ when $\tau_x<\infty$ as $x\in S^j$. Consider the last term on the right hand side of \eqref{chat}. We have
\begin{align*}
\EE_y &\! \left[\indic_{\tau_x > \tau_\delta} \int_{[\tau_\delta,\tau_{x}]} \mathrm e^{-rs} |G^i(X_s)|\, \mathrm d\Gamma^j_s \right]
\\
&= \EE_y \!\left[\indic_{\tau_x > \tau_\delta}\,\mathrm  e^{-r \tau_\delta} \left[ |G^i(X_{\tau_\delta})|(\Gamma^j_{\tau_\delta}-\Gamma^j_{\tau_\delta -})+  \Lambda^j_{\tau_\delta}\int_{(\tau_\delta,\tau_{x}]} \mathrm e^{-r(s-\tau_\delta)} |G^i(X_s)| \, \mathrm d(\Gamma^j_s\circ \theta_{\tau_\delta}) \right] \right]
\\
&= \EE_y\!\left[\indic_{\tau_x > \tau_\delta} \,\mathrm e^{-r \tau_\delta} \Lambda^j_{\tau_\delta-}\int_{[\tau_\delta,\tau_{x}]} \mathrm e^{-r(s-\tau_\delta)} |G^i(X_s)|\, \mathrm  d(\Gamma^j_s\circ \theta_{\tau_\delta}) \right]\hskip -1mm,
\end{align*}
where the first equality follows from \eqref{mark}, and the second equality follows from the facts that $\Gamma^j_{\tau_\delta}-\Gamma^j_{\tau_\delta -}=\Lambda^j_{\tau_\delta-}-\Lambda^j_{\tau_\delta}$ and that $\Lambda^j$ is continuous except at $\tau_{S^j}$ where it jumps to 0. Using this result, we have, for some constant $C'>0$,
\begin{align}
&\left|\EE_y\!\left[\indic_{\tau_x > \tau_\delta} \int_{[0,\tau_{x}]}\mathrm  e^{-rs} G^i(X_s)\,\mathrm d\Gamma^j_s \right]\right| \notag
\\
&\leq \EE_y\!\left[\indic_{\tau_x > \tau_\delta} \int_{[0,\tau_\delta)} \mathrm e^{-rs} |G^i(X_s)|\,\mathrm d\Gamma^j_s \right]+ \EE_y\!\left[\indic _{\tau_x > \tau_\delta} \int_{[\tau_\delta,\tau_{x}]} \mathrm e^{-rs} |G^i(X_s)|\,\mathrm  d\Gamma^j_s \right] \notag
\\
&\leq \sup_{[x-\delta,x+\delta]} |G^i| \, \EE_y[\indic_{\tau_x > \tau_\delta}] + \EE_y\!\left[\indic_{\tau_x > \tau_\delta}\, \mathrm  e^{-r \tau_ \delta} \Lambda^j_{\tau_\delta-}\int_{[\tau_\delta,\tau_{x}]} \mathrm e^{-r(s-\tau_\delta)} |G^i(X_s)|\,\mathrm d(\Gamma^j_s\circ \theta_{\tau_ \delta}) \right]  \allowdisplaybreaks \notag
\\
&\leq \sup_{[x-\delta,x+\delta]} |G^i| \, \EE_y[\indic_{\tau_x > \tau_\delta}] + \EE_y\!\left [\indic_{\tau_x > \tau_\delta} \, \EE_{X_{\tau_\delta}} \!\!\left[\sup_{t \geq 0}\,\mathrm e^{-rt}|G^i(X_t)|\right] \right]\notag
\\
&\leq C' \PP_y\hskip 0.3mm[\tau_x > \tau_\delta], \label{chaton}
\end{align}
where the third inequality follows from the Markov property, and the fourth inequality follows from assumption A1 along with the fact that $X_{\tau_\delta} \in \{x-\delta,x+\delta\}$ $\PP_y$-almost surely. From (\ref{chat})--\eqref{chaton}, we deduce that
\begin{align*}
\bar J^i(y) \geq C \EE_y \hskip 0.3mm[ \mathrm e^{-r\tau_x} \indic_{\tau_x < \tau_\delta} ] -C' \PP_y\hskip 0.3mm[\tau_x > \tau_\delta].
\end{align*}
The above lower bound is a continuous function of $y$ that is equal to $C$ at $x$ and to $-C'$ at $x-\delta$ and $x+\delta$. Therefore, by \eqref{eq:GRepsilon}, there exists $\delta'\in (0,\delta)$ such that $\bar J^i(y) > R^i(y)$ for all $y\in [x-\delta',x+\delta']$. This proves (c).

Finally, points (i)--(iv) can be proven exactly as in \cite[Proposition 1]{DGM2}. Hence the result.

\subsection{Proof of Equation \eqref{eq:Z=hatZ}}

For the sake of completeness, we show how to deduce \eqref{eq:Z=hatZ} from the arguments in \cite{EKLM}. Recall that $Z^x$ is the Snell envelope on the stochastic basis $(\Omega,\CF,(\CF_t)_{t\geq 0},\PP_x)$ of the process $\bar Y$ defined by
\begin{align*}
\bar Y_t:=\int_{[0,t]}\mathrm e^{-rs}G^i(X_s)\,\mathrm d\Gamma^j_s+ \Lambda^j_{t}\,\mathrm e^{-rt} R^i(X_t), \quad t \geq 0,
\end{align*}
and that $\hat Z$ is defined by
\begin{align*}
\widehat Z_t:=\int_{[0,t]}\mathrm e^{-rs}G^i(X_s)\,\mathrm d\Gamma^j_s+ \Lambda^j_{t}\,\mathrm e^{-rt}\bar J^i(X_t), \quad t \geq 0.
\end{align*}
First, it is clear that $\widehat Z\geq \bar Y$. Then, recall that (see \cite[Lemma 3.4 and the references therein]{EKLM}, noticing that we work on the smaller canonical space of continuous trajectories), for every stopping time $\tau$ of $(\CF^0_t)_{t\geq 0}$ and every stopping time $\rho$ of $(\CF^0_{t+})_{t\geq 0}$ such that $\rho \geq \tau$, there exists an $\CF^0_{\tau}\otimes \CF^0_\infty$ measurable random variable $U: \Omega \times \Omega \rightarrow [0,\infty]$ such that
\begin{itemize}

\item

$U(\omega,\tilde \omega)=0$ if $\tau(\omega)=\infty$ or if $X_0(\tilde \omega)\neq X_{\tau}(\omega)$;

\item

for each $\omega \in \Omega$, $U(\omega,\cdot)$ is a stopping time of $(\CF^0_{t+})_{t \geq 0}$;

\item

for all $\omega \in \Omega$ such that $\tau(\omega)<\infty$, $\rho(\omega)=\tau(\omega)+U(\omega,\theta_{\tau(\omega)}(\omega))$.

\end{itemize}
We deduce that, on the event $\{\tau<\infty\}$,
\begin{align*}
\EE_x \hskip 0.3mm [ \bar Y_{\rho} \! \mid \! \CF_{\tau}]& =\EE_x \! \left[\int_{[0,\rho]}\mathrm e^{-rs}G^i(X_s)\,\mathrm d\Gamma^j_s+ \Lambda^j_{\rho}\,\mathrm e^{-r\rho} R^i(X_\rho) \, \Big| \, \CF_{\tau}\right] \allowdisplaybreaks
\\
&=\EE_x \! \left[\int_{[0,\tau]}\mathrm e^{-rs}G^i(X_s)\, \mathrm d\Gamma^j_s+\int_{(\tau,\rho]}\mathrm e^{-rs}G^i(X_s)\,\mathrm d\Gamma^j_s+  \Lambda^j_{\rho}\,\mathrm e^{-r\rho} R^i(X_\rho) \, \Big| \, \CF_{\tau}\right]
\\
&=\int_{[0,\tau]}\mathrm e^{-rs}G^i(X_s)\,\mathrm d\Gamma^j_s+ \Lambda^j_{\tau}\,\mathrm e^{-r\tau}\EE_x\bigg[ \int_{(0,U(\omega,\theta_\tau (\omega))]} \mathrm e^{-rs}G^i(X_s)\,\mathrm d(\Gamma^j \circ \theta_\tau)_s
\\
&\qquad + ( \Lambda^j_{U(\omega,\theta_\tau(\omega))} \circ \theta_\tau) \,\mathrm e^{-rU(\omega,\theta_\tau(\omega))} R^i(X_{\tau+U(\omega,\theta_\tau(\omega))}) \, \Big| \, \CF_{\tau}\bigg]
\\
&=\int_{[0,\tau]}\mathrm e^{-rs}G^i(X_s)\,\mathrm d\Gamma^j_s+ \Lambda^j_{\tau}\,\mathrm e^{-r\tau}\EE_x\bigg[ \int_{[0,U(\omega,\theta_\tau (\omega))]} \mathrm e^{-rs}G^i(X_s)\,\mathrm d(\Gamma^j \circ \theta_\tau)_s
\\
&\qquad + ( \Lambda^j_{U(\omega,\theta_\tau(\omega))} \circ \theta_\tau) \,\mathrm e^{-rU(\omega,\theta_\tau(\omega))} R^i(X_{\tau+U(\omega,\theta_\tau(\omega))}) \, \Big| \, \CF_{\tau}\bigg]
\\
&\leq \int_{[0,\tau]}\mathrm e^{-rs}G^i(X_s)\,\mathrm d\Gamma^j_s+ \Lambda^j_{\tau}\,\mathrm e^{-r\tau}\bar J^i(X_\tau),
\\
&= \widehat Z_\tau,
\end{align*}
where the third equality follows from \eqref{mark} and the decomposition of stopping times, the fourth equality follows from the fact that $\Lambda^j_\tau=0$ whenever $\Gamma^j \circ \theta_\tau$ has a jump at time $0$ by \eqref{mark}, which allows us to replace the integral over $(0,U(\omega,\theta_\tau(\omega))]$ by an integral over $[0,U(\omega,\theta_\tau(\omega))]$, and the inequality follows from the Markov property. We deduce that $\EE_x \hskip 0.3mm [ \bar Y_{\rho}\! \mid\! \CF_{\tau}]\leq  \widehat Z_\tau$ as it is an equality on $\{\tau=\infty\}$. Because for each $x\in \CI$, every stopping time in $\CT$ is $\PP_x$-a.s.\! equal to a stopping time of $(\CF^0_{t+})_{t\geq 0}$ \cite[Lemma I.1.19]{JacodShiryaev}), we deduce that, for every stopping time $\tau$ of $(\CF^0_t)_{t \geq 0}$,
\begin{align*}
Z^x_\tau= \mathop{\mathrm{ess\,sup}} \limits _{\rho \geq \tau , \,\rho \in \CT}\,\EE_x \hskip 0.3mm[ \bar Y_{\rho} \! \mid \! \CF_{\tau}] \leq \widehat Z_\tau.
\end{align*}
To prove the reverse inequality, it is sufficient to prove that $\EE_x\hskip 0.3mm[\widehat Z_\tau] \leq \EE_x \hskip 0.3mm[Z^x_\tau]$. By \cite[Proposition 2.4]{EKLM},
\begin{align*}
\forall \nu \in \Delta(\CI),\; \int_\CI \bar J^i(y) \,\nu(\mathrm d y)=\sup_{\rho \in \CT^0}    \EE_{\nu}\hskip 0.3mm[\bar Y_{\rho}],
\end{align*}
where $\CT^0$ denotes the set of stopping times of the canonical filtration $(\CF^0_t)_{t \geq 0}$. Let $\hat{\nu}$ denote the finite measure on $\CI$ defined by
\begin{align*}
\hat \nu(A):=\EE_x[\Lambda^j_{\tau}\,\mathrm e^{-r\tau}\indic_{A}(X_\tau)], \quad A \in \CB(\CI).
\end{align*}
Whenever $\nu \not =0$, define the probability $\nu:=\frac{\hat{\nu}}{\hat \nu(\CI)}$. Then, denoting by $\tilde \Omega$ a copy of the canonical space endowed with the probabilities $\tilde \PP_y : =\PP_y$ for $y\in \CI$, we have
\begin{align}
\EE_x \hskip 0.3mm [\Lambda^j_{\tau}\,\mathrm e^{-r\tau}\bar J^i(X_\tau)]&= \int_\CI \bar J^i(y) \, \hat \nu (\mathrm d  y) = \hat \nu(\CI) \sup_{\rho \in \CT^0}\EE_{\nu} \hskip 0.3mm [\bar Y_{\rho}] =\sup_{\rho \in \CT^0} \EE_x \hskip 0.3mm[\Lambda^j_{\tau}\,\mathrm  e^{-r\tau} \, \tilde \EE_{X_\tau}[\bar Y_{\rho}] ]. \label{truc}
\end{align}
We deduce that
\begin{align*}
\EE_x[\widehat Z_\tau] &\leq \sup_{\rho \in \CT^0} \EE_x\!\left[\int_{[0,\tau]}\mathrm e^{-rs}G^i(X_s)\,\mathrm d\Gamma^j_s+\Lambda^j_{\tau}\,\mathrm e^{-r\tau}\,\tilde\EE_{X_\tau}[\bar Y_{\rho}] \right]
\\
& = \sup_{\rho \in \CT^0} \EE_x \bigg[\int_{[0,\tau]}\mathrm e^{-rs}G^i(X_s)\,\mathrm d\Gamma^j_s+\Lambda^j_{\tau}\,\mathrm e^{-r\tau}\,\EE_x\bigg[ \int_{[0,\rho\circ \theta_\tau ]}\mathrm e^{-rs}G^i(X_s)\,\mathrm d(\Gamma^j \circ \theta_\tau)_s
\\
&\qquad\qquad +  (\Lambda^j_{\rho\circ \theta_\tau} \circ \theta_\tau) \,\mathrm e^{-r(\rho\circ \theta_\tau)} R^i(X_{\tau+\rho\circ \theta_\tau}) \,\Big|\, \CF_{\tau}\bigg]\bigg]
\\
& = \sup_{\rho \in \CT^0} \EE_x \bigg[\int_{[0,\tau]}\mathrm e^{-rs}G^i(X_s)\,\mathrm d\Gamma^j_s+\Lambda^j_{\tau}\,\mathrm e^{-r\tau}\,\EE_x\bigg[ \int_{(0,\rho\circ \theta_\tau ]}\mathrm e^{-rs}G^i(X_s)\,\mathrm d(\Gamma^j \circ \theta_\tau)_s
\\
&\qquad\qquad +  (\Lambda^j_{\rho\circ \theta_\tau} \circ \theta_\tau) \,\mathrm e^{-r(\rho\circ \theta_\tau)} R^i(X_{\tau+\rho\circ \theta_\tau}) \,\Big|\, \CF_{\tau}\bigg]\bigg]
\\
&=  \sup_{\rho \in \CT^0} \EE_x \!\left[\int_{[0,\tau+\rho\circ \theta_\tau]} \mathrm e^{-rs}G^i(X_s)\,\mathrm d\Gamma^j_s+ \Lambda^j_{\tau+\rho\circ \theta_\tau}\,\mathrm e^{-r(\tau +\rho\circ\theta_\tau)}R^i(X_{\tau+\rho\circ \theta_\tau})\right]
\\
&=  \sup_{\rho \in \CT^0} \EE_x [ \bar Y_{\tau+\rho\circ \theta_\tau}]
\\
&\leq \EE_x[Z^x_\tau],
\end{align*}
where the first inequality follows from \eqref{truc}, the first equality follows from the strong Markov property, the second equality follows from the fact that $\Lambda^j_\tau=0$ whenever $\Gamma^j \circ \theta_\tau$ has a jump at time $0$, which allows us to replace the integral over $[0,\rho\circ \theta_\tau ]$ by an integral over $(0,\rho\circ \theta_\tau]$, and the third equality follows from \eqref{mark}. This concludes the proof of \eqref{eq:Z=hatZ}.

\subsection{Proof of Proposition \ref{prop:compact_metric}}

It is hereafter assumed without explicit mention that $\CM(\CI)$ is endowed with the topology $\vartheta$. The proof consists of three parts.

\paragraph{Metrizability}

We first prove that $\CM(\CI)$ is metrizable. By Urysohn's metrization theorem (see, e.g., \cite[Theorem 4.58]{Folland}), it is sufficient to check that $\CM(\CI)$ is Hausdorff, regular, and second countable.

First, we check that $\CM(\CI)$ is second countable. By definition, the topology $\vartheta$ has a countable subbasis of neighborhoods defined by all the sets $U,V$ of the form
\begin{align*}
U=U_{a,b,c} := \{ m \in \CM(\CI) : m((a,b))> c \} \, \text{ and } \, V=V_{a,b,d}:=\{ m \in \CM(\CI) : m([a,b])< d \}
\end{align*}
for all $a,b \in \CI\cap \QQ$, $c \in [0,\infty)\cap \QQ$, and $d\in \left( (0,+\infty)\cap \QQ \right)  \cup\{\infty\}$. Therefore, $\CM(\CI)$ is second countable.

Next, we check that $\CM(\CI)$ is regular. To this end, let $B$ be a nonempty closed set in $\CM(\CI)$ and $m\in \CM(\CI) \setminus B$. We have to prove that $B$ and $m$ have disjoint neighborhoods. The complement $B^c$ of $B$ is open and thus
\begin{align*}
B^c =\bigcup_{\alpha } \left(  \bigcap_{k=1}^{n_\alpha} O^{\alpha}_k \right) \, \text{ and } \, B =\bigcap_{\alpha } \left(\bigcup_{k=1}^{n_\alpha} (O^{\alpha}_k)^c \right)\hskip -1mm,
\end{align*}
where $\alpha$ ranges through an arbitrary countable set, and each $O^{\alpha}_k$ is of the form $U,V$ above. In particular, there exists $\alpha$ such that $m \in \bigcap_{k=1}^{n_\alpha} O^{\alpha}_k$ and $B\subset B^\alpha :=\bigcup_{k=1}^{n_\alpha} (O^{\alpha}_k)^c $. Therefore, it is sufficient to prove the claim for $B^\alpha$ instead of $B$. Thus assume that $B=\bigcup_{k=1}^{n} (O_k)^c$. In turn, it is sufficient to prove the claim for each $(O_k)^c$ and then take the union of the neighborhoods of each set $(O_k)^c$, and the (finite) intersection of the neighborhoods of $m$. Thus assume that $B=O^c$ with $O$ of the form $U,V$ above. We accordingly distinguish two cases, depending on the form of $O$.

\subparagraph{Case $B=U^c$ with $U=U_{a,b,c}$}

Let $\delta >0$ such that $m((a,b))>c+2\delta$. There exists $(a',b')\subset (a,b)$ such that $m((a',b'))>c+2\delta$ by inner regularity, so that  $U_{a',b',c+2\delta}$ is an open neighborhood of $m$. On the other hand $V_{a',b',c+\delta}$ is an open neighborhood of $B$ as
\begin{align*}
\forall \nu \in B, \; \nu([a',b'])\leq \nu((a,b))\leq c <c+\delta .
\end{align*}
To conclude, notice that $V_{a',b',c+\delta}$ and $U_{a',b',c+2\delta}$ are disjoint.

\subparagraph{Case $B=V^c$ with $V=V_{a,b,d}$}

Notice that $m\notin B$ is equivalent to $m([a,b])< d$, so that $m([a,b])<\infty$. There exists $(a',b') \supset [a,b]$ such that $m([a',b']) <d$ by outer regularity. Thus let $d',d''$ such that  $m([a',b']) <d'<d''<d$, and observe that $B\subset U_{a',b',d''}$ whereas $m\in V_{a',b',d'}$. To conclude, notice that $V_{a',b',d'}$ and $U_{a',b',d''}$ are disjoint.

\vskip 3mm

Therefore, $\CM(\CI)$ is regular.

Finally, we check that $\CM(\CI)$ is Hausdorff. As $\CM(\CI)$ is regular, it is sufficient to prove that singletons are closed. Let $m_0\in \CM(\CI)$ and consider the closed set
\begin{align*}
&C(m_0)
\\
&:= \bigcap_{a,b\in \CI\cap \QQ} \big( \{ m\in \CM(\CI) : m((a,b))\leq m_0((a,b))\} \cap \{m\in \CM(\CI) : m([a,b]) \geq m_0([a,b]) \}\big).
\end{align*}
If $m\neq m_0$, then there exists an interval $(a,b) \subset \CI$ such that $m((a,b))\neq m_0((a,b))$. By inner regularity, we can assume that $a,b\in \QQ$. If $m((a,b))>m_0((a,b))$, then $m\notin C(m_0)$. If $m((a,b))<m_0((a,b))$, then, by inner regularity, there exists an interval $[a',b']\subset (a,b)$ such that $a',b'\in \QQ$ and $m([a',b'])<m_0([a',b'])$, so that $m\notin C(m_0)$. We conclude that $C(m_0)=\{m_0\}$ and hence that singletons are closed. Therefore, $\CM(\CI)$ is Hausdorff.

The proof that $\CM(\CI)$ is metrizable is now complete.

\paragraph{Proofs of Assertions 1--3}

We prove each assertion in turn.

(1) Any open set $O \subset \CI$ can be written as $O=\bigcup_{n \geq 0} O_n$ for some nondecreasing sequence $(O_n)_{n \geq 0}$ such that each $O_n$ is a finite disjoint union of open intervals with rational endpoints. The mapping $m\mapsto m(O_n)$ is lsc as a finite sum of lsc mappings, and the mapping $m\mapsto m(O)$ is lsc as the supremum of lsc mappings.

(2) Any compact set $K\subset \CI$ can be written as $K=\bigcap_{n\geq 0} K_n$ for some nonincreasing sequence $(K_n)_{n \geq 0}$ such that each $K_n$ is a finite disjoint union of compact intervals with rational endpoints. The mapping $m\mapsto m(K_n)$ is usc as a finite sum of usc mappings, and the mapping $m\mapsto m(K)$ is usc as the infimum of usc mappings.

(3) Suppose that $m_n \to m$ in $\CM(\CI)$. If $O \cap e(m)\neq \emptyset$ with $O$ open, then $m(O)=\infty$, and thus $m_n(O)\to \infty$ by point (1). Now, let $\phi \in \CC_c^+(\CI\setminus e(m))$, with support $K$. Because $m(K)<\infty$, by outer regularity, there exists a compact neighborhood $K'$ of $K$ such that $m(K')<\infty$. Then, by point (2), $\limsup_{n \to \infty} m_n(K') \leq m(K')<\infty$. The restrictions of the measures $(m_n)_{n \geq 0}$ to the open set $O':=\mathrm {int} \hskip 0.5mm K'$ are therefore locally finite for any sufficiently large $n$, and by \cite[Lemma 4.1(iv)]{Kallenberg}, converge vaguely to the restriction of $m$ to $O'$, which implies $L_\phi(m_n)\to L_\phi(m)$.

Conversely, suppose that the sequence $(m_n)_{n\geq 0}$ in $\CM(\CI)$ and the measure $m$ satisfy the properties that, for every open set $O$ such that $O \cap e(m) \neq \emptyset$, $m_n(O) \to \infty$, and that, for each $\phi\in \CC_c^+(\CI \setminus e(m))$,  $L_\phi( m_n) \rightarrow L_\phi(m)$. We want to prove that $m_n \to m$ in $\CM(\CI)$. Let $a,b \in \CI\cap \QQ$. If $(a,b) \cap e(m) \neq \emptyset$, then, by the first property, $\liminf_{n \to \infty} m_n((a,b))=\infty = m((a,b))$. If $(a,b)\cap e(m) = \emptyset$, let $(\phi_k)_{k \geq 0}$ be a nondecreasing sequence of continuous functions with compact support in $(a,b)$ with pointwise limit $\indic_{(a,b)}$. Then, by the second property, we have, for each $k$,
\begin{align*}
\liminf_{n \to \infty} m_n((a,b)) \geq \lim_{n \to \infty} L_{\phi_k}(m_n)=L_{\phi_k}(m),
\end{align*}
and thus, by monotone convergence,
\begin{align}
\liminf_{n \to \infty} m_n((a,b)) \geq m((a,b)). \label{cafe}
\end{align}
If $[a,b] \cap e(m) \neq \emptyset$, then $\limsup m_{n \to \infty}([a,b])\leq \infty = m([a,b])$. If $[a,b]\cap e(m) = \emptyset$, then $m([a,b])<\infty$, and there exists $a',b'$ such that $[a,b]\subset (a',b')$ and $m((a',b'))<\infty$ by outer regularity; in particular, $(a',b')\cap e(m)=\emptyset$. Let $(\phi_k)_{k \geq 0}$ be a nonincreasing sequence of continuous functions with compact support in $(a',b')$ and pointwise limit $\indic_{[a,b]}$. Then, by the second property, we have, for each $k$,
\begin{align*}
\limsup_{n \to \infty} m_n([a,b]) \leq \lim_{n \to \infty} L_{\phi_k}(m_n)=L_{\phi_k}(m),
\end{align*}
and thus, by bounded convergence,
\begin{align}
\limsup_{n\to \infty} m_n([a,b]) \leq m([a,b]). \label{clope}
\end{align}
We conclude from (\ref{cafe})--(\ref{clope}) that $m_n \to m$ in $\CM(\CI)$.

\paragraph{Compactness}

We finally prove that $\CM(\CI)$ is compact. As $\CM(\CI)$ is metrizable, it is sufficient to prove that it is sequentially compact, i.e., that any sequence $(m_n)_{n \geq 0}$ has a convergent subsequence. The proof consists of three steps.

\subparagraph{Step 1}

Let $\CB=\{O_1,O_2,...\}$ denote a countable basis of open sets for $\CI$. If $\limsup_{n\to \infty} m_n\linebreak (O_1)=\infty$, then we extract a subsequence $(m^1_n)_{n\geq 0}$ such that $\lim_{n \to \infty} m^1_n(O_1)=\infty$, otherwise we let $(m^1_n)_{n\geq 0}:=(m_n)_{n\geq 0}$. Assuming that the subsequence $(m^k_n)_{n \geq 0}$ for some $k \geq 1$ is constructed, if $\limsup_{n\to \infty} m^k_n(O_{k+1})=\infty$, then we extract a subsequence $(m^{k+1}_n)_{n\geq 0}$ of $(m^k_n)_{n\geq 0}$ such that $\lim_{n \to \infty} m^{k+1}_n(O_{k+1})=\infty$, otherwise we let $(m^{k+1}_n)_{n \geq 0}:=(m^k_n)_{n\geq 0}$. By diagonal extraction, we obtain a subsequence $(m^*_n)_{n \geq 0}$ of $(m_n)_{n \geq 0}$ such that, for each $k \geq 1$, either $\lim_{n \to \infty} m^{*}_n(O_{k})=\infty$ or $\limsup_{n \to \infty} m^*_n(O_k)<\infty$. Now, for all $x\in \CI$, let $D_x : =\{ k\geq 1: x\in O_k\}$, and notice that $\{x\}=\bigcap_{k \in D_x} O_k$. Define then
\begin{align*}
S : =\big\{ x \in \CI : \forall k \in D_x, \,  \lim_{n \to \infty} m^{*}_n(O_{k})=\infty \big\}.
\end{align*}
We claim that $S$ is closed. Indeed, let $(x_p)_{p \geq 0}$ be a sequence in $S$ with limit $x\in \CI$. For each $k \in D_x$, we have $x_p\in O_k$ for any sufficiently large $p$ and thus $k\in D_{x_p}$. Therefore, $\lim_{n \to \infty} m^*_n(O_k)=\infty$ for all $k \in D_x$, which proves that $x\in S$. The claim follows.

\subparagraph{Step 2}

Let $(K_p)_{p\geq0}$ be an increasing sequence of compact sets such that $\bigcup_{p=0}^\infty K_p= \CI\setminus S$. We claim that $ \limsup_{n \to \infty} m^*_n(K_p) <\infty$ for all $p\geq 0$. Indeed, each $x \in K_p$ is such that there exists $k \in D_x$ such that $\limsup_{n \to \infty} m^{*}_n (O_k) <\infty$. These open sets form an open covering of $K_p$, and we may therefore extract a finite open cover $(O_{k_1},\ldots,O_{k_r})$. We conclude that
\begin{align*}
\limsup_{n \to \infty} m^*_n(K_p) \leq \limsup_{n \to \infty} \sum_{t=1}^r m^*_n(O_{k_t}) \leq \sum_{t=1}^r \limsup_{n \to \infty} m^*_n(O_{k_t}) < \infty.
\end{align*}
The claim follows. Because $\limsup_{n\to \infty} m_n(K_1)<\infty$, the restriction of $m^*_n$ to $K_1$ is a finite measure for any sufficiently large $n$. By \cite[Theorem 4.2]{Kallenberg}, it admits a subsequence that converges weakly to some finite measure $\mu_1$ on $K_1$.\footnote{A sequence a finite measures $(\nu_n)_{n \geq 0}$ on some metric space $E$ converge weakly to $\nu$ if $\int_E \phi \, \mathrm d\nu_n \to \int_E \phi \, \mathrm d\nu$ for every bounded and continuous function $\phi: E \to \RR$.} Iterating the process and using diagonal extraction, we can extract a subsequence $(m^{**}_n)_{n \geq 0}$ of $(m^*_n)_{n \geq 0}$ such that, for each $p\geq 0$, the sequence of the restrictions of the measures $(m^{**}_n)_{n \geq 0}$ to $K_p$ converges weakly to some finite measure $\mu_p$ on $K_p$. By construction, the measures $(\mu_p)_{p \geq 0}$ are consistent in the sense that there exists a Radon measure $\mu$ on $\CI\setminus S$ whose restriction to $K_p$ is $\mu_p$ for all $p \geq 0$, and therefore the sequence of the restrictions of the measures $(m^{**}_n)_{n \geq 0}$ to $\CI\setminus S$ converges vaguely to $\mu$.

\subparagraph{Step 3}

Define $m \in \CM(\CI)$ such that $e(m):=S$ and $m|_{\mathcal I \setminus e(m)}:=\mu$. Then, by Assertion 3, the subsequence $(m^{**}_n)_{n \geq 0}$ constructed in Step 2 converges in $\CM(\CI)$ to $m$. Hence the result.

\end{document}